\documentclass[reqno]{amsart}
\usepackage{amsmath,amsthm}
\usepackage{dsfont} 
\usepackage[english]{babel}
\usepackage{color}
\usepackage{amssymb}
\usepackage{mathrsfs}
\usepackage{graphicx}
\usepackage[font=footnotesize, labelfont=bf]{caption}
\usepackage{mathtools}
\usepackage[colorlinks=true, pdfstartview=FitV, linkcolor=blue, citecolor=blue, urlcolor=blue,pagebackref=false]{hyperref}
\usepackage{microtype}
\usepackage{enumitem}
\usepackage[normalem]{ulem}
\usepackage{soul}

\usepackage{float}

\newtheorem{theorem}{Theorem}[section]
\newtheorem{proposition}[theorem]{Proposition}
\newtheorem{lemma}[theorem]{Lemma}

\theoremstyle{definition}
\newtheorem{remark}[theorem]{Remark}
\newtheorem{example}[theorem]{Example}
\newtheorem{definition}[theorem]{Definition}

\newtheorem{assumption}{Assumption}
\numberwithin{equation}{section}

\definecolor{myblue}{RGB}{0,0,128}

\def\Xint#1{\mathchoice
	{\XXint\displaystyle\textstyle{#1}}%
	{\XXint\textstyle\scriptstyle{#1}}%
	{\XXint\scriptstyle\scriptscriptstyle{#1}}%
	{\XXint\scriptscriptstyle\scriptscriptstyle{#1}}%
	\!\int}
\def\XXint#1#2#3{{\setbox0=\hbox{$#1{#2#3}{\int}$ }
		\vcenter{\hbox{$#2#3$ }}\kern-.57\wd0}}

\def\dashint{\Xint-}

\newcommand{\dx}{\,\mathrm{d}x}
\newcommand{\dy}{\,\mathrm{d}y}

\newcommand{\e}{\varepsilon}
\newcommand{\dist}{{\rm{dist}}}
\renewcommand{\L}{\mathcal{L}}
\newcommand{\w}{\omega}

\newcommand{\R}{\mathbb{R}}
\newcommand{\Z}{\mathbb{Z}}
\newcommand{\N}{\mathbb{N}}
\newcommand{\Q}{\mathbb{Q}}

\newcommand{\F}{\mathcal{F}}

\renewcommand{\P}{\mathbb{P}}

\begin{document}

\author{Matthias Ruf}
\address[Matthias Ruf]{Section de mathématiques, Ecole Polytechnique Fédérale de Lausanne, Station 8, 1015 Lausanne, Switzerland}
\email{matthias.ruf@epfl.ch}

\author{Thomas Ruf}
\address[Thomas Ruf]{Institut f\"ur Mathematik, Universit\"at Augsburg, 86159 Augsburg, Germany}
\email{thomas.ruf@math.uni-augsburg.de}

\title[Stochastic homogenization of degenerate integral functionals]{Stochastic homogenization of degenerate integral functionals and their Euler-Lagrange equations}

\begin{abstract}	
We prove stochastic homogenization for integral functionals defined on Sobolev spaces, where the stationary, ergodic integrand satisfies a degenerate growth condition of the form
\begin{equation*}
	c|\xi A(\w,x)|^p\leq f(\w,x,\xi)\leq |\xi A(\w,x)|^p+\Lambda(\w,x)
\end{equation*}
for some $p\in (1,+\infty$) and with a stationary and ergodic diagonal matrix $A$ such that its norm and the norm of its inverse satisfy minimal integrability assumptions. We also consider the convergence when Dirichlet boundary conditions or an obstacle condition are imposed. Assuming the strict convexity and differentiability of $f$ with respect to its last variable, we further prove that the homogenized integrand is also strictly convex and differentiable. These properties allow us to show homogenization of the associated Euler-Lagrange equations.	
\end{abstract}

\maketitle
{\small
	\noindent\keywords{\textbf{Keywords:} Stochastic homogenization, integral functionals, degenerate $p$-growth, homogenization of Euler-Lagrange equations}
	
	\noindent\subjclass{\textbf{MSC 2020:} 49J45, 49J55, 60G10, 35J70}
}	

%\tableofcontents

\section{Introduction}
In their pioneering papers \cite{DMMoI,DMMoII}, Dal Maso and Modica have set the basic strategy for the stochastic homogenization of integral functionals defined on Sobolev spaces. They consider functionals of the type
\begin{equation}\label{eq:F_eps}
	F_{\e}(\w,u)=\int_D f(\w,\tfrac{x}{\e},\nabla u(x))\dx,
\end{equation}
where $f$ is measurable in $(\w,x)$ and convex in the last variable satisfying the $p$-growth condition
\begin{equation}\label{eq:p-growth}
	c|\xi|^p\leq f(\w,x,\xi)\leq C(|\xi|^p+1)
\end{equation}
for some $p\in (1,+\infty)$. In this setting, the appropriate tool to study the asymptotic behavior of minimizers as $\e\to 0$ is $\Gamma$-convergence (cf. \cite{Br,DM} for a general introduction to the topic). Randomness enters the problem through the parameter $\w$, that is, the integrands are chosen at random in the above class of integrands. A crucial assumption that allows to homogenize the functionals to a simpler functional is stationarity, that is, for all $z\in\R^d$ (or all $z\in\Z^d$, depending on the model) and for any finite point set $\{x_1,\ldots,x_n\}\subset\R^d$, the random vectors $(f(\w,x_1,\xi),\ldots,f(\w,x_n,\xi))$ and $(f(\w,x_1+z,\xi),\ldots,f(\w,x_n+z,\xi))$ have the same distribution (this can be expressed conveniently with measure preserving group actions, cf. Section \ref{s.preliminaries}). Under those assumptions one can prove that the $\Gamma$-limit exists almost surely and is given by an integral functional of the form
\begin{equation*}
	F_{\rm hom}(\w,u)=\int_D f_{\rm hom}(\w,\nabla u(x))\dx.
\end{equation*}
Under the additional assumption of ergodicity, the limit is deterministic. The integrand still satisfies $p$-growth conditions and is convex in the gradient variable. The basic strategy can be summarized as follows: by deterministic arguments, one proves that the given class of integral functionals is compact with respect to $\Gamma$-convergence, which implies that (up to a subsequence) the $\Gamma$-limit of $F_{\e}(\w,\cdot)$ has the form
\begin{equation*}
	F_{0}(\w,u)=\int_D f_0(\w,x,\nabla u(x))\dx.
\end{equation*}
In a second step, a blow-up formula, Jensen's inequality, and the convergence of minima show that
\begin{align*}
f_{0}(\w,x_0,\xi)&=\lim_{\delta\to 0}\inf\left\{\dashint_{Q_{\delta}(x_0)}f_0(\w,x,\xi+\nabla v(x))\dx:\,v\in W^{1,p}_0(Q_{\delta}(x_0))\right\}
\\
&=\lim_{\delta\to 0}\lim_{\e\to 0}\inf\left\{\dashint_{Q_{\delta}(x_0)}f(\w,\tfrac{x}{\e},\xi+\nabla v(x))\dx:\,v\in W^{1,p}_0(Q_{\delta}(x_0))\right\}
\\
&\!\!\!\!\overset{t=1/\e}{=}\lim_{\delta\to 0}\lim_{t\to +\infty}\inf\left\{\dashint_{Q_{\delta t}(x_0)}f(\w,y,\xi+\nabla v(x))\dx:\,v\in W^{1,p}_0(Q_{\delta t}(x_0))\right\}.
\end{align*}
It is then a consequence of the subadditive ergodic theorem \cite{AkKr} that the limit in $t$ exists and is independent of $x_0$ and $\delta$. This proves the full $\Gamma$-convergence result and ergodicity implies that the integrand is deterministic.
\\
This approach has been extended to nonconvex integrands in \cite{MeMi} with the same growth condition \eqref{eq:p-growth} and a quantitative continuity assumption in the gradient variable compatible with the $p$-growth of $f$. The additional continuity assumption is no major restriction since the relaxed functional has the same $\Gamma$-limit and by a general theory has an integrand that is quasiconvex in the gradient variable. For quasiconvex functions satisfying the $p$-growth condition \eqref{eq:p-growth} the additional continuity estimate comes for free (see \cite[Proposition 4.64]{FoLe}). To the best of our knowledge, the by now most general stochastic homogenization result for integral functionals on Sobolev spaces is \cite{DG_unbounded}, where the authors consider integrands $f$ that are either convex with the lower bound in \eqref{eq:p-growth} for $p>d$ and such that $\sup_{\w,x}f(\w,x,\cdot)$ has zero in the interior of its domain, but with no other growth condition from above, or nonconvex integrands with a convex lower and upper bound of the above type together with a technical upper semicontinuity condition (see \cite[Definition 2.5]{DG_unbounded}) that covers the case of adding nonconvex perturbations satisfying the $p$-growth condition \eqref{eq:p-growth} to a possibly unbounded convex integrand. In this setting, the approach by Dal Maso and Modica no longer works since no integral representation theorem exists for such functionals and therefore there is little chance to prove that the class of functionals is compact with respect to $\Gamma$-convergence.  

More recently, also the discrete-to-continuum analysis of finite-difference models on an $\e$-scaled lattice either with random weights or a random geometry attracted attention. Under the same $p$-growth condition \eqref{eq:p-growth} and a decay assumption for long-range interactions, general homogenization results for discrete energies defined on an $\e$-scaled stationary stochastic lattice or on a fixed periodic lattice $\e\Z^d$ with stationary, ergodic interactions were obtained in \cite{ACG2}. The latter case was extended to degenerate weights with a finite range of interaction in \cite{NSS}. More precisely, the authors consider energies of the type
\begin{equation*}
E_{\e}(\w,u)=\e^d\sum_{z\in \e\Z^d\cap D}\sum_{e\in \mathcal{E}}f_e(\w,\tfrac{z}{\e},\nabla u(e)),
\end{equation*} 
where $\mathcal{E}$ is a finite set of edges and $\nabla u(e)$ denotes the discrete gradient along the edge $e$. Up to constants, the density $f_e$ satisfies the growth condition
\begin{equation}\label{eq:deg-p-growth}
c\lambda_e(\w,x)|\xi|^p\leq f_e(\w,x,\xi)\leq \lambda_e(\w,x)(|\xi|^p+1),
\end{equation}
but the weights $\lambda_e$ are not required to be bounded uniformly from above and below (in that case \eqref{eq:deg-p-growth} and \eqref{eq:p-growth} would be equivalent). Instead, the following integrability assumptions are taken into account (here and in what follows $\Omega$ denotes a probability space):
\begin{itemize}
	\item $\lambda_e(\cdot,0) \in L^1(\Omega)$ and $\lambda_e^{-1/(p-1)}(\cdot,0)\in L^{1}(\Omega)$ if $u$ is scalar-valued;
	\item $\lambda_e(\cdot,0) \in L^{\alpha}(\Omega)$ for some $\alpha>1$ and $\lambda_e^{-\beta}(\cdot,0)\in L^{1}(\Omega)$ for some $\beta$ such that
	\begin{equation}\label{eq:moment_vectorial}
		\frac{1}{\alpha}+\frac{1}{\beta}\leq\frac{p}{d},
	\end{equation}
	if $u$ is a vectorial function.
\end{itemize}
In the scalar case, the authors need an additional 'convexity at infinity' assumption to be satisfied by $f_{e}$. Note that the moment condition in the vectorial case is strictly stronger than in the scalar case.

In this paper we consider a continuum version of \cite{NSS}, that is, we consider integral functionals of the type \eqref{eq:F_eps} with the integrand $f$ satisfying the degenerate growth condition
\begin{equation}\label{eq:deg-growth}
	c|\xi A(\w,x)|^p\leq f(\w,x,\xi)\leq |\xi A(\w,x)|^p+\Lambda(\w,x),
\end{equation}
where $A:\Omega\times\R^d\to \mathbb{M}_{d}$ is a diagonal matrix-valued function and $\Lambda:\Omega\times\R^d\to (0,+\infty)$. The novelty of our result is that we only assume the moment conditions $|A(\cdot,0)|^p,\Lambda(\cdot,0) \in L^1(\Omega)$ and $|A(\cdot,0)^{-1}|^{p/(p-1)}\in L^{1}(\Omega)$, both in the scalar and the vectorial case, and we drop the convexity assumption at infinity. These moment conditions are optimal in the sense that the multi-cell formula defining the homogenized integrand degenerates for some examples violating the above integrability assumptions (see Remark \ref{r.optimal} and Example \ref{ex:optimal}). Note that the different integrability exponents compared to \cite{NSS} are due to the fact that the matrix $A(\w,x)$ occurs inside the $p$th power, while in \eqref{eq:deg-p-growth} the edge weights $\lambda_e$ correspond to the $p$th power of the eigenvalues. The degeneracy via the matrix $A$ allows us to consider anisotropically degenerated integrands. However, the diagonal structure, that yields a single weight for each partial derivative similar to \eqref{eq:deg-p-growth}, is crucial for our proof. Besides joint measurability of $f$, the only further regularity assumption we make is the lower semicontinuity in the gradient-variable, which we need for measurability issues (cf. Lemma \ref{l.measurable}).

%Finally, we remark our approach requires the same weight on both sides of \eqref{eq:deg-growth}. Different weights were for instance considered in \cite{ArSm} in the context of stochastic homogenization of fully nonlinear elliptic PDEs, where it was shown that the optimal integrability condition for the lower bound is $\lambda(\cdot,0)^{-1}\in L^d(\Omega)$ (the upper bound being deterministic and the growth conditions corresponding to $p=2$).

Assuming the above degenerate growth condition together with the stationarity and ergodicity of $f,A$ and $\Lambda$ (cf. Assumption \ref{a.1}), we show in Theorem \ref{thm.Gamma_pure} that $u\mapsto F_{\e}(\w,u,D)$ $\Gamma$-converges in $L^1(D,\R^m)$ to a deterministic functional $F_{\rm hom}$ that is finite only on $W^{1,p}(D,\R^m)$, taking the form
\begin{equation*}
F_{\rm hom}(u)=\int_D f_{\rm hom}(\nabla u(x))\dx,
\end{equation*}
where the homogenized density is given by a standard multi-cell formula involving minimizing the heterogeneous functional under affine Dirichlet boundary conditions. Furthermore, the integrand $f_{\rm hom}$ satisfies the standard $p$-growth condition \eqref{eq:p-growth}. In Theorem \ref{thm:Dirichlet_and_forces} we consider Lipschitz-continuous boundary conditions and an external linear force, which both pass to the limit. In this case, compactness of minimizing sequences (or, more general, energy-bounded sequences) holds with respect to weak convergence in $W^{1,1}(D,\R^m)$ and strong convergence in $L^{d/(d-1)}(D,\R^m)$. In order to obtain the strong convergence with exponent $d/(d-1)$, in Theorem \ref{thm.embedding} we prove the complete continuity of the non-compact Sobolev embedding $W^{1,1}\hookrightarrow L^{d/(d-1)}$, that means, it maps weakly converging sequences to norm-converging sequences (cf. Theorem \ref{thm.embedding}). Very recently, D'Onofrio and Zeppieri obtained a similar stochastic homogenization result \cite{D'OnZe}, assuming additionally that the weights further belong to the Muckenhoupt class $\mathcal{A}_p$ (which ensures a slightly higher stochastic integrability and gives more structure to the corresponding weighted Sobolev space) as well as a local Lipschitz-continuity of $f$ in the last variable. However, they also prove a representation result for the $\Gamma$-limit in the non-homogenization regime which requires to work in weighted spaces even in the limit $\e\to 0$ and for which our strategy of proof would not be feasible.

In Theorem \ref{thm.obstacle}, we add an obstacle-type constraint of the form $u\geq\varphi_{\e}$ on $D$ (to be interpreted componentwise), where $\varphi_{\e}$ converges weakly$^*$ in $W^{1,\infty}(D,\R^m)$ to some function $\varphi\in W^{1,\infty}(D,\R^m)$. In the limit $\e\to 0$ we obtain an obstacle problem for $F_{\rm hom}$ with the obstacle $\varphi$ and boundary condition $g$. 

As in \cite{NSS}, the subtle point in the proof of the $\Gamma$-convergence is the possibility to locally modify a sequence on a small set with a controlled increase of energy (this is the so-called fundamental estimate in the language of $\Gamma$-convergence of local functionals). In the non-degenerate setting, this can be done if the sequence converges strongly in $L^p$, where $p$ is the growth-exponent of the integrand. However, in the degenerate setting the corresponding term is weighted by $|A(\w,\tfrac{x}{\e})|^p$. Moreover, the strong convergence in $L^p$ might not be an appropriate topology for the $\Gamma$-convergence since we can prove compactness of sublevel sets of the energy only in $W^{1,1}$. We overcome this issue using two ingredients: via a vectorial truncation, we show that up to a small error in energy, we can assume that the sequence is bounded in $L^{\infty}$. In order to pass to the limit in the critical term, we further have to control the oscillating weight function $|A(\w,\tfrac{x}{\e})|^p$. We prove a strengthened version of the ergodic theorem in the sense that this family of oscillating functions converges weakly in $L^1(D)$ for almost every realization. The ergodic theorem yields the convergence when integrating over cubes or, more generally, sets with a decent boundary, which would show weak convergence if $|A|^p$ possessed higher integrability. In our setting however, to show the weak convergence one needs to establish the ergodic theorem for averages of the form
\begin{equation*}
	\dashint_E |A(\w,\tfrac{x}{\e})|^p\dx
\end{equation*} 
for an arbitrary Borel set $E\subset D$. We show the $L^1$-weak convergence by an abstract approach identifying the biting limit of the sequence (cf. Lemma \ref{l.weakL1}). 

In a second part of the paper, we focus on the convergence of the associated Euler-Lagrange equations. To this end, we consider a strengthened set of assumptions, namely we assume that $f$ is strictly convex and differentiable with respect to the last variable. In Theorem \ref{thm.PDEs} we show that when $f_{\e}\rightharpoonup f_0$ in $L^d(D,\R^m)$, the unique weak solutions of the degenerate elliptic PDE
\begin{align*}
	-{\rm div}(\partial_{\xi }f(\w,\tfrac{\cdot}{\e},\nabla u))&=f_{\e}\quad\text{ on }D,
	\\
	u&=g\quad\text{ on }\partial D
\end{align*}
with finite energy converge almost surely to the unique weak solution of the PDE
\begin{align*}
	-{\rm div}(\nabla f_{\rm hom}(\nabla u))&=f_0\quad\text{ on }D,
	\\
	u&=g\quad\text{ on }\partial D.
\end{align*}
However, we emphasize that our result is naturally restricted to variational models, that is, the monotone operator $\partial f_{\xi}$ has to be the gradient of a potential. It would be interesting to study non-variational degenerate PDEs with variational techniques (see \cite{ArMo} for the approach in the uniformly elliptic setting). However, this is beyond the scope of the present paper. To deduce the above convergence statement for the Euler-Lagrange equations from our $\Gamma$-convergence result, two properties are of fundamental importance: the strict convexity and the differentiability of $f_{\rm hom}$. To prove these two properties, we establish a non-asymptotic formula for the homogenized integrand that involves a single minimization problem on the probability space (cf. Lemma \ref{l.F_pot_formula}). This formula is well-known in the non-degenerate case, but some care has to be taken when extending it to our setting. From the single minimization problem it is quite straightforward to show that strict convexity of the homogenized integrand is inherited from the heterogeneous integrands. To prove differentiability we show directly that the convex function  $f_{\rm hom}$ is upper semidifferentiable in the sense of \cite{BKK}. Let us mention that in the non-degenerate, deterministic, but non-periodic case the differentiability of the integrand of the $\Gamma$-limit (which in that setting exists up to subsequences) was proven under the assumption of convexity in $\xi$ and a local equicontinuity of the derivative of $\partial_{\xi}f_{\e}$ in \cite[Propisition 3.5]{GiPo}, while the non-convex case was treated in  \cite[Theorem 2.8]{ADMZ} under a global estimate on the modulus of continuity of $\partial_{\xi}f_{\e}$ (for the sake of completeness, we prove the differentiability of $f_{\rm hom}$ in the non-convex case with a corresponding degenerate modulus of continuity in Appendix \ref{app:0}).

The paper is organized as follows: in Section \ref{s.preliminaries} we introduce the precise framework and recall some notions from probability theory. We then present the main results in Section \ref{s.results}, while the proofs are contained in Section \ref{s.proofs}. In the appendix we show the differentiability of $f_{\rm hom}$ without convexity assumptions, a measurability result, and prove the complete continuity of the Sobolev embedding $W^{1,1}\hookrightarrow L^{d/(d-1)}$, the latter being independent of the rest of the paper.

\section{Preliminaries and notation}\label{s.preliminaries}
\subsection{General notation}
We fix $d\geq 2$. Given a measurable set $S\subset\R^d$, we denote by $|S|$ its $d$-dimensional Lebesgue measure. For $x\in\R^d$ we denote by $|x|$ the Euclidean norm and $B_{\rho}(x)$ denotes the open ball with radius $\rho>0$ centered at $x$.
Given $x_0\in\R^d$ and $\rho>0$ we set $Q_{\rho}(x_0)=x_0+(-\rho/2,\rho/2)^d$. We let $\mathbb{M}_d$ be the set of real-valued $d\times d$-matrices equipped with the operator-norm $|\cdot|$ induced from the Euclidean norm on $\R^d$. We further define $\mathbb{D}_d$ be the set of diagonal matrices in $\mathbb{M}_d$. For a measurable set with positive measure, we define $\dashint_S=\frac{1}{|S|}\int_S$. We use standard notation for $L^p$-spaces and Sobolev spaces $W^{1,p}$. In case of functions $W^{1,\infty}(D,\R^m)$, we always choose the Lipschitz-continuous representative (which exists since $D$ is an extension domain). The Borel $\sigma$-algebra on $\R^d$ will be denoted by $\mathcal{B}^d$, while we use $\mathcal{L}^d$ for the $\sigma$-algebra of Lebesgue-measurable sets.
Throughout the paper, we use the continuum parameter $\e$, but statements like $\e\to 0$ stands for an arbitrary sequence $\e_n\to 0$. Finally, the letter $C$ stand for a generic positive constant that may change every time it appears.

\subsection{Stationarity and ergodicity}
Let $\Omega=(\Omega,\mathcal{F},\mathbb{P})$ be a complete probability space. We recall the notion of measure-preserving group actions.
\begin{definition}[Measure-preserving group action]\label{def:group-action} A measure-preserving additive group action on $(\Omega,\F,\P)$ is a family $\{\tau_z\}_{z\in\R^d}$ of measurable mappings $\tau_z:\Omega\to\Omega$ satisfying the following properties:
\begin{enumerate}[label=(\arabic*)]
	\item\label{joint} (joint measurability) the map $(\w,z)\mapsto\tau_z(\w)$ is $\mathcal{F}\otimes\mathcal{L}^d-\mathcal{F}$-measurable;
	\item\label{inv} (invariance) $\P(\tau_z F)=\P(F)$, for every $F\in\F$ and every $z\in\R^d$;
	\item\label{group} (group property) $\tau_0=\rm id_\Omega$ and $\tau_{z_1+z_2}=\tau_{z_2}\circ\tau_{z_1}$ for every $z_1,z_2\in\R^d$.
\end{enumerate}
If, in addition, $\{\tau_z\}_{z\in\R^d}$ satisfies the implication
\begin{equation*}
	\mathbb{P}(\tau_zF\Delta F)=0\quad\forall\, z\in\R^d\implies \mathbb{P}(F)\in\{0,1\},
\end{equation*}
then it is called ergodic.
\end{definition}
We will use several times the additive ergodic theorem in the following form (see \cite[Proposition 2.1]{ArSm}):
\begin{theorem}[Additive ergodic theorem]\label{thm.additiv_ergodic}
Let $g\in L^1(\Omega)$ and $\{\tau_z\}_{z\in\R^d}$ be a measure-preserving, ergodic group action. Then a.s. for any bounded, open set $O\subset\R^d$ with Lipschitz boundary it holds that
\begin{equation*}
\lim_{t\to +\infty}\dashint_{tO}g(\tau_x\w)\dx=\mathbb{E}[g].
\end{equation*} 
\end{theorem}
\subsection{Framework and assumptions}
Fix an open, bounded set $D\subset \R^d$ $(d\geq 2$) with Lipschitz boundary and let $(\Omega,\mathcal{F},\mathbb{P})$ be a complete probability space. For $\e>0$, we consider integral functionals defined on $L^1(D,\R^m)$ with domain contained in $W^{1,1}(D,\R^m)$, taking the form
\begin{equation*}
	F_{\e}(\w,u,D)=\int_D f(\w,\tfrac{x}{\e},\nabla u(x))\,\mathrm{d}x\in [0,+\infty]
\end{equation*}
with the integrand $f$ satisfying the following assumptions:
\begin{assumption}\label{a.1} The function $f:\Omega\times\R^d\times\R^{m\times d}\to [0,+\infty)$ is $\mathcal{F}\otimes\mathcal{L}^d\otimes\mathcal{B}^{m\times d}$-measurable and
\begin{itemize}
	\item[(A1)] for all $\w\in\Omega$ and all $x\in\R^d$ the map $\xi\mapsto f(\w,x,\xi)$ is lower semicontinuous;
	
	\vspace*{2.5mm}
	
	\item[(A2)] let $p\in (1,+\infty)$. For all $\w\in\Omega,x\in\R^d$ and $\xi\in \R^{m\times d}$ it holds that
	\begin{equation*}
		c\,|\xi A(\w,x)|^p\leq f(\w,x,\xi)\leq |\xi A(\w,x)|^p+\Lambda(\w,x),
	\end{equation*}
	with $c>0$ and jointly measurable functions $A:\Omega\times\R^d\to \mathbb{D}_{d}$ and $\Lambda:\Omega\times\R^d\to [0,+\infty)$ such that $|A(\cdot,0)|^p,\Lambda(\cdot,0)\in L^1(\Omega)$ and $|A(\cdot,0)^{-1}|^{p/(p-1)}\in L^{1}(\Omega)$;
	
	\vspace*{2.5mm}
	
	\item[(A3)] there exists a measure-preserving, ergodic group action $\{\tau_{z}\}_{z\in\R^d}$ such that 
	\begin{equation*}
		\begin{split}
			f(\tau_z\w,x,\xi)&=f(\w,x+z,\xi),
			\\
			A(\tau_z\w,x)&=A(\w,x+z) 
			\\
			\Lambda(\tau_z\w,x)&=\Lambda(\w,x+z)
		\end{split}
	\end{equation*}
for all $(z,\w,x,\xi)\in \R^d\times\Omega\times\R^d\times\R^{m\times d}$.
\end{itemize}
\end{assumption}
\begin{remark}\label{r.on_assump}
\begin{itemize}
	\item[a)] It would be more natural to require the properties in (A1) and (A2) for almost every $\w\in\Omega$ and almost every $x\in\R^d$. Assumption (A3), however, can actually serve as a definition given the random functions $f(\w,0,\xi)$, $A(\w,0)$ and $\Lambda(\w,0)$. If they satisfy (A1) and (A2) almost surely for $x=0$, then by \cite[Lemma 7.1]{JKO} properties (A1) and (A2) hold for almost every $\w\in\Omega$ and for every $x\in \R^d\setminus N_{\w}$ with a null set $N_{\w}\subset\R^d$, while (A3) holds pointwise. Our results remain valid in this setting with obvious modifications in the proofs.
	\item[b)] Assumption (A1) will be only used to prove measurability of the stochastic process whose limit defines the homogenized integrand. One can replace (A1) by upper semicontinuity, which simplifies the proof of measurability since in this case an infimum can be taken over a countable dense set.
	\item[c)] The a priori weaker growth condition
	\begin{equation*}
		c|\xi A(\w,x)|^p-\Lambda(\w,x)\leq f(\w,x,\xi)\leq |\xi A(\w,x)|^p+\Lambda(\w,x)
	\end{equation*}
	can be treated by considering the new integrand $\widetilde{f}(\w,x,\xi)=f(\w,x,\xi)+\Lambda(\w,x)$, which satisfies Assumption \ref{a.1} with the weight $\widetilde{\Lambda}=2\Lambda$, and which yields the same $\Gamma$-limit except an additive constant given by $\mathbb{E}[\Lambda(\cdot,0)]|D|$ (this is a consequence of the ergodic theorem \ref{thm.additiv_ergodic}). The growth condition (A2) simplifies some estimates.
	\item[d)] Stationarity (A3) and the stochastic integrability assumptions in (A2) together with Fubini's theorem show that for a.e. $\w\in\Omega$ it holds that $|A(\w,\cdot)|^p,\Lambda(\w,\cdot),|A(\w,\cdot)^{-1}|^{p/(p-1)}\in L^1_{\rm loc}(\R^d)$.
\end{itemize}
\end{remark}
As explained in the introduction, we also consider a set of strengthened assumptions to obtain further results on the stochastic homogenization of degenerate nonlinear elliptic PDEs in divergence form.
\begin{assumption}\label{a.2}
In addition to Assumption \ref{a.1}, for all $\w\in\Omega$ and all $x\in\R^d$, the function $\xi\mapsto f(\w,x,\xi)$ is differentiable and strictly convex.
\end{assumption} 
\section{Main results}\label{s.results}
In this section we state the results of this paper. We first state the $\Gamma$-convergence result without boundary conditions and external forces. For reader's convenience we recall that
\begin{equation*}
F_{\e}(\w,u,D)=
\begin{cases}
	\displaystyle \int_D f(\w,\tfrac{x}{\e},\nabla u(x))\dx &\mbox{if $u\in W^{1,1}(D,\R^m)$},
	\\
	\\ +\infty &\mbox{otherwise on $L^1(D,\R^m)$.}
\end{cases}
\end{equation*}
\begin{theorem}\label{thm.Gamma_pure}
Under Assumption \ref{a.1}, almost surely as $\e\to 0$, the random integral functionals 
$u\mapsto F_{\e}(\w,u,D)$ do $\Gamma$-converge in $L^1(D,\R^m)$ to the deterministic integral functional $F_{\rm hom}:L^1(D;\R^m)\to [0,+\infty]$ given by
\begin{equation*}
F_{\rm hom}(u)=	\int_D f_{\rm hom}(\nabla u(x))\dx\quad\text{ if }u\in W^{1,p}(D,\R^m)
\end{equation*}
and $+\infty$ otherwise. The integrand is given by the asymptotic cell-formula
\begin{equation*}
f_{\rm hom}(\xi)=\lim_{t\to +\infty}\frac{1}{t^d}\inf\{F_1(\w,u,(0,t)^d):\,u\in \xi x+W^{1,1}_0((0,t)^d,\R^m)\}.
\end{equation*}
Moreover, the function $\xi\mapsto f_{\rm hom}(\xi)$ is continuous and satisfies the two-sided $p$-growth condition
\begin{equation*}
	c_0|\xi|^p\leq f_{\rm hom}(\xi)\leq C_0|\xi|^p+C_1
\end{equation*}
with $c_0=c\,\mathbb{E}[|A(\cdot,0)^{-1}|^{p/(p-1)}]^{1-p}$, $C_0=\sup_{|\eta|=1}\mathbb{E}[|\eta A(\cdot,0)|^p]$ and $C_1=\mathbb{E}[\Lambda(\cdot,0)]$. Here $c$ is given by Assumption \ref{a.1} and the supremum with respect to $\eta$ runs over $\R^{m\times d}$.
\end{theorem}
\begin{remark}[On the optimality of the integrability assumptions]\label{r.optimal}
	In \cite[Remark 2.3]{NSS} the optimality of the integrability assumption $|A(\cdot,0)|^p\in L^1(\Omega)$ and $|A(\cdot,0)^{-1}|^{p/(p-1)}\in L^1(\Omega)$ was discussed for the discrete setting in the following form (stated for simplicity in the continuum): given $k\in\N$ and $Q=(0,1)^d$, the formula
	\begin{equation*}
	f_{\rm per}(\xi):=\lim_{k\to +\infty}\mathbb{E}\left[\inf\left\{\dashint_{kQ}f(\w,x,\xi+\nabla u(x))\dx:\,u\in W^{1,1}_{\rm per}(kQ,\R^m)\right\}\right]
	\end{equation*}
	with periodic boundary conditions can degenerate in the sense that $f_{\rm per}(\xi)=+\infty$ or $f_{\rm per}(\xi)=0$ occurs for some examples and some $\xi$ when one of the two integrability conditions is violated. However, strictly speaking the equality $f_{\rm hom}=f_{\rm per}$ is usually proven under the assumption that $\Gamma$-convergence holds (see, for instance, \cite[Section 4.7]{NSS}) and in general one expects only the inequality $f_{\rm per}\leq f_{\rm hom}$, so just the degeneracy $f_{\rm per}(\xi)=+\infty$ transfers to $f_{\rm hom}(\xi)$. In Example \ref{ex:optimal} we slightly refine the argument of \cite{NSS} to obtain a stationary, ergodic integrand of the form $f(\w,x,\xi)=|\xi A(\w,x)|^p$ such that almost surely 
	\begin{align*}
	&\lim_{k\to +\infty}\inf\left\{\dashint_{kQ}f(\w,x,\xi+\nabla u(x))\dx:\,u\in W^{1,1}_{0}(kQ)\right\}
	\\
	=&
	\begin{cases}
	+\infty &\mbox{if $\xi\notin\R e_1$ and $|A(\cdot,0)|^p\notin L^1(\Omega)$,}
	\\
	0 &\mbox{if $\xi\in\R e_1$ and $|A(\cdot,0)^{-1}|^{p/(p-1)}\notin L^1(\Omega)$, but $|A(\cdot,0)|^p\in L^1(\Omega)$.}
	\end{cases}
	\end{align*}
	If both integrability conditions are violated, we are not able to to pass from periodic boundary conditions to Dirichlet boundary conditions in the case $\xi\in\R e_1$. Nevertheless, this example shows that our integrability assumptions are optimal for stationary, ergodic media in the sense that in general the multi-cell formula with affine boundary conditions degenerates when one of the two assumptions is violated.
\end{remark}
We next focus on Dirichlet boundary conditions. We also include a linear forcing term. Given $g\in W_{\rm loc}^{1,\infty}(\R^d,\R^m)$ and $f_{\e}\in L^d(D,\R^m)$, we define the constrained functional
\begin{equation}\label{eq:constrained}
F_{\e,f_{\e},g}(\w,u,D)=
\begin{cases}
	F_{\e}(\w,u,D)-\int_D f_{\e}(x)\cdot u(x)\dx &\mbox{if $u\in g+W^{1,1}_0(D,\R^m)$,}
	\\
	\\
	+\infty &\mbox{otherwise on $L^1(D,\R^m)$.}
\end{cases}
\end{equation}
Note that the integral involving $f_{\e}$ is finite for $u\in W^{1,1}(D,\R^m)$ due to the Sobolev embedding.
\begin{theorem}\label{thm:Dirichlet_and_forces}
Let $g\in W_{\rm loc}^{1,\infty}(\R^d,\R^m)$ and $f_{\e},f_0\in L^{d}(D,\R^m)$ be such that $f_{\e}\rightharpoonup f_0$ in $L^d(D,\R^m)$. Under Assumption \ref{a.1}, almost surely as $\e\to 0$, the random functionals $u\mapsto F_{\e,f_{\e},g}(\w,u,D)$ do $\Gamma$-converge in $L^1(D,\R^m)$ to the deterministic integral functional $F_{{\rm hom},f_0,g}:L^1(D;\R^m)\to [0,+\infty]$ given by
\begin{equation*}
\int_D f_{\rm hom}(\nabla u(x))\dx-\int_Df_0(x)\cdot u(x)\dx \quad\text{ if }u\in g+W^{1,p}_0(D,\R^m)
\end{equation*}
and $+\infty$ otherwise. The integrand $f_{\rm hom}$ is given by Theorem \ref{thm.Gamma_pure}. Moreover, any sequence $u_{\e}$  such that 
\begin{equation*}
	\limsup_{\e\to 0}F_{\e,f_{\e},g}(\w,u_{\e},D)<+\infty
\end{equation*}
is weakly compact in $W^{1,1}(D,\R^m)$ and hence strongly compact in $L^{d/(d-1)}(D,\R^m)$. 
\end{theorem}
\begin{remark}
The assumption $g\in W^{1,\infty}_{\rm loc}(\R^d,\R^m)$ ensures that $u=g$ is admissible in the infimum problem. Since $|A(\w,\cdot)|^p$ only belongs to $L^1_{\rm loc}(\R^d)$, the energy of $g$ is guaranteed to be finite only if $\nabla g\in L^{\infty}(D,\R^{m\times d})$. The integrability imposed on the external force $f_{\e}$ is the most general that still guarantees coercivity of the tilted functional under Dirichlet-boundary conditions. We could instead consider a fixed element $f_0$ in the dual space $W^{1,1}(D,\R^m)'$. Taking an $\e$-dependent family in this dual space would lead to a technical notion of convergence, that we do not discuss here.
\end{remark}
\begin{remark}[Existence and convergence of minimizers]\label{r.existence}
By the general theory of $\Gamma$-convergence, Theorem~\ref{thm:Dirichlet_and_forces} implies the convergence of (almost) minimizers to minimizers of the $\Gamma$-limit. To ensure existence of minimizers of $u\mapsto F_{\e,f_{\e},g}(\w,u,D)$ it suffices to assume that the map $\xi\mapsto f(\w,x,\xi)$ is quasiconvex. Indeed, then the functional is lower semicontinuous with respect to weak convergence in $W^{1,1}(D,\R^m)$ as can be seen as follows: for fixed ($\w,x)\in\Omega\times\R^d$ and given $j\in\N$, consider the quasiconvex envelope $f_j(\w,x,\xi)$ of $\xi\mapsto\min\{j(1+|\xi|),f(\w,x,\xi)\}$. Clearly $0\leq f_j(\w,x,\xi)\leq j(1+|\xi|)$, so that by \cite[Proposition 9.5]{Da} the function $f_j(\w,\cdot,\cdot)$ is a Carath\'eodory-function and by \cite[Theorem 8.11]{Da} the corresponding integral functional is lower semicontinuous with respect to weak convergence in $W^{1,1}(D,\R^m)$. As shown in the proof of \cite[Lemma 4.1]{Kristensen} it holds that $f_j(\w,x,\xi)\uparrow f(\w,x,\xi)$ as $j\to +\infty$, so that $F_{\e}(\w,\cdot)$ can be written as the supremum of weakly lower semicontinuous functions and is therefore weakly lower semicontinuous itself. The existence of minimizers can then be proven as in Lemma \ref{l.epsexistence}.	
\end{remark}
Our third theorem covers the case of an $\e$-dependent obstacle constraint on $D$. In what follows, vector-valued inequalities are understood componentwise. Given $\varphi_{\e}\in W^{1,\infty}(D,\R^m)$ such that $g\geq\varphi_{\e}$ on $\partial D$, we define
\begin{equation*}
F^{\varphi_{\e}}_{\e,f_{\e},g}(\w,u,D)=
\begin{cases}
F_{\e,f_{\e},g}(\w,u,D) &\mbox{if $u\geq\varphi_{\e}$ a.e. in $D$,}
\\
+\infty &\mbox{otherwise,}	
\end{cases}
\end{equation*}
where $F_{\e,f_{\e},g}$ is defined in \eqref{eq:constrained}.
\begin{theorem}\label{thm.obstacle}
Let $g$ and $f_{\e},f_0$ be as in Theorem \ref{thm:Dirichlet_and_forces} and let $\varphi_{\e},\varphi\in W^{1,\infty}(D,\R^m)$ be such that $\varphi_{\e}\overset{*}{\rightharpoonup}\varphi$ in $W^{1,\infty}(D,\R^m)$. Assume moreover that $g\geq\varphi_{\e}$ on $\partial D$ for all $\e>0$. Under Assumption \ref{a.1}, almost surely as $\e\to 0$, the random functionals $u\mapsto F^{\varphi_{\e}}_{\e,f_{\e},g}(\w,u,D)$ do $\Gamma$-converge in $L^1(D,\R^m)$ to the deterministic integral functional $F^{\varphi}_{{\rm hom},f_0,g}:L^1(D;\R^m)\to [0,+\infty]$ given by
\begin{equation*}
	\int_D f_{\rm hom}(\nabla u(x))\dx-\int_Df_0(x)\cdot u(x)\dx \quad\text{ if }u\in g+W^{1,p}_0(D,\R^m)\text{ and }u\geq\varphi\text{ a.e. in }D
\end{equation*}
and $+\infty$ otherwise. The integrand $f_{\rm hom}$ is given by Theorem \ref{thm.Gamma_pure}. Moreover, any sequence $u_{\e}$  such that 
\begin{equation*}
	\limsup_{\e\to 0}F^{\varphi_{\e}}_{\e,f_{\e},g}(\w,u_{\e},D)<+\infty
\end{equation*}
is weakly compact in $W^{1,1}(D,\R^m)$ and hence strongly compact in $L^{d/(d-1)}(D,\R^m)$. 
\end{theorem}
\begin{remark}
Similar to Remark \ref{r.existence}, the existence of minimizers for fixed $\e>0$ is guaranteed by the quasiconvexity of $f$ in the last variable. Indeed, it suffices to note that the constraint $u\geq\varphi_{\e}$ is closed under $L^1$-convergence. Moreover, by the Sobolev embedding we also have that $g\geq \varphi$ on $\partial D$, so that the limit functional is non-trivial.
\end{remark}

Our last result concerns the stochastic homogenization of the Euler-Lagrange equations associated to $F_{\e}(\w,\cdot,D)$. This is the only result where we consider the stronger Assumption \ref{a.2}. For notational convenience, given $g\in W^{1,\infty}_{\rm loc}(\R^d,\R^m)$, we introduce the affine energy space
\begin{equation*}
	\mathcal{A}_{g,\e}(\w):=\left\{u\in g+W^{1,1}_0(D,\R^m):\,\int_D|\nabla u(x)A(\w,\tfrac{x}{\e})|^p\dx<+\infty\right\}.
\end{equation*}
\begin{theorem}\label{thm.PDEs}
Under Assumption \ref{a.2} the function $f_{\rm hom}$ given by Theorem \ref{thm.Gamma_pure} is strictly convex and continuously differentiable. The derivative satisfies the estimate
\begin{equation*}
	|\nabla f_{\rm hom}(\xi)|\leq C(1+|\xi|^{p-1}).
\end{equation*}
Moreover, let $g$ and $f_{\e},f_0$ be as in Theorem \ref{thm:Dirichlet_and_forces}. Then for almost every $\w\in\Omega$ there exists a unique weak solution $u_{\e,\w}\in \mathcal{A}_{g,\e}(\w)$ of the PDE
\begin{equation}\label{eq:epsPDE}
	\begin{split}-{\rm div}(\partial_{\xi }f(\w,\tfrac{\cdot}{\e},\nabla u))&=f_{\e}\quad\text{ on }D,
	\\
	u&=g\quad\text{ on }\partial D.
	\end{split}
\end{equation}
As $\e\to 0$, the solution converges weakly in $W^{1,1}(D,\R^m)$ and strongly in $L^{d/(d-1)}(D,\R^m)$ to the unique weak solution $u\in g+W^{1,p}_0(D,\R^m)$ of the PDE 
\begin{align*}
	-{\rm div}(\nabla f_{\rm hom}(\nabla u))&=f_0\quad\text{ on }D,
	\\
	u&=g\quad\text{ on }\partial D.
\end{align*}
\end{theorem}
\begin{remark}\label{r.weaksol}
a) The weak formulation of \eqref{eq:epsPDE} has to be understood in the following sense: for every $\varphi\in\mathcal{A}_{0,\e}$ it holds that
\begin{equation*}
	\int_D \partial_{\xi}f(\w,\tfrac{x}{\e},\nabla u(x)) \nabla\varphi(x)\dx=\int_D f_{\e}(x)\cdot\varphi(x)\dx.
\end{equation*}
In particular, the equality is valid for all $\varphi\in W^{1,\infty}_0(D,\R^m)$. The homogenized equation can be tested against all functions $\varphi\in W^{1,p}_0(D,\R^m)$.

\vspace*{2.5mm}

\noindent b) The differentiability of $f_{\rm hom}$ can be proven without convexity assumptions. In this case, the non-asymptotic formula for $f_{\rm hom}$ in Lemma \ref{l.F_pot_formula} is not available and differentiability has to be proven starting with the multi-cell formula. For this reason, we need a more quantitative $C^1$-assumption on the map $\xi\mapsto f(\w,x,\xi)$. More precisely, if we assume that for some $c_0>0$ and $0<\alpha\leq\min\{1,p-1\}$ the estimates
\begin{equation}\label{eq:quantitativeC1}
| \partial_{\xi}f(\w,x,\xi_1)-\partial_{\xi}f(\w,x,\xi_0)|\leq c_1|A(\w,x)|(\Lambda(\w,x)^{\frac{1}{p}}+|\xi_1A(\w,x)|+|\xi_0A(\w,x)|)^{p-1-\alpha}|(\xi_1-\xi_0)A(\w,x)|^{\alpha}
\end{equation}
and
\begin{equation}\label{eq:quantitativeBound}
|\partial_{\xi}f(\w,x,0)|\leq\Lambda(\w,x)
\end{equation}
hold true, then $f_{\rm hom}$ is also continuously differentiable with $|\nabla f_{\rm hom}(\xi)|\leq C(1+|\xi|^{p-1})$. Since this might be of interest in the quasiconvex case (where the bound \eqref{eq:quantitativeBound} is automatically satisfied), we provide a proof of this fact in the appendix (cf. Lemma \ref{l.app_diff}). Note that by the chain rule \eqref{eq:quantitativeC1} is satisfied by functions of the form $f(\w,x,\xi)=g(\xi A(\w,x))$ with $g$ a homogeneous and deterministic function satisfying the non-degenerate version of \eqref{eq:quantitativeC1}.  

The existence of solutions to \eqref{eq:epsPDE} can be shown replacing the strict convexity in Assumption \ref{a.2} by the quasiconvexity of the map $\xi\mapsto f(\w,x,\xi)$. Indeed, a close inspection of the proof of Lemma \ref{l.epsexistence} reveals that, regarding existence, the convexity is only used to show that $F_{\e}(\w,\cdot)$ is lower semicontinuous with respect to weak convergence in $W^{1,1}(D,\R^m)$ and to have \eqref{eq:f_quantitativedifferentiable} available. As explained in Remark \ref{r.existence}, the quasiconvexity already implies the weak lower semicontinuity, while we prove \eqref{eq:f_quantitativedifferentiable} for separately convex functions. We could formulate an analogue of Theorem \ref{thm.PDEs} without strict convexity, but assuming quasiconvexity and \eqref{eq:quantitativeC1}-\eqref{eq:quantitativeBound}. However, solutions might be non-unique and therefore convergence only holds up to subsequences.
\end{remark}
\section{Proofs}\label{s.proofs}
In this section we occasionally apply the (sub)additive ergodic theorem. In such a situation we have to exclude a null set of $\Omega$. We stress that we apply the ergodic theorem only countably many times (some care has to be taken when proving Lemma \ref{l.existence_f_hom} below). For this reason, we do not always mention this step in the proofs, but assume tacitly that the element $\w\in\Omega$ is chosen from a suitable set of full probability.
\subsection{The ergodic theorem as weak convergence in $L^1$}
In this section we show a strengthened version of the additive ergodic theorem \ref{thm.additiv_ergodic} in the sense of weak convergence in $L^1(D)$. 
\begin{lemma}\label{l.weakL1}
Let $g\in L^1(\Omega)$ and $\{\tau_z\}_{z\in\R^d}$ be a measure-preserving, ergodic group action. Then for a.e. $\w\in\Omega$ the sequence of functions $x\mapsto g(\tau_{x/\e}\w)$ converges weakly in $L^1(D)$ as $\e\to 0$ to the constant function $\mathbb{E}[g]$. In particular, for a.e. $\w\in\Omega$ the sequences of functions $|A(\w,\cdot/\e)|^p$, $|A(\w,\cdot/\e)^{-1}|^{p/(p-1)}$ and $\Lambda(\w,\tfrac{x}{\e})$ converge weakly in $L^1(D)$ as $\e\to 0$ to the constant functions $\mathbb{E}[|A(\cdot,0)|^p]$, $\mathbb{E}[|A(\cdot,0)^{-1}|^{p/(p-1)}]$ and $\mathbb{E}[\Lambda(\cdot,0)]$, respectively. 
\end{lemma}
\begin{proof}
Splitting $g$ into its positive and negative part, we can assume without loss of generality that $g\geq 0$. Due to the ergodic theorem \ref{thm.additiv_ergodic}, for a.e. $\w\in\Omega$ it holds that
\begin{equation}\label{eq:onD}
	\lim_{\e\to 0}\dashint_Dg(\tau_{x/\e}\w)\dx=\lim_{\e\to 0}\dashint_{D/\e}g(\tau_y\w)\dy= \mathbb{E}[g].
\end{equation}
Hence for such $\w$ the sequence $g(\tau_{\cdot/\e}\w)$ is bounded in $L^1(D)$ and due to the biting lemma (see \cite{BaMu}) we find a function $\sigma_{\w}\in L^1(D)$, a subsequence $\e_n\to 0$ and a decreasing sequence of measurable sets $(E_j)_{j\in\N}$ with $\lim_{j\to +\infty}|E_j|=0$ such that $g(\tau_{\cdot/\e_n}\w)\rightharpoonup \sigma_{\w}$ in $L^1(D\setminus E_j)$ as $n\to +\infty$ for every $j\in\N$. In order to identify the function $\sigma_{\w}$, we follow \cite[Section 3]{BaMu} and consider the truncation of $g(\tau_{\cdot/\e}\w)$ defined for $k\in\N$ by $g_{\e,k}(\w,x):=\min\{g(\tau_{x/\e}\w),k\}$. For fixed $k\in\N$, the sequence $g_{\e,k}(\w,\cdot)$ is bounded in $L^{\infty}(D)$. Applying the ergodic theorem \ref{thm.additiv_ergodic} to $g_{\e,k}$, up to excluding another null set in $\Omega$, we deduce that
\begin{equation*}
	g_{\e,k}(\w,\cdot)\overset{*}{\rightharpoonup} \mathbb{E}[\min\{g,k\}]\quad\text{ in }L^{\infty}(D)\text{ as }\e\to 0.
\end{equation*}
By the monotone convergence theorem it holds that
\begin{equation*}
	\lim_{k\to +\infty}\mathbb{E}[\min\{g,k\}]=\mathbb{E}[g],
\end{equation*}
which according to \cite[Proposition on p. 659]{BaMu} yields that $\sigma_{\w}=\mathbb{E}[g]$. Since we assume $g$ to be nonnegative,  \cite[Proposition 2.67]{FoLe} and \eqref{eq:onD} allow us to upgrade the biting convergence to weak convergence in $L^1(D)$. Since the limit does not depend on the subsequence $\e_n$, this shows the claim. The second part of the lemma follows from  Assumption \ref{a.1}.
\end{proof}
\subsection{Compactness for energy-bounded sequences}
We derive a compactness property for the gradients of sequences $u_{\e}$ with equibounded energy. Note that, in the domain of the $\Gamma$-limit, gradients have better integrability than in the compactness statement. Finally, compactness for the sequence $u_{\e}$ itself can only hold with further assumptions that allow us to apply a Poincar\'e inequality. 
\begin{lemma}\label{l.compactness}
For $\e>0$ let $u_{\e}\in W^{1,1}(D,\R^m)$ be a function such that
\begin{equation*}
\sup_{\e\in (0,1)}F_{\e}(\w,u_{\e},D)<+\infty.
\end{equation*}
Then, as $\e\to 0$, the gradients $\nabla u_{\e}$ are relatively weakly compact in $L^{1}(D,\R^{m\times d})$. If moreover $u_{\e}$ is bounded in $L^1(D)$, then, up to subsequences, there exists $u\in W^{1,p}(D,\R^m)$ such that $u_{\e}\rightharpoonup u$ weakly in $W^{1,1}(D,\R^m)$.
\end{lemma}
\begin{proof}
We first show that $|\nabla u_{\e}|$ is bounded in $L^1(D)$. H\"older's inequality and the lower bound in Assumption~\ref{a.1} yield that
\begin{align}\label{eq:energy_estimate}
\int_D |\nabla u_{\e}(x)|\dx&\leq \int_{D} |\nabla u_{\e}(x)A(\w,\tfrac{x}{\e})|\,|A(\w,\tfrac{x}{\e})^{-1}|\dx\nonumber
\\
&\leq\bigg(\underbrace{\int_D|\nabla u_{\e}(x)A(\w,\tfrac{x}{\e})|^p\dx}_{\leq c^{-1}F_{\e}(\w,u_{\e},D)\leq C}\bigg)^{\frac{1}{p}}\left(\int_D|A(\w,\tfrac{x}{\e})^{-1}|^{\frac{p}{p-1}}\dx\right)^{\frac{p-1}{p}}
\end{align}
and due to Lemma \ref{l.weakL1} the right-hand side is bounded when $\e\to 0$. This proves the boundedness of $|\nabla u_{\e}|$ in $L^1(D)$. Since $D$ has finite measure, it remains to show that $\nabla u_{\e_n}$ is equi-integrable along any sequence $\e_n\to 0$. We are going to show that
\begin{equation}\label{eq:equiintegrable}
	\limsup_{k\to +\infty}\limsup_{\e\to 0}\int_{D\cap\{|\nabla u_{\e}|\geq k\}}|\nabla u_{\e}(x)|\dx=0,
\end{equation}
which implies the equi-integrability along sequences $\e_n\to 0$. Again H\"older's inequality implies that
\begin{equation*}
\int_{D\cap\{|\nabla u_{\e}|\geq k\}}|\nabla u_{\e}(x)|\dx\leq C^{\frac{1}{p}} \bigg(\;\int_{D\cap\{|\nabla u_{\e}|\geq k\}}|A(\w,\tfrac{x}{\e})^{-1}|^{\frac{p}{p-1}}\dx\bigg)^{\frac{p-1}{p}}.
\end{equation*}
Since $|\nabla u_{\e}|$ is bounded in $L^1(D)$, it follows from Markov's inequality that 
\begin{equation*}
	|D\cap\{|\nabla u_{\e}|\geq k\}|\leq k^{-1}\|\nabla u_{\e}\|_{L^1(D)}\leq Ck^{-1}.
\end{equation*}
Using the equi-integrability of $x\mapsto |A(\w,\tfrac{x}{\e})^{-1}|^{\frac{p}{p-1}}$ as $\e\to 0$ (see Lemma \ref{l.weakL1}), we obtain \eqref{eq:equiintegrable} due to
\begin{equation*}
	\lim_{k\to +\infty}\limsup_{\e\to 0}\int_{D\cap\{|\nabla u_{\e}|\geq k\}}|A(\w,\tfrac{x}{\e})^{-1}|^{\frac{p}{p-1}}\dx=0.
\end{equation*}

Finally, we treat the case when $u_{\e}$ is also bounded in $L^1(D,\R^m)$. Then by standard embedding theorems and the first part, we obtain that there exists $u\in W^{1,1}(D,\R^m)$ and a subsequence (not relabeled) such that $u_{\e}\rightharpoonup u$ in $W^{1,1}(D,\R^m)$. It remains to show that $u\in W^{1,p}(D,\R^m)$. Invoking Poincar\'e's inequality, it suffices to show that $\nabla u\in L^p(D,\R^{m\times d})$. As in \cite{NSS}, this will follow from duality. For any $\varphi\in L^{\infty}(D,\R^{m\times d})$ we have by weak convergence and H\"older's inequality that
\begin{align*}
\int_D|\nabla u(x)\cdot\varphi(x)|\dx&\leq\liminf_{\e\to 0}\int_D|\nabla u_{\e}(x)\cdot\varphi(x)|\dx
\\
&\leq \limsup_{\e\to 0}\left(\int_D|\nabla u_{\e}(x)A(\w,\tfrac{x}{\e})|^p\dx\right)^{\frac{1}{p}}\left(\int_D|A(\w,\tfrac{x}{\e})^{-1}|^{\frac{p}{p-1}}|\varphi(x)|^{\frac{p}{p-1}}\dx\right)^{\frac{p-1}{p}}
\\
&\leq C^{\frac{1}{p}}\limsup_{\e\to 0}\left(\int_D|A(\w,\tfrac{x}{\e})^{-1}|^{\frac{p}{p-1}}|\varphi(x)|^{\frac{p}{p-1}}\dx\right)^{\frac{p-1}{p}}.
\end{align*}
Using Lemma \ref{l.weakL1}, we know that the integrand in the last line converges weakly in $L^1(D)$ to the function $x\mapsto \mathbb{E}[|A(\cdot,0)^{-1}|^{\frac{p}{p-1}}]|\varphi(x)|^{\frac{p}{p-1}}$, from which we infer that
\begin{equation*}
\int_D|\nabla u(x)\cdot\varphi(x)|\dx\leq C\, \mathbb{E}[|A(\cdot,0)^{-1}|^{\frac{p}{p-1}}]^{\frac{p-1}{p}}\|\varphi\|_{L^{\frac{p}{p-1}}(D)}.
\end{equation*}
By density this estimate extends to the dual space of $L^p(D,\R^{m\times d})$, so that indeed $\nabla u\in L^p(D,\R^{m\times d})$.
\end{proof}

\subsection{Existence of the homogenized integrand}
Here we prove the existence of the limit defining $f_{\rm hom}(\xi)$. To this end, we introduce a suitable stochastic process for which we can apply the subadditive ergodic theorem. Some care has to be taken when dealing with different macroscopic gradients $\xi$ since the exceptional sets of $\w$, where convergence fails, should be independent of $\xi$. To this end, we establish a continuity property that allows to extend the convergence from rational matrices $\xi\in\mathbb{Q}^{m\times d}$ to all matrices by deterministic arguments. Moreover, the continuity property implies the continuity of $\xi\mapsto f_{\rm hom}(\xi)$, which will be crucial for proving the upper bound for the $\Gamma$-convergence by density arguments. 
\begin{lemma}\label{l.existence_f_hom}
Given $f$ satisfying Assumption \ref{a.1}, a bounded open set $A\subset\R^d$ and $\xi\in\R^{m\times d}$, we define
\begin{equation*}
	\mu_{\xi}(\w,O)=\inf\left\{F_1(\w,u,O):\,u-\xi x\in W^{1,1}_{0}(O,\R^m)\right\}.
\end{equation*}
Then a.s. for every cube $Q=x+(-\eta,\eta)^d\subset\R^d$ there exists the deterministic limit
\begin{equation*}
f_{\rm hom}(\xi)=\lim_{t\to +\infty}\frac{1}{|tQ|}\mu_{\xi}(\w,tQ),
\end{equation*}
which is independent of the cube $Q$. Moreover, we have the estimate
\begin{equation*}
c_0|\xi|^p\leq f_{\rm hom}(\xi)\leq C_0|\xi|^p+C_1
\end{equation*}
with $c_0=c\,\mathbb{E}[|A(\cdot,0)^{-1}|^{p/(p-1)}]^{1-p}$ ($c$ given by Assumption \ref{a.1}), $C_0=\sup_{|\eta|=1}\mathbb{E}[|\eta A(\cdot,0)|^p]$ and $C_1=\mathbb{E}[\Lambda(\cdot,0)]$, and the function $\xi\mapsto f_{\rm hom}(\xi)$ is continuous.
\end{lemma}
\begin{proof}
In order to apply the subadditive ergodic theorem, we first need to establish the integrability of the function $\w\mapsto \mu_{\xi}(\w,O)$. Measurability is proven in Lemma \ref{l.measurable} in the appendix. In order to show the integrability, we use the affine function $u(x)=\xi x$ as a candidate in the infimum problem. Since $F_1$ is nonnegative, the upper bound in Assumption \ref{a.1} implies that
\begin{equation}\label{eq:pointwisebound}
0\leq \mu_{\xi}(\w,O)\leq \int_O f(\w,x,\xi)\,\dx\leq \int_{O}|\xi A(\w,x)|^p+\Lambda(\w,x)\,\mathrm{d}x.
\end{equation}
From Tonelli's theorem we infer that
\begin{align}\label{eq:mu_bound}
\mathbb{E}\left[\mu_{\xi}(\cdot,O)\right]&\leq \int_O\mathbb{E}[|\xi A(\cdot,x)|^p]+\mathbb{E}[\Lambda(\cdot,x)]\,\mathrm{d}x=(\mathbb{E}[|\xi A(\cdot,0)|^p]|+\mathbb{E}[\Lambda(\cdot,0)])|O|\nonumber
\\
&\leq \left(\sup_{|\eta|= 1}\mathbb{E}[|\eta A(\cdot,0)|^p]|\xi|^p+\mathbb{E}[\Lambda(\cdot,0)]\right)|O|.
\end{align}
where the equality follows from stationarity and a change of variables in $\Omega$. Hence $\mu_{\xi}(\cdot,O)\in L^1(\Omega)$. We claim that $\mu_{\xi}$ is $\tau$-stationary in the sense that
\begin{equation}\label{eq:mu_stationary}
\mu(\tau_z\w,O)=\mu(\w,O+z)\quad\text{ for all }\w\in\Omega.
\end{equation}
Indeed, given $v\in \xi x+W_0^{1,1}(O,\R^m)$, the map $\widetilde{v}(x)=v(x-z)+\xi z$ belongs to $\xi x+W_0^{1,1}(O+z,\R^m)$ and by stationarity of the energy density $f$ we have
\begin{equation*}
	\int_{O+z}f(\w,x,\nabla\widetilde{v}(x))\dx=\int_{O}f(\w,x+z,\nabla v(x))\dx=\int_{O}f(\tau_z\w,x,\nabla v(x))\dx,
\end{equation*}
which implies \eqref{eq:mu_stationary} by minimizing both sides. Finally, if $(U_j)_{j=1}^n\subset\R^d$ are bounded open sets with
\begin{equation*}
\bigcup_{j=1}^n U_j\subset O, \quad U_j\cap U_k=\emptyset\text{ for all }1\leq j<k\leq n, \quad |O\setminus \bigcup_{j=1}^n U_j|=0,
\end{equation*} 
and for every $1\leq j\leq n$ we are given a function $v_j\in \xi x+W_0^{1,1}(U_j,\R^m)$, then the function $v=\sum_{j=1}^n v_j\chi_{U_j}$ belongs to $\xi x+W^{1,1}_0(O,\R^m)$ and therefore
\begin{equation*}
	\mu_{\xi}(\w,O)\leq F_1(\w,v,O)=\sum_{j=1}^n F_1(\w,v_j,U_j).
\end{equation*}
Minimizing the right-hand side with respect to the variables $v_j$, we deduce subadditivity in the form of
\begin{equation}\label{eq:mu_subadditive}
	\mu_{\xi}(\w,O)\leq \sum_{j=1}^n\mu(\w,U_j).
\end{equation}
It follows from the subadditive ergodic theorem (see \cite[Theorem 2.7]{AkKr}) that for a set of full probability there exists the a~priori random limit
\begin{equation}\label{eq:integer_convergence}
f_{0}(\w,\xi):=\lim_{\substack{n\to +\infty\\ n\in\N}}\frac{1}{|nQ|}\mu_{\xi}(\w,nQ)
\end{equation}
for all cubes of the form $Q=z+(-k,k)^d$ with integer vertices $k\in\N$ and $z\in\Z^d$. To extend the convergence to arbitrary sequences $t\to +\infty$ and general cubes $Q=x+(-\eta,\eta)^d$ with $x\in\R^d$ and $\eta>0$, we argue by approximation following the standard method in the non-degenerate case. Given $\delta>0$, we choose $x^{\pm}_{\delta}\in\mathbb{Q}^d$ and $\eta^{\pm}_{\delta}\in\Q\cap(0,+\infty)$ such that, setting $Q^{\pm}_{\delta}=x^{\pm}_{\delta}+(-\eta^{\pm}_{\delta},\eta^{\pm}_{\delta})^d$, it holds that 
\begin{equation}\label{eq:cube_preparation}
	\begin{split}
	&Q^{-}_{\delta}\subset Q\subset Q^+_{\delta},
	\\
	&\dist(\partial Q_{\delta}^{\pm},\partial Q)>0,
	\\
	&|Q^-_{\delta}|\geq (1-\delta)|Q^+_{\delta}|.
\end{split}
\end{equation}
Then there exists $R=R_{\delta}\in\Z$ such that the cubes $RQ^{\pm}_{\delta}$ have an integer center and integer vertices. Set $t_-=\lfloor t\rfloor$ as the integer part of $t$ and write
\begin{equation*}
tRQ= (tRQ\cap t_-RQ_{\delta}^-)\cup (tRQ\setminus t_-RQ^-_{\delta}).
\end{equation*}
%Given $t>0$, we define the numbers $k_{t}^-=\lfloor \eta t\rfloor-1$ and $k_t^+=\lfloor \eta t\rfloor+1$ and let $z_t\in\Z^d$ be a point such that $|z_t-tx|_{\infty}\leq 1/2$.  Then by construction one has
%\begin{equation*}
%Q_t^-:=z_t+(-k_t^-,k_t^+)\subset tQ\subset z_t+(-k_t^+,k_t^+)^d=:Q_t^+
%\end{equation*}
%and
%\begin{equation}\label{eq:balanced}
%\lim_{t\to +\infty}\frac{|Q_t^-|}{|tQ|}=\lim_{t\to +\infty}\frac{|tQ|}{|Q_t^+|}=1.
%\end{equation}
From subadditivity \eqref{eq:mu_subadditive}, the nonnegativity of $f$ and the bound \eqref{eq:pointwisebound} we deduce that
\begin{equation}\label{eq:fromintegertoall}
	\frac{1}{|tRQ|}\mu_{\xi}(\w,tRQ)\leq \frac{1}{|t_-RQ_{\delta}^-|}\mu_{\xi}(\w,t_-RQ_{\delta}^-)+\frac{(|\xi|^p+1)}{|tRQ|}\int_{tRQ\setminus t_-RQ_{\delta}^-}\underbrace{|A(\w,x)|^p+\Lambda(\w,x)}_{=:\kappa(\w,x)}\dx.
\end{equation}
Due to \eqref{eq:integer_convergence}, the first term on the right-hand side converges to $f_0(\w,\xi)$ as $t\to +\infty$. We want to write the last integral as a difference of integrals over cubes. To this end, note that the second condition in \eqref{eq:cube_preparation} implies that for $t$ large enough the inclusion $t_-RQ_{\delta}^-\subset tRQ$ holds true. Indeed, given $x\in Q_{\delta}^-$ we know that $tRx\in tRQ$ and $|tRx-t_-Rx|\leq R|x|$, but 
\begin{equation*}
	\dist(tRx,\partial tRQ)=tR\,\dist(x,\partial Q)\geq tR\,\dist(\partial Q_{\delta}^-,\partial Q).
\end{equation*}
Hence indeed
\begin{equation*}
\frac{1}{|tRQ|}\int_{tRQ\setminus t_-RQ_{\delta}^-}\kappa(\w,x)\dx=\dashint_{tRQ}\kappa(\w,x)\dx- \left(\frac{t_-}{t}\right)^d\frac{|Q_{\delta}|}{|Q|}\dashint_{t_-RQ_{\delta}^-}\kappa(\w,x)\dx.	
\end{equation*}
The additive ergodic theorem \ref{thm.additiv_ergodic} applied to $\kappa$, \eqref{eq:fromintegertoall} and the third condition in \eqref{eq:cube_preparation} yield that
\begin{equation*}
	\limsup_{t\to +\infty}\frac{1}{|tRQ|}\mu_{\xi}(\w,tRQ)\leq  f_{0}(\w,\xi)+(|\xi|^p+1)\mathbb{E}[\kappa(\cdot,0)]\delta .
\end{equation*}
Since this estimate holds for all diverging sequences $t\to +\infty$, it follows from the arbitrariness of $\delta$ that
\begin{equation*}
	\limsup_{t\to +\infty}\frac{1}{|tQ|}\mu_{\xi}(\w,tQ)\leq f_{0}(\w,\xi).
\end{equation*} 
A similar argument using the cubes $tRQ$ and $(t_-+1)RQ_{\delta}^+$ instead yields the reverse inequality
\begin{equation*}
	f_{0}(\w,\xi)\leq\liminf_{t\to +\infty}\frac{1}{|tQ|}\mu_{\xi}(\w,tQ).
\end{equation*}
We still need to prove that $f_0$ is deterministic\footnote{This is well-known to experts, but we could not find a reference for $\R^d$- stationary, subadditive processes.}. To this end, it suffices to show that $f_0$ is invariant under the group action $\tau_z$. Take any two cubes $Q_1\subset\subset Q\subset\subset Q_2$. Then for $t$ large enough it holds that 
\begin{equation*}
	tQ_1\subset t(Q+z/t)\subset tQ_2.
\end{equation*}
Using stationarity and subadditivity of $\mu_{\xi}$, as well as \eqref{eq:pointwisebound}, we obtain that for $z\in\R^d$ and $\kappa=|A|^p+\Lambda$
\begin{align*}
\frac{1}{|tQ|}\mu_{\xi}(\tau_z\w,tQ)=\frac{1}{|tQ|}\mu_{\xi}(\w,t(Q+z/t))&\leq \frac{1}{|tQ_1|}\mu_{\xi}(\w,tQ_1)+\frac{(|\xi|^p+1)}{|tQ|}\int_{t(Q+z/t)\setminus tQ_1}\kappa(\w,x)\dx
\\
&\leq \frac{1}{|tQ_1|}\mu_{\xi}(\w,tQ_1)+\frac{(|\xi|^p+1)}{|tQ|}\int_{tQ_2\setminus tQ_1}\kappa(\w,x)\dx.
\end{align*}
Applying the additive ergodic theorem \ref{thm.additiv_ergodic} to the last integral we infer that
\begin{equation*}
\limsup_{t\to +\infty}\frac{1}{|tQ|}\mu_{\xi}(\tau_z\w,tQ)\leq f_{0}(\w,\xi)+(|\xi|^p+1)\mathbb{E}[\kappa(\cdot,0)]\frac{|Q_2|-|Q_1|}{|Q|}.	
\end{equation*}
Since $Q_1\subset\subset Q\subset\subset Q_2$ were arbitrary, we conclude that
\begin{equation*}
\limsup_{t\to +\infty}\frac{1}{|tQ|}\mu_{\xi}(\tau_z\w,tQ)\leq f_{0}(\w,\xi).	
\end{equation*}
By a similar argument one proves the reverse inequality for the limit inferior, so that $\tau_z\w$ belongs to the set where the limit exists. Consequently $f_0(\w,\xi)=f_0(\tau_z\w,\xi)$ for almost every $\w\in\Omega$ and every $z\in\R^d$. Ergodicity then yields that $f_0$ is deterministic and we call this value $f_{\rm hom}(\xi)$.

It remains to prove the bounds for $f_{\rm hom}$. Let $Q$ be a fixed cube. Since $f_{\rm hom}$ is deterministic, an application of Fatou's lemma yields that
\begin{equation*}
f_{\rm hom}(\xi)=\mathbb{E}[f_{\rm hom}(\xi)]\leq \liminf_{t\to +\infty}\frac{1}{|tQ|}\mathbb{E}[\mu_{\xi}(\cdot,tQ)]\leq \sup_{|\eta|=1}\mathbb{E}[|\eta A(\cdot,0)|^p]\,|\xi|^p+\mathbb{E}[\Lambda(\cdot,0)],
\end{equation*}
where we used the bound \eqref{eq:mu_bound} in the last inequality. In order to prove the lower bound, note that for any $v\in \xi x+W_0^{1,1}(Q,\R^m)$ H\"older's inequality yields
\begin{equation*}
|\xi|=\left|\dashint_Q \nabla v(x)\dx\right|\leq \dashint_Q |\nabla v(x)|\dx\leq \left(\dashint_Q |\nabla v(x)A(\w,x)|^p\dx\right)^{\frac{1}{p}}\left(\dashint_Q |A(\w,x)^{-1}|^{\frac{p}{(p-1)}}\dx\right)^{\frac{p-1}{p}}.
\end{equation*}
Taking $p$th powers and using Assumption \ref{a.1} we infer that
\begin{equation*}
|\xi|^p\leq \frac{1}{c|Q|}F_1(\w,v,Q)\left(\dashint_Q |A(\w,x)^{-1}|^{\frac{p}{(p-1)}}\dx\right)^{p-1}.
\end{equation*}
Rearranging terms we obtain by the arbitrariness of $v$ that
\begin{equation}\label{eq:mu_superlinear}
c\left(\dashint_Q |A(\w,x)^{-1}|^{\frac{p}{(p-1)}}\dx\right)^{1-p}|\xi|^p\leq \frac{1}{|Q|}\mu_{\xi}(\w,Q).
\end{equation}
Replacing $Q$ by $tQ$ and letting $t\to +\infty$, the ergodic theorem \ref{thm.additiv_ergodic} yields that
\begin{equation*}
c\;\mathbb{E}\left[|A(\cdot,0)^{-1}|^{p/(p-1)}\right]^{1-p}|\xi|^p\leq f_{\rm hom}(\xi).
\end{equation*}
So far we have proved the almost sure existence with an exceptional set depending on $\xi$. In what follows, we remove this constraint. Fix $\xi_0,\xi\in\R^{m\times d}$. In order to compare $\mu_{\xi}$ and $\mu_{\xi_0}$, we use cubes of different size. For a cube $Q=x+(-\eta,\eta)^d$ and $s>0$ set $Q(s)=x+(-s\eta,s\eta)^d$ and fix $\delta>0$. Then there exists a smooth cut-off function $\varphi=\varphi_{\delta,t}\in C_c^{\infty}(\R^d,[0,1])$ such that 
\begin{equation*}
\varphi\equiv 1\text{ on }tQ,\quad\quad\varphi\equiv 0\text{ on }\R^d\setminus tQ(1+\delta/2),\quad\quad\|\nabla\varphi\|_{L^{\infty}(\R^d)}\leq \frac{C_Q}{\delta t}.
\end{equation*}
We first extend a given $v\in\xi x+W^{1,1}_0(tQ,\R^m)$ to $\R^d$ setting $v(x)=\xi x$ on $\R^d\setminus tQ$ and then define $\widetilde{v}\in \xi_0 x+W_0^{1,1}(tQ(1+\delta),\R^m)$ by
\begin{equation*}
	\widetilde{v}(x)=\varphi(x) v(x)+(1-\varphi(x))\xi_0 x.
\end{equation*}
By the properties of $\varphi$, the product rule, and Assumption \ref{a.1}, we can estimate
\begin{align*}
\mu_{\xi_0}(\w,tQ(1+\delta))&\leq F_1(\w,\widetilde{v},tQ(1+\delta))
\\
&\leq\int_{tQ}f(\w,x,\nabla v(x))\dx+\int_{tQ(1+\delta)\setminus tQ}|\nabla\widetilde{v}(x)|^p|A(\w,x)|^p+\Lambda(\w,x)\dx
\\
&\leq F_1(\w,v,tQ)+C\int_{tQ(1+\delta)\setminus tQ}\kappa(\w,x)(|\nabla\varphi(x)|^p|\xi x-\xi_0 x|^p+|\xi|^p+|\xi_0|^p+1)\dx
\\
&\leq F_1(\w,v,tQ)+C\int_{tQ(1+\delta)\setminus tQ}\kappa(\w,x)((\delta t)^{-p}|\xi-\xi_0|^p|x|^p+|\xi|^p+|\xi_0|^p+1)\dx.
\end{align*}
Since $|x|\leq C_Q(1+\delta)t$ on $tQ(1+\delta)$, for $\delta\leq 1$ we can pass to the infimum over $v$ to deduce that
\begin{equation*}
\mu_{\xi_0}(\w,tQ(1+\delta))\leq  \mu_{\xi}(\w,tQ)+C(\delta ^{-p}|\xi-\xi_0|^p+|\xi|^p+|\xi_0|^p+1)\int_{tQ(1+\delta)\setminus tQ}\kappa(\w,x)\dx.
\end{equation*}
Now assume that $\xi_0\in\mathbb{Q}^{m\times d}$ and consider $\w$ in the set of full probability such that the limit of $t\mapsto |tQ|^{-1}\mu_{\xi_0}(\w,tQ)$ at $+\infty$ exists for all rational matrices $\xi_0$ and all cubes. The additive ergodic theorem \ref{thm.additiv_ergodic} applied to $\kappa$ then yields that
\begin{equation}\label{eq:liminf_est}
f_{\rm hom}(\xi_0)\leq \liminf_{t\to +\infty}\frac{1}{|tQ|}\mu_{\xi}(\w,tQ)+C(\delta^{-p}|\xi-\xi_0|^p+|\xi|^p+|\xi_0|^p+1)\mathbb{E}[\kappa(\cdot,0)]\left(1-\frac{1}{(1+\delta)^d}\right).
\end{equation}
A similar construction based on the cubes $tQ$ and $tQ(1-\delta)$ yields that
\begin{equation*}
\mu_{\xi}(\w,tQ)\leq \mu_{\xi_0}(\w,tQ(1-\delta))+C(\delta ^{-p}|\xi-\xi_0|^p+|\xi|^p+|\xi_0|^p+1)\int_{tQ\setminus tQ(1-\delta)}\kappa(\w,x)\dx,
\end{equation*}
which again by the ergodic theorem shows that
\begin{equation}\label{eq:limsup_est}
	\limsup_{t\to +\infty}\frac{1}{|tQ|}\mu_{\xi}(\w,tQ)\leq f_{\rm hom}(\xi_0) +C(\delta^{-p}|\xi-\xi_0|^p+|\xi|^p+|\xi_0|^p+1)\mathbb{E}[\kappa(\cdot,0)]\left(1-(1-\delta)^d\right).
\end{equation}
Combining the two inequalities \eqref{eq:liminf_est} and \eqref{eq:limsup_est} then yields
\begin{align*}
\limsup_{t\to +\infty}\frac{1}{|tQ|}\mu_{\xi}(\w,tQ)&\leq \liminf_{t\to +\infty}\frac{1}{|tQ|}\mu_{\xi}(\w,tQ)
\\
&\quad+C(\delta^{-p}|\xi-\xi_0|^p+|\xi|^p+|\xi_0|^p+1)\mathbb{E}[\kappa(\cdot,0)]\left(2-(1-\delta)^d-\frac{1}{(1+\delta)^d}\right).	
\end{align*}
Considering a sequence of rational matrices that converges to $\xi$ and then letting $\delta\to 0$ we deduce that 
\begin{equation*}
	\limsup_{t\to +\infty}\frac{1}{|tQ|}\mu_{\xi}(\w,tQ)\leq \liminf_{t\to +\infty}\frac{1}{|tQ|}\mu_{\xi}(\w,tQ),
\end{equation*}
so that the limit exists for such $\w$ and consequently the convergence statement indeed holds for a uniform (with respect to $\xi$ and $Q$) set of full probability. It remains to show the continuity of $f_{\rm hom}$. With what we have shown we can evaluate \eqref{eq:liminf_est}  for all $\xi,\xi_0\in\R^{m\times d}$ and infer that for all $0<\delta<1$ it holds that
\begin{equation*}
f_{\rm hom}(\xi_0)\leq  f_{\rm hom}(\xi)+C(\delta^{-p}|\xi-\xi_0|^p+|\xi|^p+|\xi_0|^p+1)\mathbb{E}[\kappa(\cdot,0)]\left(1-\frac{1}{(1+\delta)^d}\right).
\end{equation*}
This estimate implies the continuity of $f_{\rm hom}$.
\end{proof}

\subsection{Optimality of the integrability assumptions}
In the proof of Lemma \ref{l.existence_f_hom} we have strongly used the integrability properties of the matrix $A$ given by Assumption \ref{a.1}. In this section, we show that they are indeed necessary to have a non-degenerate model in the sense of Remark \ref{r.optimal}.
\begin{example}\label{ex:optimal}
We only consider the scalar case $m=1$. The vectorial case can be treated the same way arguing separately for each component. Following \cite{NSS}, we consider a sequence $\lambda_k:\Omega\to (0,+\infty)$ of iid random variables. Those can be extended to a random function $\lambda:\Omega\times\R\to (0,+\infty)$ via piecewise constant interpolation on the intervals $[k,k+1)$ with $k\in\Z$. The resulting function, which is a priori only $\Z$-stationary and ergodic, can be turned into a $\R$-stationary and ergodic weight with the same piecewise constant structure using non-integer translations on an extended probability space (for details on this construction, see \cite[p. 236]{JKO}). We then define the integrand $f(\w,x,\xi)=|\xi A(\w,x)|^p$, where the matrix $A:\Omega\times\R^d\to\mathbb{D}_d$ is the matrix $A(\w,x)=\lambda(\w,x_1)I$ with $I$ denoting the identity from $\R^d$ to $\R^d$. By definition, it holds that $|A(\w,x)|^p=\lambda(\w,x_1)^p$ and $|A(\w,x)^{-1}|^{p/(p-1)}=\lambda(\w,x_1)^{-p/(p-1)}$.

We first discuss the case $\mathbb{E}[\lambda(\cdot,0)^p]=+\infty$. Fix $u\in W^{1,1}_0(kQ)$ and define the lower dimensional cube $Q_{d-1}=(0,1)^{d-1}$. Then for a.e. $x_1\in (0,k)$ it holds that $u(x_1,\cdot)\in W_0^{1,1}(kQ_{d-1})$ and by Tonelli's theorem
\begin{equation*}
\dashint_{kQ}|(\xi+\nabla u(x))A(\w,x)|^p\dx\geq\dashint_0^k\lambda(\w,x_1)^p \dashint_{kQ_{d-1}}|(\xi_2,\ldots,\xi_d)+\nabla_yu(x_1,y)|^p\dy
\end{equation*}
The inner integral on the right-hand side is minimal for $u(x_1,\cdot)\equiv 0$ due to Jensen's inequality. Therefore
\begin{equation*}
	\inf\left\{\dashint_{kQ}|(\xi+\nabla u(x))A(\w,x)|^p\dx:\,u\in W^{1,1}_0(kQ)\right\}\geq |(\xi_2,\ldots,\xi_d)|^p\dashint_{kQ}\lambda(\w,x_1)^p\dx.
\end{equation*}
Combining a truncation of the weight $\lambda$ with the ergodic theorem \ref{thm.additiv_ergodic}, for $\xi\notin\R e_1$ it follows that a.s.
\begin{equation}\label{eq:toinfinity}
	\lim_{k\to +\infty}\inf\left\{\dashint_{kQ}|(\xi+\nabla u(x))A(\w,x)|^p\dx:\,u\in W^{1,1}_0(kQ)\right\}=+\infty.
\end{equation}
Next we consider the case when $\mathbb{E}[\lambda(\cdot,0)^{-p/(p-1)}]=+\infty$ and $\xi=e_1$. In order to treat the case of Dirichlet boundary conditions instead of periodic minimizers, we need to assume in addition that $\mathbb{E}[\lambda(\cdot,0)^p]<+\infty$. Given $s\in (0,1)$, we define the regularized weight $\lambda_s(\w,x_1)=\max\{s,\lambda(\w,x_1)\}$ and the map $u\in W^{1,1}_{\rm per}(kQ)$ by $u(x)=v_{k,s}(x_1)$, where
\begin{equation*}
	v_{k,s}(x_1)=\left(\dashint_0^k\lambda_s(\w,t)^{-p/(p-1)}\,\mathrm{d}t\right)^{-1}\int_0^{x_1}\lambda_s(\w,t)^{-p/(p-1)}\,\mathrm{d}t-x_1.
\end{equation*}
An elementary calculation yields that 
\begin{align*}
	\dashint_{kQ}\lambda_s(\w,x_1)^p|e_1+\nabla u(x)|^p\dx= \left(\dashint_0^k\lambda_s(\w,t)^{-p/(p-1)}\,\mathrm{d}t\right)^{1-p}.
\end{align*}
Passing to the limit in $k$, it follows that
\begin{align*}
	f_{\rm per,s}(e_1,\w):&=\limsup_{k\to +\infty}\inf\left\{\dashint_{kQ}\lambda_{s}(\w,x_1)^p|e_1+\nabla u(x)|^p\dx:\,u\in W^{1,1}_{\rm per}(kQ)\right\}
	\\
	&\leq \lim_{k\to +\infty}\left(\dashint_0^k\lambda_s(\w,t)^{-1/(p-1)}\,\mathrm{d}t\right)^{1-p}=\mathbb{E}[\lambda_s(\cdot,0)^{-1/(p-1)}]^{1-p},
\end{align*}
where the last equality follows from the additive ergodic theorem. Since the weight $\lambda_s$ satisfies the deterministic estimate $s\leq\lambda_s$ and is integrable, our $\Gamma$-convergence result in Theorem \ref{thm.Gamma_pure} holds and by standard arguments (see \cite[Section 4.7]{NSS}) one can prove that almost surely
\begin{equation*}
	f_{\rm per,s}(e_1,\w)=\lim_{k\to +\infty}\inf\left\{\dashint_{kQ}\lambda_s(\w,x)^p|e_1+\nabla u(x)|^p\dx:\,u\in W_0^{1,1}(kQ)\right\}.
\end{equation*}
Since $\lambda\leq\lambda_s$ for all $s>0$, we deduce that
\begin{equation*}
	\mathbb{E}[\lambda_s(\cdot,0)^{-1/(p-1)}]^{1-p}\geq \limsup_{k\to +\infty}\inf\left\{\dashint_{kQ}|(e_1+\nabla u(x))A(\w,x)|^p\dx:\,u\in W_0^{1,1}(kQ)\right\}.
\end{equation*}
Letting $s\to 0$ in the left-hand side, we deduce from the monotone convergence theorem that 
\begin{equation*}
	0=\lim_{k\to +\infty}\inf\left\{\dashint_{kQ}|(e_1+\nabla u(x))A(\w,x)|^p\dx:\,u\in W_0^{1,1}(kQ)\right\},
\end{equation*}
so that the multi-cell formula with affine boundary condition $e_1x$ degenerates. By $p$-homogeneity it follows that we can replace $e_1$ by $re_1$ for all $r\in\R$ and we obtain the statement of Remark \ref{r.optimal}.
\end{example}

\subsection{Proof of unconstrained $\Gamma$-convergence}
Here we prove the $\Gamma$-convergence without boundary conditions and external forces. We first prove the existence of a recovery sequence for a given map $u\in W^{1,p}(D,\R^m)$. We argue by density using a piecewise affine approximation of $u$ and then focus on a single cell $T$ where the target function is affine with gradient $\xi$. We partition this cell (up to a boundary layer) with small cubes. On such cubes $Q$ we use a rescaled version of an almost solution of the infimum problem defining $\mu_{\xi}(\w,\e^{-1}Q)$. This construction will lead to the following result. 
\begin{proposition}\label{p.ub}
Given $u\in W^{1,p}(D,\R^m)$, there exists a sequence $u_{\e}\in W^{1,1}(D,\R^m)$ such that $u_{\e}\to u$ in $L^1(D,\R^m)$ and
\begin{equation*}
\limsup_{\e\to 0}F_{\e}(\w,u_{\e},D)\leq \int_D f_{\rm hom}(\nabla u(x))\,\mathrm{d}x.
\end{equation*}
\end{proposition}
\begin{proof}
For convenience, we define the abstract $\Gamma$-upper limit of $F_{\e}$ as the function $F''(\w,\cdot):L^1(D,\R^m)\to [0,+\infty]$ given by
\begin{equation*}
F''(\w,u)=\inf\{\limsup_{\e\to 0}F_{\e}(\w,u_{\e},D):\,u_{\e}\to u\text{ in }L^1(D,\R^m)\}.
\end{equation*}
It is well-known that $u\mapsto F''(\w,u)$ is lower semicontinuous on $L^1(D,\R^m)$. We will show that 
\begin{equation}\label{eq:Gamma-limsup}
F''(\w,u)\leq \int_D f_{\rm hom}(\nabla u(x))\dx	
\end{equation}
for all $u\in W^{1,p}(D,\R^m)$. Since $D$ has Lipschitz boundary, we can assume without loss of generality that $u\in W^{1,p}(\R^d,\R^m)$. Due to the $p$-growth from above and continuity of $f_{\rm hom}$, the right-hand side in \eqref{eq:Gamma-limsup} is continuous with respect to strong convergence in $W^{1,p}(D,\R^m)$. Hence by standard density arguments it suffices to prove \eqref{eq:Gamma-limsup} for continuous, piecewise affine functions $u:\R^d\to\R^m$, that means, there exists a locally finite triangulation $\{T_i\}_{i\in\N}$ of $\R^d$ into non-degenerate $(d+1)$-simplices such that $u|_{T_i}$ is affine for every $i\in\N$ and $u$ is continuous. We will provide a local construction such that $u_{\e}\in u+W^{1,1}_0(T_i,\R^m)$ for all $i\in\N$, which due to the continuity of $u$ can then be glued together to obtain a full recovery sequence $u_{\e}\in W^{1,1}_{\rm loc}(\R^d,\R^m)$. Hence for the moment we consider a single simplex $T=T_i$ and write $u|_T(x)=\xi x+b$. Given $0<\delta\ll 1$, consider the set of cubes
\begin{equation*}
\mathcal{Q}_{\delta}(T):=\{Q=\delta z+(-\delta/2,\delta/2)^d:\, z\in\Z^d,\,Q\subset T\}
\end{equation*}
and define an interior approximation of $T$ by $T_{\delta}:=\bigcup_{Q\in\mathcal{Q}_{\delta}(T)}Q$. We will define the sequence $u_{\e}=u_{\e,\delta}$ separately on each cube in $\mathcal{Q}_{\delta}(T)$. For $\e>0$ and $Q\in\mathcal{Q}_{\delta}(T)$ we choose $v_{\e,Q}\in W^{1,1}_0(Q/\e,\R^m)$ satisfying
\begin{equation*}
\int_{Q/\e}f(\w,x,\xi+\nabla v(x))\dx\leq \mu_{\xi}(\w,Q/\e)+\e,
\end{equation*}
for which we recall that
\begin{equation*}
	\mu_{\xi}(\w,Q/\e)=\inf\left\{\int_{Q/\e}f(\w,x,\nabla v(x))\dx:\,v\in \xi x+W_0^{1,1}(Q/\e,\R^m)\right\}.
\end{equation*}
We then define $u_{\e}$ on $T$ (and depending on $\delta$) by $u_{\e}(x)=\xi x+b+\sum_{Q\in\mathcal{Q}_{\delta}(T)}\e v_{\e, Q}(x/\e)\chi_{Q}(x)$. Due to the zero boundary conditions of $v_{\e,Q}$ it holds that $u_{\e}\in u+W^{1,1}_0(T,\R^m)$. From Lemma \ref{l.existence_f_hom} we know that
\begin{equation*}
\lim_{\e\to 0}\frac{1}{|Q/\e|}\mu_{\xi}(\w,Q/\e)\to f_{\rm hom}(\xi).	
\end{equation*}
Using the upper bound for $f$ in Assumption \ref{a.1}, we can estimate the energy of $u_{\e}$ on the simplex $T$ by
\begin{align*}
F_{\e}(\w,u_{\e},T)&= \sum_{Q\in\mathcal{Q}_{\delta}(T)}\int_Qf(\w,\tfrac{x}{\e},\xi+\nabla v_{\e,Q}(\tfrac{x}{\e}))\dx+\int_{T\setminus T_{\delta}}f(\w,\tfrac{x}{\e},\xi)\dx
\\
&\overset{y=x/\e}{\leq} \sum_{Q\in\mathcal{Q}_{\delta}(T)}\e^{d}\int_{Q/\e}f(\w,y,\xi+\nabla v_{\e,Q}(y))\dy+\e^d\int_{(T\setminus T_{\delta})/\e}|\xi|^p|A(\w,y)|^p+\Lambda(\w,y)\dy.
\end{align*}
Since $|Q/\e|=|Q|\e^{-d}$ and $\mathcal{Q}_{\delta}(T)$ is a family of pairwise disjoint cubes, we deduce in the limit $\e\to 0$ that
\begin{align*}
	\limsup_{\e\to 0}F_{\e}(\w,u_{\e},T)&\leq  \sum_{Q\in\mathcal{Q}_{\delta}(T)}|Q|f_{\rm hom}(\xi)+\limsup_{\e\to 0}\e^d\int_{(T\setminus T_{\delta})/\e}|\xi|^p|A(\w,y)|^p+\Lambda(\w,y)\dy
	\\
	&\leq \int_T f_{\rm hom}(\nabla u(y))\dy+(|\xi|^p+1)\limsup_{\e\to 0}\e^d\int_{(T\setminus T_{\delta})/\e}|A(\w,y)|^p+\Lambda(\w,y)\dy.
\end{align*}
The last limit can be treated with the ergodic theorem \ref{thm.additiv_ergodic} applied to $|A|^p+\Lambda$, which yields that
\begin{equation*}
\lim_{\e\to 0}\e^d\int_{(T\setminus T_{\delta})/\e}|A(\w,y)|^p+\Lambda(\w,y)\dy=|T\setminus T_{\delta}|\,\mathbb{E}[|A(\cdot,0)|^p+\Lambda(\cdot,0)].
\end{equation*}
The right-hand side vanishes when $\delta\to 0$. Summing up, we have proved that
\begin{equation}\label{eq:almost_recovery}
\limsup_{\e\to 0}F_{\e}(\w,u_{\e},T)\leq \int_Tf_{\rm hom}(\nabla u(x))\dx+(|\xi|^p+1)|T\setminus T_{\delta}|\,\mathbb{E}[|A(\cdot,0)|^p+\Lambda(\cdot,0)].	
\end{equation}
We next consider the asymptotic behavior of $u_{\e}$. Due to \eqref{eq:almost_recovery} we can combine Poincar\'e's inequality and Lemma \ref{l.compactness} to infer that $u_{\e}$ is bounded in $W^{1,1}(T,\R^m)$ and that up to a subsequence (not relabeled) $u_{\e}\to \widetilde{u}_{\delta}$ in $L^1(T,\R^m)$. Let us estimate the difference between $\widetilde{u}_{\delta}$ and the target function $u$ in $L^1(T,\R^m)$. By Poincar\'e's inequality on the small cubes $Q\in\mathcal{Q}_{\delta}$, we have that
\begin{align*}
\|\widetilde{u}_{\delta}-u\|_{L^1(T)}&=\lim_{\e\to 0}\sum_{Q\in\mathcal{Q}_{\delta}(T)} \int_Q|\e v_{\e,Q}(\tfrac{x}{\e})|\dx\leq C\delta\liminf_{\e\to 0}\sum_{Q\in\mathcal{Q}_{\delta}(T)}\int_{Q}|\nabla v_{\e,Q}(\tfrac{x}{\e})|\dx
\\
&\leq C\delta \left(|\xi||T|+\liminf_{\e\to 0}\sum_{Q\in\mathcal{Q}_{\delta}(T)}\int_{Q}|\xi+\nabla v_{\e,Q}(\tfrac{x}{\e})|\dx\right)
\\
&\leq  C\delta \left(|\xi||T|+\liminf_{\e\to 0}\int_{T}|\nabla u_{\e}(x)|\dx\right).
\end{align*}
Applying H\"older's inequality, the right-hand side can be further estimated leading to
\begin{equation*}
\|\widetilde{u}_{\delta}-u\|_{L^1(T)}\leq C\delta\left(|\xi||T|+\lim_{\e\to 0}\left(\int_{T}f(\w,\tfrac{x}{\e},\nabla u_{\e})\dx\right)^{\frac{1}{p}}\left(\int_T |A(\w,\tfrac{x}{\e})^{-1}|^{\frac{p}{p-1}}\dx\right)^{\frac{p-1}{p}}\right).
\end{equation*}
Due to Lemma \ref{l.weakL1} and \eqref{eq:almost_recovery} the term inside the parenthesis is finite, so we conclude that $\widetilde{u}_{\delta}\to u$ in $L^1(T,\R^m)$ as $\delta\to 0$. Considering then the global sequence $u_{\e}\in W^{1,1}(\R^d,\R^m)$ and the corresponding $L^1(D,\R^m)$-limit $\widetilde{u}_{\delta}$, from \eqref{eq:almost_recovery} and lower semicontinuity of $u\mapsto F''(\w,u)$ we deduce that
\begin{align*}
F''(\w,u)&\leq\liminf_{\delta\to 0}F''(\w,\widetilde{u}_{\delta})\leq\liminf_{\delta\to 0} \sum_{T\cap D\neq\emptyset}\limsup_{\e\to 0}F_{\e}(\w,u_{\e},T)\leq \sum_{T\cap D\neq\emptyset}\int_Tf_{\rm hom}(\nabla u(x))\dx
\\
&\leq \int_D f_{\rm hom}(\nabla u(x))\dx+\sum_{T\cap\partial D\neq\emptyset}\int_Tf_{\rm hom}(\nabla u(x))\dx.
\end{align*}
For the fixed piecewise affine function $u$, we can refine the triangulation with triangles of arbitrarily small diameter and repeat the above construction to make the last term arbitrarily small. Therefore
\begin{equation*}
F''(\w,u)\leq \int_{D}f_{\rm hom}(\nabla u(x))\dx,
\end{equation*}
that is, equation \eqref{eq:Gamma-limsup} holds true for all piecewise affine functions, which concludes the proof.
\end{proof}

Next we prove the $\Gamma$-liminf inequality. The basic idea is to use the standard blow-up method. The subtle point is that we need to locally adapt the boundary values of a sequence with equibounded energy, which is beyond the standard framework with the presence of the weight $A(\w,\cdot)$. We first prove the lower bound under the assumptions that the sequence under consideration is bounded in $L^{\infty}$. With a vectorial truncation we can then remove this restriction in a second step. Since we need the vectorial truncation several times in the paper, we formulate it as a separate result below. It is here where we need that $A$ is a diagonal matrix.
\begin{lemma}\label{l.truncation}
Let $u_{\e}\in W^{1,1}(D,\R^m)$ be such that $u_{\e}\to u$ in $L^1(D,\R^m)$ as $\e\to 0$. Then for every $\delta>0$ there exists a constant $C_{\delta}>0$ and a function $u_{\e,\delta}\in W^{1,1}(D,\R^m)$ such that $u_{\e,\delta}=u_{\e}$ a.e. on $\{|u_{\e}|\leq \delta^{-1}\}$ and for a.e. $x\in D$
\begin{equation*}
	|u_{\e,\delta}(x)|\leq |u_{\e}(x)|,\quad\quad |u_{\e,\delta}(x)|\leq C_{\delta}.
\end{equation*}
Moreover, for $\e$ small enough it holds that
\begin{equation*}
F_{\e}(\w,u_{\e,\delta},D)\leq (1+\delta)F_{\e}(\w,u_{\e},D)+\delta.
\end{equation*}
\end{lemma}
\begin{proof}
As proven in \cite[Section 4]{CDMSZ18}, for every $r>0$ there exists $\psi_{r}\in C_c^{\infty}(\R^m,\R^m)$ that is $1$-Lipschitz, $\psi_r(x)=x$ for $|x|\leq r$, $|\psi_r(x)|\leq 2r$ and ${\rm supp}(\psi_r)\subset B_{3r}(0)$. Note that the $1$-Lipschitz continuity implies that $|\psi_r(x)|\leq |x|$. Adapting \cite[Lemma 4.1]{CDMSZ18} to our setting, we now construct a suitable truncation of $u_{\e}$. To this end, we discretize the co-domain of $u_{\e}$ and use an averaging argument to find a good truncation level. Fix $N\in\N$ and for $0\leq i\leq N$ let $r_{0}=N$ and $r_{i+1}=3r_i$. Then $N\leq r_i\leq 3^NN<+\infty$. We consider the truncated sequence $u_{\e,i}=\psi_{r_{i}}\circ u_{\e}$. By the chain rule $\nabla u_{\e,i}=\nabla \psi_{r_i}(u_{\e})\nabla u_{\e}$, so that 
\begin{equation*}
\nabla u_{\e,i}=\nabla u_{\e}\quad\text{ a.e. on }\{|u_{\e}|\leq r_i\},
	\quad\quad
\nabla u_{\e,i}(x)=0\quad\text{ a.e. on }\{|u_{\e}|\geq 3r_i=r_{i+1}\}.
\end{equation*}
In the intermediate region $\{r_i<|u_{\e}|<r_{i+1}\}$ we use that $A$ is a diagonal matrix. The $k$th partial derivatives of $u_{\e,i}$ satisfies
\begin{equation*}
|\partial_k u_{\e,i}|=|\nabla \psi_{r_i}(u_{\e})\partial_ku_{\e}|\leq |\partial_k u_{\e}|,
\end{equation*}
where the last inequality follows from the $1$-Lipschitz continuity of $\psi_{r_i}$. Write $A={\rm diag}(\lambda_{k})_{k=1}^d$. Then for the Frobenius norm $|\cdot|_F$ we have that
\begin{equation*}
|\nabla u_{\e,i}(x)A(\w,\tfrac{x}{\e})|_F=\left(\sum_{k}\lambda_{k}(\w,\tfrac{x}{\e})^2|\partial_{k}u_{\e,i}|^2\right)^{\frac{1}{2}}\leq\left(\sum_{k,\ell}\lambda_{k}(\w,\tfrac{x}{\e})^2|\partial_k u_{\e}|^2\right)^{\frac{1}{2}} =|\nabla u_{\e}(x)A(\w,\tfrac{x}{\e})|_F.
\end{equation*}
Using the upper bound in Assumption \ref{a.1}, it follows from the equivalence of norms on $\R^{m\times d}$ that
\begin{equation*}
	F_{\e}(\w,u_{\e,i},D)\leq F_{\e}(\w,u_{\e},D)+\int_{\{r_i< |u_{\e}|< r_{i+1}\}}C|\nabla u_{\e}(x)A(\w,\tfrac{x}{\e})|^p+\Lambda(\w,\tfrac{x}{\e})\dx+\int_{\{|u_{\e}|\geq r_{i+1}\}}\Lambda(\w,\tfrac{x}{\e})\dx.
\end{equation*}
By construction the sets $\{r_i< |u_{\e}|< r_{i+1}\}_{i=1}^N$ are pairwise disjoint and contained in $D$, so there exists $1\leq i_*\leq N$ such that
\begin{align*}
	F_{\e}(\w,u_{\e,i_*},D)&\leq\frac{1}{N}\sum_{i=1}^N F_{\e}(\w,u_{\e,i},D)
	\\
	&\leq F_{\e}(\w,u_{\e},D)+\frac{C}{N}\int_{D}|\nabla u_{\e}(x)A(\w,\tfrac{x}{\e})|^p\dx+2\int_{\{|u_{\e}|\geq N\}}\Lambda(\w,\tfrac{x}{\e})\dx
	\\
	&\leq F_{\e}(\w,u_{\e},D)+\frac{C}{N}F_{\e}(\w,u_{\e},D)+2\int_{\{|u_{\e}|\geq N\}}\Lambda(\w,\tfrac{x}{\e})\dx.
\end{align*}
Due to Lemma \ref{l.weakL1} we know that $\Lambda(\w,\cdot/\e)$ is equi-integrable as $\e\to 0$. Moreover, since $u_{\e}$ converges in $L^1$, it follows that the set $\{|u_{\e}|>N\}$ has small measure for $N$ large uniformly in $\e$. Hence, given $\delta>0$ and $\e>0$ small enough, there exists $N_{\delta}\geq \delta^{-1}$ such that
\begin{equation*}
	2\int_{\{|u_{\e}|\geq N_{\delta}\}}\Lambda(\w,\tfrac{x}{\e})\dx\leq \delta
\end{equation*}
and $1+C/N_{\delta}\leq 1+\delta$. For such $N_{\delta}$ we find that
\begin{equation*}
	F_{\e}(\w,u_{\e,i_*},D)\leq (1+\delta)F_{\e}(\w,u_{\e,},D)+\delta.
\end{equation*}
Since $r_{i_*}\geq N_{\delta}\geq \delta^{-1}$, it suffices to set $u_{\e,\delta}:=u_{\e,i_*}$ and $C_{\delta}=3^{N_{\delta}+1}N_{\delta}$. The property $|u_{\e,\delta}|\leq |u_{\e}|$ follows from the inequality $|\psi_{r}(x)|\leq |x|$.
\end{proof}

The following proposition establishes the lower bound for the $\Gamma$-convergence statement in Theorem \ref{thm.Gamma_pure}.
\begin{proposition}\label{p.lb}
Let $u_{\e}\in L^1(D,\R^m)$ and $u\in W^{1,1}(D,\R^m)$ be such that $u_{\e}\to u$ in $L^1(D,\R^m)$ as $\e\to 0$. Then
\begin{equation*}
\int_D f_{\rm hom}(\nabla u(x))\dx\leq\liminf_{\e\to 0}F_{\e}(\w,u_{\e},D).	
\end{equation*}
In particular, $u\in W^{1,p}(D,\R^m)$ whenever the right-hand side is finite.
\end{proposition}
\begin{proof}
Without loss of generality, we assume that the limit inferior is finite and, passing to a non-relabeled subsequence, it is actually a limit. By Lemma \ref{l.compactness} we deduce that $u\in W^{1,p}(D,\R^m)$. Define the absolutely continuous Radon-measure $\mu_{\e}$ on $D$ by
\begin{equation*}
\mu_{\e}(B)=\int_B f(\w,\tfrac{x}{\e},\nabla u_{\e}(x))\dx.
\end{equation*}
By our assumption, the sequence of measures $\mu_{\e}$ is equibounded, so that (up to passing to a further subsequence) $\mu_{\e}\overset{\star}{\rightharpoonup}\mu$ for some nonnegative finite Radon measure $\mu$ (possibly depending on $\w$). Using  Lebesgue's decomposition theorem, we can write $\mathrm{d}\mu= \widetilde{f}(x)\dx+ \nu$, with $\nu$ a nonnegative measure $\nu$ that is singular to the Lebesgue measure. Since $D$ is open, the weak$^*$ convergence implies that
\begin{equation*}
\liminf_{\e\to 0}F_{\e}(\w,u_{\e},D)=\liminf_{\e\to 0}\mu_{\e}(D)\geq \mu(D)\geq \int_{D}\widetilde{f}(x)\dx.
\end{equation*} 
Hence it suffices to show that $\widetilde{f}(x_0)\geq f_{\rm hom}(\nabla u(x_0))$ for a.e. $x_0\in D$. For $x\in D$ let $r_x>0$ be such that $Q_r(x)\subset D$ for all $0<r<r_x$. Since $\mu$ is a finite measure, it follows that $\mu(\partial Q_r(x))=0$ except for a countable number of radii $r\in (0,r_x)$. The Besicovitch differentiation theorem \cite[Theorem 1.153]{FoLe} and Portmanteau's theorem imply that for a.e. $x_0\in D$ we have (along a suitable sequence $r\to 0$)
\begin{equation*}
\widetilde{f}(x_0)=\lim_{r\to 0}\frac{\mu(Q_{r}(x_0))}{r^d}= \lim_{r\to 0}\lim_{\e\to 0}\frac{\mu_{\e}(Q_{r}(x_0))}{r^d}.
\end{equation*}
Therefore it suffices to prove that for a.e. $x_0\in D$ we have
\begin{equation}\label{eq:blowup}
\liminf_{r\to 0}\liminf_{\e\to 0}\dashint_{Q_r(x_0)}f_{\e}(\w,\tfrac{x}{\e},\nabla u_{\e}(x))\dx\geq f_{\rm hom}(\nabla u(x_0)).
\end{equation}
%Note that the above property even implies that
%\begin{equation}\label{eq:lb_onopen}
%	\liminf_{\e\to 0} F_{\e}(\w,u_{\e},U)\geq \int_U f_{\rm hom}(\nabla u(x))\dx\quad\text{ for all open sets }U\subset D,
%\end{equation}
%which we will use in Step 2.
In what follows, we let $x_0$ be a Lebesgue point of $u$ and $\nabla u$. To reduce notation, we define the linearization of $u$ at $x_0$ by $L_{u,x_0}(x)=u(x_0)+\nabla u(x_0)(x-x_0)$.

\vspace*{5mm}

\textbf{Step 1:} We prove \eqref{eq:blowup} under the additional assumption that there exists $C_0>0$ such that
\begin{equation}\label{eq:Lipschitz}
\sup_{0<\e<1}\|u_{\e}\|_{L^{\infty}(D)}\leq C_0,
\end{equation}
which also implies that $u\in L^{\infty}(D)$. We modify $u_{\e}$ close to $\partial Q_r(x_0)$: Given $0<\eta<1$ and $N\in\N$ we define for $0\leq i\leq N$ the numbers $\eta_i=1-\eta\tfrac{i}{N}$ and the cubes $Q_{\eta_i r}=Q_{\eta_i r}(x_0)$ (i.e., we drop $x_0$ from the notation). Note that $Q_{(1-\eta)r}\subset Q_{\eta_i r}\subset Q_r$. For $1\leq i\leq N$ we pick a cut-off function $\varphi_{i,\eta}\in C_c^{\infty}(\R^d,[0,1])$ such that $\varphi_{i,\eta}(x)=1$ on $Q_{\eta_i r}$ and ${\rm supp}(\varphi_{i,\eta})\subset Q_{\eta_{i-1}r}$, which can be chosen such that $\|\nabla\varphi_{i,\eta}\|_{\infty}\leq \tfrac{CN}{\eta r}$. Define then the function $u_{\e,\eta_i}:D\to\R^m$ by
\begin{equation*}
u_{\e,\eta_i}(x)=\varphi_{i,\eta}(x)u_{\e}(x)+(1-\varphi_{i,\eta}(x))L_{u,x_0}(x).
\end{equation*}
Since $u_{\e}\in W^{1,1}(D,\R^m)$ due to the global energy bound, it holds that $u_{\e,\eta_i}\in W^{1,1}(D,\R^m)$. By the product rule we have that
\begin{equation*}
\nabla u_{\e,\eta_i}(x)=\nabla\varphi_{i,\eta}(x)\otimes u_{\e}(x)+\varphi_{i,\eta}(x)\nabla u_{\e}(x)-\nabla\varphi_{i,\eta}(x)\otimes L_{u,x_0}(x)+(1-\varphi_{i,\eta}(x))\nabla u(x_0),
\end{equation*}
so that $0\leq \varphi_{i,\eta}\leq 1$ implies the estimate
\begin{align}
|\nabla u_{\e,\eta_i}(x)A(\w,\tfrac{x}{\e})|^p&\leq C\Big(|\nabla\varphi_{i,\eta}(x)|^p |u_{\e}(x)-L_{u,x_0}(x)|^p|A(\w,\tfrac{x}{\e})|^p\nonumber
\\
&\quad\quad+|\nabla u_{\e}(x)A(\w,\tfrac{x}{\e})|^p+|\nabla u(x_0)A(\w,\tfrac{x}{\e})|^p\Big).\label{eq:gradbound}
\end{align}
Since $u_{\e,\eta_i}=u_{\e}$ on $Q_{\eta_i r}$ and $u_{\e,\eta_i}=L_{u,x_0}$ on $\R^d\setminus Q_{\eta_{i-1}r}$, we can estimate the energy of $u_{\e,\eta_i}$ on $Q_r$ by
\begin{equation}\label{eq:split}
\frac{1}{r^d}F_{\e}(\w,u_{\e,\eta_i},Q_r)\leq \frac{1}{r^d}F_{\e}(\w,u_{\e},Q_r)+\frac{1}{r^d}F_{\e}(\w,u_{\e,\eta_i},Q_{\eta_{i-1}r}\setminus Q_{\eta_i r})+\frac{1}{r^d}F_{\e}(\w,L_{u,x_0},Q_{r}\setminus Q_{\eta_{i-1}r}).
\end{equation}
We argue that the last two terms are asymptotically negligible for a suitable choice of $i$. To reduce notation, we set $S_{i,\eta}^r=Q_{\eta_{i-1}r}\setminus Q_{\eta_i r}$. Using the bounds in Assumption~\ref{a.1} and \eqref{eq:gradbound} we have that
\begin{align*}
\frac{1}{r^d}F_{\e}(\w,u_{\e,\eta_i},S^r_{i,\eta})&\leq \frac{1}{r^d}\int_{S^r_{i,\eta}}|\nabla u_{\e,\eta_i}(x)A(\w,\tfrac{x}{\e})|^p+\Lambda(\w,\tfrac{x}{\e})\dx
\\
&\leq \frac{C}{r^d}\int_{S^r_{i,\eta}}(N/(\eta r))^{p} |u_{\e}(x)-L_{u,x_0}(x)|^p|A(\w,\tfrac{x}{\e})|^p+|\nabla u_{\e}(x)A(\w,\tfrac{x}{\e})|^p\dx
\\
&\quad+\frac{C}{r^d}\int_{S^r_{i,\eta}}|\nabla u(x_0)A(\w,\tfrac{x}{\e})|^p+\Lambda(\w,\tfrac{x}{\e})\dx
\\
&\leq \frac{C}{r^d} \left(\int_{S_{i,\eta}^r}(N/(\eta r))^{p} |u_{\e}(x)-L_{u,x_0}(x)|^p|A(\w,\tfrac{x}{\e})|^p\dx+F_{\e}(\w,u_{\e},S_{i,\eta}^r)\right)
\\
&\quad +\frac{C}{r^d}\int_{Q_r\setminus Q_{(1-\eta)r}}|\nabla u(x_0)|^p|A(\w,\tfrac{x}{\e})|^p+\Lambda(\w,\tfrac{x}{\e})\dx.
\end{align*}
To bound the last term in \eqref{eq:split}, we argue similarly and obtain that
\begin{equation*}
\frac{1}{r^d}F_{\e}(\w,L_{u,x_0},Q_{r}\setminus Q_{\eta_{i-1}r})\leq \frac{C}{r^d}\int_{Q_r\setminus Q_{(1-\eta) r}}|\nabla u(x_0)|^p|A(\w,\tfrac{x}{\e})|^p+\Lambda(\w,\tfrac{x}{\e})\dx.
\end{equation*}
Combining the last two estimates and summing \eqref{eq:split} over $1\leq i\leq N$, the fact that the sets $(S^r_{i,\eta})_{i=1}^N$ are pairwise disjoint and all contained in $Q_r\setminus Q_{(1-\eta)r}$ implies that
\begin{align*}
\frac{1}{N}\sum_{i=1}^N\frac{1}{r^d}F_{\e}(\w,u_{\e,\eta_i},Q_r)&\leq \frac{1}{r^d}F_{\e}(\w,u_{\e},Q_r)+\frac{C}{Nr^d} F_{\e}(\w,u_{\e},Q_r\setminus Q_{(1-\eta)r})
\\
&\quad+\frac{CN^{p-1}}{\eta^p r^{d+p}}\int_{Q_r}|u_{\e}(x)-L_{u,x_0}(x)|^p|A(\w,\tfrac{x}{\e})|^p\dx
\\
&\quad+\frac{C}{r^d}\int_{Q_r\setminus Q_{(1-\eta)r}}|\nabla u(x_0)|^p|A(\w,\tfrac{x}{\e})|^p+\Lambda(\w,\tfrac{x}{\e})\dx.
\end{align*}
Taking $i_*=i_{\e,\eta}\in\{1,\ldots,N\}$ such that $F_{\e}(\w,u_{\e,\eta_{i_*}},Q_r)\leq F_{\e}(\w,u_{\e,\eta_i},Q_r)$ for all $1\leq i\leq N$, the energy of the corresponding sequence is bounded by
\begin{align}\label{eq:an_estimate}
\frac{1}{r^d}F_{\e}(\w,u_{\e,\eta_{i_*}},Q_r)&\leq \left(1+\frac{C}{N}\right)\frac{1}{r^d}F_{\e}(\w,u_{\e},Q_r)+\frac{CN^{p-1}}{\eta^p r^{d+p}}\int_{Q_r}|u_{\e}(x)-L_{u,x_0}(x)|^p|A(\w,\tfrac{x}{\e})|^p\dx\nonumber
\\
&\quad + \frac{C}{r^d}\int_{Q_r\setminus Q_{(1-\eta)r}}|\nabla u(x_0)|^p|A(\w,\tfrac{x}{\e})|^p+\Lambda(\w,\tfrac{x}{\e})\dx.
\end{align}
We now pass to the limit in $\e$. Due to the construction it holds that $u_{\e,\eta_{i_*}}-u(x_0)+\nabla u(x_0)x_0\in \nabla u(x_0)x+W^{1,1}_0(Q_r,\R^m)$. Since the energy is invariant under the shift $u\mapsto u+a$ for any fixed $a\in\R^m$, by a change of variables from $Q_r$ to $Q_r/\e$ we conclude that
\begin{equation*}
	\frac{1}{r^d}F_{\e}(\w,u_{\e,\eta_{i_*}},Q_r)\geq \frac{1}{|Q_r/\e|}\mu_{\nabla u(x_0)}(\w,Q_r/\e).
\end{equation*}
For the last integral in \eqref{eq:an_estimate} we can use the ergodic theorem \ref{thm.additiv_ergodic}. The other integral in \eqref{eq:an_estimate} is the nontrivial term in the case of degenerate growth conditions. Since we assume that $u_{\e}$ is bounded in $L^{\infty}(D)$ and converges in $L^1(D)$ to $u$, (up to a subsequence) we can assume that $|u_{\e}(x)-L_{u,x_0}(x)|^p$ converges a.e. to $|u(x)-L_{u,x_0}(x)|^p$ and is uniformly bounded. Moreover, we know from Lemma \ref{l.weakL1} that $|A(\w,\tfrac{\cdot}{\e})|^p$ converges weakly in $L^1(D)$ to $\mathbb{E}[|A(\cdot,0)|^p]$. By \cite[Proposition 2.61]{FoLe} the product thus converges weakly in $L^1(D)$ to the product of the limits. Hence by Lemma \ref{l.existence_f_hom}
\begin{align*}
\frac{N}{N+C}f_{\rm hom}(\nabla u(x_0))&\leq \liminf_{\e\to 0}\frac{1}{r^d}F_{\e}(\w,u_{\e},Q_r)+C\,\mathbb{E}[|A(\cdot,0)|^p]\frac{N^{p-1}}{\eta^p r^{d+p}}\int_{Q_r}|u(x)-L_{u,x_0}(x)|^p\dx
\\
&\quad+C\left(\mathbb{E}[|A(\cdot,0)|^p]|\nabla u(x_0)|^p+\mathbb{E}[\Lambda(\cdot,0)]\right)(1-(1-\eta)^d).
\end{align*}
Our construction allows us to consider the following order of limits: first we let $r\to 0$. Since $u\in W^{1,p}(D,\R^m)$, the $L^p$-differentiability of Sobolev functions (cf. \cite[Theorem 2, p. 230]{EvGa}) yields that
\begin{equation*}
\lim_{r\to 0}\frac{1}{r^p}\dashint_{Q_r}|u(x)-L_{u,x_0}(x)|^p\dx=0.	
\end{equation*}
In a second step we let $\eta\to 0$ and $N\to +\infty$ and conclude that
\begin{equation*}
f_{\rm hom}(\nabla u(x_0))\leq \liminf_{r\to 0}\liminf_{\e\to 0}\frac{1}{r^d}F_{\e}(\w,u_{\e},Q_r),
\end{equation*}
which coincides with \eqref{eq:blowup}.

\vspace*{5mm}

\textbf{Step 2:} Now consider a general sequence $u_{\e}\in W^{1,1}(D,\R^m)$ such that $u_{\e}\to u$ in $L^1(D,\R^m)$ and with equibounded energy $F_{\e}(\w,u_{\e},D)$. Instead of proving \eqref{eq:blowup}, we directly show the lower bound using Lemma \ref{l.truncation}. Given $\delta>0$, let $u_{\e,\delta}\in W^{1,1}(D,\R^m)$ be the function given by Lemma \ref{l.truncation}, so that
\begin{equation*}
\liminf_{\e\to 0}F_{\e}(\w,u_{\e},D)\geq \frac{1}{1+\delta}\liminf_{\e\to 0}F_{\e}(\w,u_{\e,\delta},D)-\delta.
\end{equation*}
In particular, by Lemma \ref{l.compactness} it follows that (up to a subsequence) $u_{\e,\delta}\rightharpoonup u_{\delta}$ in $W^{1,1}(D,\R^m)$ with $u_{\delta}\in W^{1,p}(D,\R^m)$. From Step 1 and the uniform boundedness of $u_{\e,\delta}$ we infer that
\begin{equation*}
	\liminf_{\e\to 0}F_{\e}(\w,u_{\e},D)\geq \frac{1}{1+\delta}\int_Df_{\rm hom}(\nabla u_{\delta}(x))\dx-\delta.
\end{equation*} 
Hence it suffices to show that
\begin{equation}\label{eq:lsc}
	\liminf_{\delta\to 0}\int_D f_{\rm hom}(\nabla u_{\delta}(x))\dx\geq \int_D f_{\rm hom}(\nabla u(x))\dx.
\end{equation}
Since $u_{\e,\delta}=u_{\e}$ on $\{|u_{\e}|\leq\delta^{-1}\}$, it follows that $u_{\delta}=u$ a.e. on $\{|u|\leq \delta^{-1}\}$. Moreover, it is a consequence of Lemma \ref{l.truncation} that $|u_{\delta}(x)|\leq |u(x)|$ a.e. on $D$. Dominated convergence thus implies that $u_{\delta}\to u$ in $L^1(D,\R^m)$, while $\nabla u_{\delta}$ is bounded in $L^p(D,\R^{m\times d})$ due to the $p$-growth from below of $f_{\rm hom}$ (cf. Lemma \ref{l.existence_f_hom}). Hence $u_{\delta}\rightharpoonup u$ in $W^{1,p}(D,\R^m)$. Finally, note that by what we have proved we already know that $u_{\delta}=u$ for $\delta$ small enough if $u\in L^{\infty}(D,\R^m)$. Hence, we have identified the $\Gamma$-limit on $W^{1,p}(D,\R^m)\cap L^{\infty}(D,\R^m)$. Since $f_{\rm hom}$ has $p$-growth and the $\Gamma$-limit has to be lower semicontinuous on $L^1(D)$, it follows by standard results (\cite[Theorem 8.4]{Da}) that $f_{\rm hom}$ is quasiconvex. Hence the integral functional 
\begin{equation*}
	u\mapsto \int_D f_{\rm hom}(\nabla u(x))\dx
\end{equation*}
is lower-semicontinuous with respect to weak convergence in $W^{1,p}(D,\R^m)$ (see \cite[Theorem 8.11]{Da}) and therefore we obtain \eqref{eq:lsc} which concludes the proof.
\end{proof}
\begin{proof}[Proof of Theorem \ref{thm.Gamma_pure}]
Lemma \ref{l.compactness} shows that the domain of the $\Gamma$-limit is $W^{1,p}(D,\R^m)$. Propositions \ref{p.ub} and \ref{p.lb} yield the $\Gamma$-convergence statement, while the properties of $f_{\rm hom}$ are proven in Lemma~\ref{l.existence_f_hom}. 	
\end{proof}
\subsection{Proof with boundary data}
In this section we fix a boundary condition $g\in W_{\rm loc}^{1,\infty}(\R^d,\R^m)$ and prove a compactness statement and the corresponding lower and upper bounds when the functionals are restricted to maps $u_{\e}$ satisfying the boundary condition $u_{\e}=g$ on $\partial D$ in the sense of traces. The case of external forces is postponed to the next section.
\begin{lemma}\label{l.boundarycompactness}
Let $u_{\e}\in g+W^{1,1}_0(D,\R^m)$ be such that
\begin{equation*}
	\sup_{\e\in (0,1)}F_{\e}(\w,u_{\e},D)<+\infty.
\end{equation*}
Then there exists $u\in g+W_0^{1,p}(D,\R^m)$ such that up to a subsequence $u_{\e}\rightharpoonup u$ in $W^{1,1}(D,\R^m)$ and $u_{\e}\to u$ in $L^{\frac{d}{d-1}}(D,\R^m)$.
\end{lemma}
\begin{proof}
Due to Lemma \ref{l.compactness} and Theorem \ref{thm.embedding}, the convergence follows once we show that $u_{\e}$ is bounded in $L^1(D,\R^m)$. This is a consequence of Poincar\'e's inequality since by Lemma \ref{l.compactness} $\nabla u_{\e}$ is bounded in $L^1(D,\R^{m\times d})$. It remains to prove that $u=g$ on $\partial D$. Since $u_{\e}\rightharpoonup u$ in $W^{1,1}(D,\R^m)$, the trace of $u_{\e}$ converges weakly in $L^1(\partial\Omega)$ to the trace of $u$. Hence $u=g$ on $\partial D$ in the sense of traces and therefore $u\in g+W_0^{1,p}(D,\R^m)$.
\end{proof}
Next we prove the lower and upper bound with active boundary conditions.
\begin{proposition}\label{p.lb_contrained}
Let $u_{\e}\in g+W_0^{1,1}(D,\R^m)$ be such that $u_{\e}\to u$ in $L^1(D,\R^m)$ and moreover $\liminf_{\e\to 0}F_{\e}(\w,u_{\e},D)<+\infty$. Then $u\in g+W_0^{1,p}(D,\R^m)$ and
\begin{equation*}
	\int_D f_{\rm hom}(\nabla u(x))\dx\leq \liminf_{\e\to 0}F_{\e}(\w,u_{\e},D).
\end{equation*} 
\end{proposition}
\begin{proof}
By Lemma \ref{l.compactness} we know that $u_{\e}\rightharpoonup u$ in $W^{1,1}(D,\R^m)$ and $u\in g+W_0^{1,p}(D,\R^m)$. The lower bound is a consequence of Proposition \ref{p.lb}.
\end{proof}
\begin{proposition}\label{p.ub_constrained}
Let $u\in g+W_0^{1,p}(D,\R^m)$. Then there exists $u_{\e}\in g+W_0^{1,1}(D,\R^m)$ such that $u_{\e}\to u$ in $L^1(D,\R^m)$ and
\begin{equation*}
\limsup_{\e\to 0}F_{\e}(\w,u_{\e},D)\leq\int_D f_{\rm hom}(\nabla u(x))\dx.
\end{equation*}
\end{proposition}
\begin{proof}
By density it is enough to prove the claim when $u-g\in C_c^{\infty}(D,\R^m)$. Due to Proposition \ref{p.ub} we find a sequence $u_{\e}\in W^{1,1}(D,\R^m)$ such that $u_{\e}\to u$ in $L^1(D,\R^m)$ and
\begin{equation}\label{eq:recoverysequence}
\limsup_{\e\to 0}F_{\e}(\w,u_{\e},D)\leq \int_D f_{\rm hom}(\nabla u(x))\dx.	
\end{equation}
We modify $u_{\e}$ near $\partial D$ such that it belongs to $ g+W_0^{1,1}(D,\R^m)$ with a negligible increase in energy. This will be achieved by certain convex combinations as in the proof of Proposition \ref{p.lb}. Hence we need again to truncate the sequence $u_{\e}$. Using Lemma \ref{l.truncation}, for every $\delta>0$ there exists $C_{\delta}>0$ and a sequence $u_{\e,\delta}\in W^{1,1}(D,\R^m)$ with the following properties:
\begin{align}
 &\|u_{\e,\delta}\|_{L^{\infty}(D)}\leq C_{\delta},\label{eq:bound}
\\
&u_{\e,\delta}=u_{\e}\quad \text{ a.e. on }\{|u_{\e}|\leq \delta^{-1}\},\label{eq:equal}
\\
&|u_{\e,\delta}(x)|\leq |u_{\e}(x)|,\label{eq:smaller}
\\
&\limsup_{\e\to 0}F_{\e}(\w,u_{\e,\delta},D)\leq(1+\delta)\limsup_{\e}F_{\e}(\w,u_{\e},D)+\delta.\label{eq:energy}
\end{align}
Choosing $\delta$ sufficiently small, we can additionally assume that 
\begin{equation}\label{eq:dominateg}
	\|u\|_{L^{\infty}(D)}+\|g\|_{L^{\infty}(D)}< \delta^{-1}.
\end{equation}
Next, let $\eta>0$ be such that $\{x\in D:\dist(x,\partial D)\leq \eta\}\subset \{u=g\}$. Given $N\in\N$  and $0\leq i\leq N$ we define the sets
\begin{equation*}
D_i=\left\{x\in D:\,\dist(x,\partial D)>\frac{i}{N}\eta\right\}
\end{equation*}
and for $1\leq i\leq N$ we choose a cut-off function $\varphi_{i}\in C_c^{\infty}(\R^d,[0,1])$ such that $\varphi_i\equiv 1$ on $D_i$, ${\rm supp}(\varphi_i)\subset D_{i-1}$, and $\|\nabla\varphi_i\|_{L^{\infty}(\R^d)}\leq \tfrac{CN}{\eta}$. We then define the interpolation between $u_{\e,\delta}$ and $g$ by
\begin{equation}\label{eq:convexcombo}
u_{\e,\delta,i}=\varphi_i u_{\e,\delta}+(1-\varphi_i)g\in g+W_0^{1,1}(D,\R^m).
\end{equation}
In order to estimate its energy we can argue as for \eqref{eq:split} and the subsequent estimates to obtain
\begin{align*}
F_{\e}(\w,u_{\e,\delta,i},D)&\leq F_{\e}(\w,u_{\e,\delta},D)+F_{\e}(\w,u_{\e,\delta,i},D_{i-1}\setminus D_i)+F_{\e}(\w,g,D\setminus D_{i-1})
\\
&\leq F_{\e}(\w,u_{\e,\delta},D)+C F_{\e}(\w,u_{\e,\delta},D_{i-1}\setminus D_i)+(1+C)F_{\e}(\w,g,D\setminus D_N)
\\
&\quad +C\int_{D_{i-1}\setminus D_i}(N/\eta)^p |u_{\e,\delta}(x)-g(x)|^p|A(\w,\tfrac{x}{\e})|^p+\Lambda(\w,\tfrac{x}{\e})\dx.
\end{align*}
The sets $\{D_{i-1}\setminus D_i\}_{i=1}^N$ are pairwise disjoint and contained in $D\setminus D_N$. Hence, similar to \eqref{eq:an_estimate}, for every $\e>0$ there exists $1\leq i_{\e}\leq N$ such that, setting $u^*_{\e,\delta}=u_{\e,\delta,i_{\e}}$, it holds that
\begin{align}\label{eq:averaging_bc}
F_{\e}(\w,u_{\e,\delta}^*,D)&\leq \left(1+\frac{C}{N}\right)F_{\e}(\w,u_{\e,\delta},D)+(1+C)F_{\e}(\w,g,D\setminus D_N)\nonumber
\\
&\quad+\frac{C}{N}\int_{D\setminus D_N}(N/\eta)^p |u_{\e,\delta}(x)-g(x)|^p|A(\w,\tfrac{x}{\e})|^p+\Lambda(\w,\tfrac{x}{\e})\dx.
\end{align}
To pass to the limit in $\e$, we need to bound the last two terms. Since $g\in W^{1,\infty}(D,\R^m)$, we conclude that
\begin{equation*}
F_{\e}(\w,g,D\setminus D_N)\leq (\|\nabla g\|_{L^{\infty}(D)}+1)\int_{D\setminus D_N}|A(\w,\tfrac{x}{\e})|^p+\Lambda(\w,\tfrac{x}{\e})\dx.
\end{equation*}
The definition of $D_N$ yields that $D\setminus D_N\subset \{x\in D:\,\dist(x,\partial D)\leq \eta\}$. The measure of the latter set vanishes as $\eta\to 0$. Lemma \ref{l.weakL1} on the $L^1$-weak convergence in the ergodic theorem then implies that for $\eta=\eta(\delta)$ small enough, it holds that
\begin{equation}\label{eq:control_g}
\limsup_{\e\to 0}(1+C)F_{\e}(\w,g,D\setminus D_N)+\frac{C}{N}\int_{D\setminus D_N}\Lambda(\w,\tfrac{x}{\e})\dx\leq\delta.	
\end{equation}
In order to bound the last term in \eqref{eq:averaging_bc}, we note that due to \eqref{eq:bound} the function $x\mapsto |u_{\e,\delta}(x)-g(x)|^p$ is uniformly bounded as $\e\to 0$. Moreover, for a.e. $x\in D\setminus D_N$, due to \eqref{eq:dominateg} we have along a subsequence
\begin{equation*}
	g(x)=u(x)=\lim_{\e_j\to 0}u_{\e_j}(x)\overset{\eqref{eq:equal}}{=}\lim_{\e_j\to 0}u_{\e_j,\delta}(x),
\end{equation*} 
so that we can apply \cite[Proposition 2.61]{FoLe} and deduce that
\begin{equation*}
	\lim_{\e_j\to 0}\int_{D\setminus D_N}|u_{\e_j,\delta}(x)-g(x)|^p|A(\w,\tfrac{x}{\e_j})|^p\dx=0.
\end{equation*}
The limit is independent of the subsequence, so this convergence is valid along any sequence $\e\to 0$. Combined with \eqref{eq:energy}, \eqref{eq:recoverysequence}, and \eqref{eq:control_g}, the estimate \eqref{eq:averaging_bc} with $C/N\leq\delta$ yields
\begin{equation}\label{eq:BC_limsup}
\limsup_{\e\to 0} F_{\e}(\w,u_{\e,\delta}^*,D)\leq (1+\delta)^2\int_D f_{\rm hom}(\nabla u(x))\dx+2\delta.
\end{equation}
We finally argue that $u_{\e,\delta}^*\to u$ in $L^1(D,\R^m)$. Due to \eqref{eq:dominateg} and \eqref{eq:equal}, we know that along a subsequence, for a.e. $x\in D$, it holds that
\begin{equation*}
u(x)=\lim_{\e_j\to 0}u_{\e_j}(x)=\lim_{\e_j\to 0}u_{\e_j,\delta}(x).
\end{equation*}
Since 
\begin{equation*}
u_{\e_j,\delta}^*(x)=\varphi_{i_{\e}}(x)u_{\e_j,\delta}(x)+(1-\varphi_{i_{\e}}(x))g(x)=
\begin{cases} 
	u_{\e_j,\delta}(x)&\mbox{on $D_N$},
	\\
	\varphi_{i_{\e}}(x)u_{\e_j,\delta}(x)+(1-\varphi_{i_{\e}}(x))u(x) &\mbox{on $D\setminus D_{N}$,}
\end{cases}
\end{equation*}
it follows that along the same subsequence we have $u_{\e_j,\delta}^*(x)\to u(x)$. Due to \eqref{eq:smaller} we can apply Lebesgue's dominated convergence theorem to deduce that (now along the whole sequence) $u_{\e,\delta}^*\to u$ in $L^1(D,\R^m)$. Since $\delta>0$ was arbitrary, a diagonal argument in \eqref{eq:BC_limsup} proves the claim.
\end{proof}
\subsection{Proof with boundary data and external forces}\label{s.bd+forces}
We add the additional force term to the energy. As $\Gamma$-convergence is stable under continuously converging perturbations, the proof comes almost for free. 
\begin{proof}[Proof of Theorem \ref{thm:Dirichlet_and_forces}]
We first show the compactness statement. Repeating the estimate \eqref{eq:energy_estimate}, we have that
\begin{equation*}
	\left(|D|\dashint_{D/\e}|A(\w,y)^{-1}|^{\frac{p}{p-1}}\dy\right)^{1-p}\left(\int_{D}|\nabla u_{\e}(x)|\dx\right)^p\leq C F_{\e}(\w,u_{\e},D).
\end{equation*}
Applying the ergodic theorem \ref{thm.additiv_ergodic} to the first integral on the left-hand side, we deduce that there exists a constant $C>0$ such that for $\e>0$ small enough we have
\begin{equation}\label{eq:linearlb}
\frac{1}{C}\,\|\nabla u_{\e}\|_{L^1(D)}^p\leq F_{\e}(\w,u_{\e},D).
\end{equation}
Set $q=\frac{p}{p-1}$ as the dual exponent of $p$. From H\"older's inequality, the Sobolev embedding, Young's inequality, and Poincar\'e's inequality in $W^{1,1}_0(D,\R^m)$, we deduce that for any $\delta>0$ it holds that
\begin{align*}
\left|\int_D f_{\e}(x)\cdot u_{\e}(x)\dx\right|&\leq \|f_{\e}\|_{L^{d}(D)}\|u_{\e}\|_{L^{d/(d-1)}(D)}\leq C \|f_{\e}\|_{L^d(D)}\|u_{\e}\|_{W^{1,1}(D)}
\\
&\leq \frac{C\delta^{-q}}{ q}\|f_{\e}\|^{q}_{L^d(D)}+\frac{C\delta^p}{p}\|u_{\e}\|^p_{W^{1,1}(D)}
\\
&\leq \frac{C\delta^{-q}}{q}\|f_{\e}\|^q_{L^d(D)}+\frac{C\delta^p}{p}(\|\nabla u_{\e}\|^p_{L^1(D)}+\|g\|^p_{W^{1,1}(D)}).
\end{align*}
From \eqref{eq:linearlb} we conclude that for $\delta$ small enough 
\begin{equation*}
\left|\int_D f_{\e}(x)\cdot u_{\e}(x)\dx\right|\leq C_{\delta,p}\|f_{\e}\|_{L^d(D)}^q+\frac{1}{2}F_{\e}(\w,u_{\e},D)+\|g\|_{W^{1,1}(D)}^p.
\end{equation*}
In particular, since $f_{\e}$ is bounded in $L^d(D,\R^m)$, a bound of the form
\begin{equation*}
	\limsup_{\e\to 0} \left(F_{\e}(\w,u_{\e},D)-\int_D f_{\e}(x)\cdot u_{\e}(x)\dx\right)<+\infty
\end{equation*}
implies that $F_{\e}(\w,u_{\e},D)$ is bounded as $\e\to 0$ and therefore the compactness statement follows from Lemma \ref{l.boundarycompactness}. Since $u_{\e}\to u$ in $L^{d/(d-1)}(D,\R^m)$ implies that
\begin{equation*}
	\lim_{\e\to 0}\int_D f_{\e}(x)\cdot u_{\e}(x)\dx=\int_D f_0(x)\cdot u(x)\dx,
\end{equation*}
the $\Gamma$-convergence follows from the result without external forces (see Propositions \ref{p.lb_contrained} and \ref{p.ub_constrained}).
\end{proof}

\subsection{Proof for the obstacle problem}
The last constraint we treat is the inequality $u\geq \varphi_{\e}$.
\begin{proof}[Proof of Theorem \ref{thm.obstacle}]
To reduce notation, we just consider the scalar case $m=1$, but the same arguments can be applied for every component. Moreover, since $g$ enters the problem only with its values on $\partial D$, we can replace it by another $W^{1,\infty}(D)$-function that has the same trace. We construct such a function $\tilde{g}$ satisfying $\tilde{g}\geq\varphi_{\e}$ on $\overline{D}$ for all $\e$ small enough, which turns out to be convenient for the proof. Note that such $\tilde{g}$ also satisfies $\tilde{g}\geq\varphi$ in $\overline{D}$ since $\varphi_{\e}\to\varphi$ uniformly on $\overline{D}$. To construct $\widetilde{g}$, we fix a large constant $c_g>0$ and set $\phi(x)=\min\{1,\dist(x,\partial D)\}$ and
\begin{equation*}
\tilde{g}(x)=c_g\phi(x)+(1-\phi(x))g(x).
\end{equation*}
Then $\tilde{g}\in W^{1,\infty}(D)$ due to the Lipschitz continuity of the distance function. Moreover, for $x\in\overline{D}$ let $x_{p}$ be any point such that $|x-x_{p}|=\dist(x,\partial D)$. Since $g\geq\varphi_{\e}$ on $\partial D$ by assumption, we obtain 
\begin{align*}
\tilde{g}(x)-\varphi_{\e}(x)&\geq \phi(x)(c_g-g(x))+g(x)-g(x_{p})+\varphi_{\e}(x_{p})-\varphi_{\e}(x)
\\
&\geq  \phi(x)(c_g-g(x))-C|x-x_{p}|
\geq\begin{cases}
	c_g-g(x)-C\,{\rm diam}(D) &\mbox{if $\dist(x,\partial D)\geq 1$,}
	\\
	|x-x_{p}|(c_g-g(x)-C) &\mbox{otherwise.}
\end{cases},
\end{align*}
where in the penultimate estimate we used that $g+\varphi_{\e}$ is bounded in $W^{1,\infty}(D)$ as $\e\to 0$, so that by the Lipschitz regularity of $\partial D$ the sequence $g+\varphi_{\e}$ is equi-Lipschitz on $\overline{D}$ as $\e\to 0$. For $c_g$ large enough the right-hand side terms in the above estimate are nonnegative. Finally, it holds that $\tilde{g}=g$ on $\partial D$, so that from now on we may assume that for $\e>0$ small enough \begin{equation}\label{eq:boundaryconditionprepared}
g\geq \max\{\varphi,\varphi_{\e}\}\quad\text{ on }\overline{D}.	
\end{equation}
We now come to the actual proof. Note that when $u_{\e}\to u$ and $\varphi_{\e}\to\varphi$ in $L^1(D)$, the condition $u_{\e}\geq \varphi_{\e}$ a.e. implies that $u\geq\varphi$ a.e. Since $F_{\e,f_{\e},g}^{\varphi_{\e}}(\w,u,D)\geq F_{\e,f_{\e},g}(\w,u,D)$ for all $u\in L^1(D,\R^m)$, the compactness property and the $\Gamma$-liminf inequality follow from Lemma \ref{l.boundarycompactness} and Proposition \ref{p.lb_contrained}, respectively. Therefore it suffices to show the $\Gamma$-limsup inequality.	Without loss of generality we set $f_{\e}=0$ since the linear term is a continuously converging perturbation with respect to weak convergence in $W^{1,1}(D)$. 

For the moment, we ignore the boundary condition $g$ and consider a general map $u\in W^{1,\infty}(D)$ such that $u\geq \varphi$. For $\eta>0$ consider the function $u_{\eta}:=u+\eta\in W^{1,\infty}(D)$, which satisfies 
\begin{equation}\label{eq:eta_barrier}
	u_{\eta}\geq \varphi+\eta.
\end{equation}
Fix $\delta>0$. By Lemma \ref{l.truncation} we find a sequence $u_{\e,\delta,\eta}\in W^{1,1}(D)$ such that $\|u_{\e,\delta,\eta}\|_{\infty}\leq C_{\delta}$, $u_{\e,\delta,\eta}\to u_{\eta}$ in $L^1(D)$ as $\e\to 0$ and
\begin{equation*}
\limsup_{\e\to 0}F_{\e}(\w,u_{\e,\delta,\eta},D)\leq(1+\delta)\int_D f_{\rm hom}(\nabla u_{\eta}(x))\dx+\delta=(1+\delta)\int_D f_{\rm hom}(\nabla u(x))\dx+\delta.
\end{equation*}
In order to satisfy the constraint, we introduce $v_{\e,\delta,\eta}=\max\{u_{\e,\delta,\eta},\varphi_{\e}\}\in W^{1,1}(D)$. Writing $v_{\e,\delta,\eta}=\max\{\varphi_{\e}-u_{\e,\delta,\eta},0\}+u_{\e,\delta,\eta}$, we see that $v_{\e,\delta,\eta}\to u_{\eta}$ in $L^1(D)$ and from the chain rule it follows that
\begin{equation*}
\nabla v_{\e,\delta,\eta}=\nabla(\varphi_{\e}-u_{\e,\delta,\eta})\chi_{\{\varphi_{\e}>u_{\e,\delta,\eta}\}}+\nabla u_{\e,\delta,\eta}=\nabla\varphi_{\e}\chi_{\{\varphi_{\e}>u_{\e,\delta,\eta}\}}+\nabla u_{\e,\delta,\eta}\chi_{\{\varphi_{\e}\leq u_{\e,\delta,\eta}\}}\quad\text{ a.e. in }D.
\end{equation*} 
Using the upper bound in Assumption \ref{a.1} and the nonnegativity of the integrand $f$, we deduce that
\begin{align*}
F_{\e}(\w,v_{\e,\delta,\eta},D)&\leq F_{\e}(\w,u_{\e,\delta,\eta},D)+\int_{\{\varphi_{\e}> u_{\e,\delta,\eta}\}} |\nabla \varphi_{\e}(x)|^p|A(\w,\tfrac{x}{\e})|^p+\Lambda(\w,\tfrac{x}{\e})\dx
\\
&\leq F_{\e}(\w,u_{\e,\delta,\eta},D)+(\|\nabla\varphi_{\e}\|_{\infty}+1)\int_{\{\varphi_{\e}> u_{\e,\delta,\eta}\}}|A(\w,\tfrac{x}{\e})|^p+\Lambda(\w,\tfrac{x}{\e})\dx.
\end{align*}
We claim that the last term vanishes as $\e\to 0$. Since $\varphi_{\e}$ is bounded in $W^{1,\infty}(D)$ and $x\mapsto |A(\w,\tfrac{x}{\e})|^p+\Lambda(\w,\tfrac{x}{\e})$ is equi-integrable by Lemma \ref{l.weakL1}, it suffices to show that $|\{\varphi_{\e}>u_{\e,\delta,\eta}\}|\to 0$ as $\e\to 0$. Since $\varphi_{\e}\to\varphi$ and $u_{\e,\delta,\eta}\to u_{\eta}$ in $L^1(D)$, this is a consequence of \eqref{eq:eta_barrier}. Summing up, we have constructed a sequence $v_{\e,\delta,\eta}\in W^{1,1}(D)$ such that $v_{\e,\delta,\eta}\to u_{\eta}$ in $L^1(D)$, $\|v_{\e,\delta,\eta}\|\leq C'_\delta$, $v_{\e,\delta,\eta}\geq\varphi_{\e}$ in $D$ and
\begin{equation*}
	\limsup_{\e\to 0}F_{\e}(\w,v_{\e,\delta,\eta},D)\leq (1+\delta)\int_D f_{\rm hom}(\nabla u(x))\dx+\delta.
\end{equation*}
Since $u_{\eta}\to u$ in $L^1(D)$ as $\eta\to 0$, using a diagonal argument we find a sequence ${u}_{\e,\delta}\in W^{1,1}(D)$ such that $u_{\e,\delta}\to u$ in $L^1(D)$ as $\e\to 0$, $\|u_{\e,\delta}\|_{\infty}\leq C'_\delta$, $u_{\e,\delta}\geq\varphi_{\e}$ in $D$ and
\begin{equation*}
	\limsup_{\e\to 0}F_{\e}(\w,\tilde{u}_{\e,\delta},D)\leq (1+\delta)\int_D f_{\rm hom}(\nabla u(x))\dx+\delta.
\end{equation*}

Next we include the boundary condition $g$. Fix $u\in g+W^{1,p}_0(D)$ such that $ u\geq\varphi$ a.e. In order to repeat the argument for Proposition \ref{p.ub_constrained} we need to reduce the analysis to the case that $u-g$ is has compact support in $D$ and that $u\in L^{\infty}(D)$. To this end, consider a sequence $u_n\in g+C_c^{\infty}(D)$ such that $u_n\to u$ in $W^{1,p}(D)$. In general, this sequence does not satisfy the constraint $u_n\geq\varphi$ a.e. Hence, we consider the modified sequence $v_n=\max\{u_n,\varphi\}\in W^{1,\infty}(D)$. Since $g\geq\varphi$ on $\overline{D}$, it holds that $v_n-g\in W^{1,\infty}_c(D)$ and $v_n\geq\varphi$ a.e. Moreover, writing $v_n=\max\{\varphi-u_n,0\}+u_n$, the Lipschitz continuity of the map $x\mapsto \max\{x,0\}$ and the convergence $u_n\to u$ in $W^{1,p}(D)$ imply that $v_n\to \max\{\varphi-u,0\}+u=u$ in $W^{1,p}(D)$ (cf. \cite[Theorem 1]{MM}). As a consequence, it suffices to show the upper bound for functions $u\in g+W_c^{1,\infty}(D)$ such that $u\geq\varphi$. From the above argument we know that for every $\delta>0$ there exists a sequence $u_{\e,\delta}\in W^{1,1}(D)$ with $u_{\e,\delta}\to u$ in $L^1(D)$ as $\e\to 0$, $\|u_{\e,\delta}\|_{\infty}\leq C'_\delta$, $u_{\e,\delta}\geq \varphi_{\e}$ in $D$ and
\begin{equation*}
	\limsup_{\e\to 0}F_{\e}(\w,u_{\e},D)\leq(1+\delta)\int_D f_{\rm hom}(\nabla u(x))\dx+\delta.
\end{equation*}
We modify the sequence $u_{\e,\delta}$ in the same manner as in the proof of Proposition \ref{p.ub_constrained}. The estimates are analogous, but we have to ensure that the constraint is preserved. To this end, recall that due to the boundedness in $L^{\infty}(D)$, the only modification is the adjustment of the boundary condition via a convex combination of $u_{\e,\delta}$ and $g$ (cf. \eqref{eq:convexcombo}). Due to \eqref{eq:boundaryconditionprepared} this construction still dominates $\varphi_{\e}$. The remaining part of the proof is unchanged and we conclude by a diagonal argument with respect to $\delta$.
\end{proof}

\subsection{Stochastic homogenization of the Euler-Lagrange equations}
In this section (and only here) we add Assumption \ref{a.2} to the setting. In order to prove the differentiability and strict convexity of the homogenized integrand, we first derive a non-asymptotic formula for $f_{\rm hom}$ that is well-known in the non-degenerate setting (see, for instance, \cite[Chapter 15]{JKO} or \cite[Lemma 3.7]{DG_unbounded}). 
	
Define the set 
\begin{align}\label{eq:def_F_pot}
		F_{\rm pot}^1:=\{h\in L^1(\Omega,\R^d):\,&\mathbb{E}[h]=0\text{ and for a.e. }\w\in\Omega \text{ the function }x\mapsto h(\tau_x\w)\in L^1_{\rm loc}(\R^d,\R^d)\text{ satisfies }\nonumber
		\\
		&\;\partial_i h_j-\partial_j h_i=0 \text{ on }\R^d
		\text{ in the sense of distributions for all }1\leq i,j\leq d\}. 
\end{align}
Even though $d\neq 3$ in general, we refer to the property $\partial_i h_j-\partial_j h_i=0$ as being curl-free. The following lemma holds.
	\begin{lemma}\label{l.F_pot}
		The space $F_{\rm pot}$ is a closed subspace of $L^1(\Omega,\R^d)$. Moreover, given $h\in F_{\rm pot}^1$, there exists a map $\varphi:\Omega\to W^{1,1}_{\rm loc}(\R^d)$ such that $\nabla\varphi(\w,x)=h(\tau_x\w)$ almost surely as maps in $L^1_{\rm loc}(\R^d,\R^d)$ and such that for every bounded set $B\subset\R^d$ the maps $\w\mapsto \varphi(\w,\cdot)$ and $\w\mapsto \nabla\varphi(\w,\cdot)$ are measurable from $\Omega$ to $L^1(B)$ and to $L^1(B,\R^d)$, respectively.
	\end{lemma} 
	\begin{proof}
		$F_{\rm pot}^1$ is a linear subspace of $L^1(\Omega,\R^d)$. To show that it is closed, consider a sequence $h_n\in F_{\rm pot}^1$ such that $h_n\to h$ in $L^1(\Omega,\R^d)$. Then $\mathbb{E}[h]=0$ and, as shown on \cite[p. 224]{JKO}, the convergence implies that (up to a subsequence) it holds that $x\mapsto h_n(\tau_x\w)\to x\mapsto h(\tau_x\w)$ in $L^1_{\rm loc}(\R^d,\R^d)$ for almost every $\w\in\Omega$. Hence it follows that $h\in F_{\rm pot}^1$. Next, we argue that for almost every $\w\in\Omega$ there exists $\varphi(\w,\cdot)\in W^{1,1}_{\rm loc}(\R^d)$ such that $\nabla\varphi(\w,x)=h(\tau_x\w)$ as elements in $L^1_{\rm loc}(\R^d,\R^d)$. Since $\Omega$ is complete, we can assume without loss of generality that $\partial_ih_j= \partial_jh_i$ for all $\w\in\Omega$. To reduce notation, we temporarily suppress the dependence on $\w$ and just write $h=h(x)$. Given $\eta>0$ we consider the regularization $h_{\eta}=h*\theta_{\eta}$, where $\theta_{\eta}\in C_c^{\infty}(\R^d)$ is a family of standard mollifiers. Then $h_{\eta}\in C^{\infty}(\R^d,\R^d)$ and due to Fubini's theorem it follows that in a distributional (and hence classical) sense $\partial_i h_{\eta,j}-\partial_j h_{\eta,i}=0$ for all $1\leq i,j\leq d$. By the classical Poincaré lemma on simply connected domains there exists a function $\varphi_{\eta}\in C^{\infty}(\R^d)$ such that $\nabla\varphi_{\eta}=h_{\eta}$ and $\dashint_{B_1(0)}\phi_{\eta}\dx=0$. Fix now any ball $B'$ centered at the origin and containing $B_1(0)$. Then we have a Poincaré inequality of the form
		\begin{equation*}
			\int_{B'} |u-\dashint_{B_1(0)}u(y)\dy|\dx\leq C(B')\int_{B'}|\nabla u|\dx\quad\text{ for all }u\in W^{1,1}(B').
		\end{equation*} 
		By well-known properties of convolution, we have that $\nabla\varphi_{\eta}\to h$ in $L^1(B',\R^d)$ and by the Sobolev embedding also $\varphi_{\eta}\to \varphi$ for some $\varphi\in W^{1,1}(B')$ with $\nabla \varphi=h$ in $B'$. Since $B'$ was arbitrary, we conclude that $\varphi\in W^{1,1}_{\rm loc}(\R^d)$ with $\nabla \varphi=h$. Moreover, it follows that for any bounded set $B\subset\R^d$ we have that $\varphi_{\eta}\to \varphi$ in $W^{1,1}(B)$. Hence the measurability properties of the map $\w\mapsto \varphi(\w,\cdot)$ and its gradient follow once we prove them for the approximating map $\varphi_{\eta}(\w,\cdot)$ and its gradient. The construction by convolution yields an explicit formula for $\varphi_{\eta}$ which reads
		\begin{equation*}
			\varphi_{\eta}(\w,x)=\int_0^1 h_{\eta}(\w,tx)\cdot x	\,\mathrm{d}t=\int_0^1\int_{\R^d}h(\tau_y\w)\theta_{\eta}(tx-y)\cdot x\dy\,\mathrm{d}t.
		\end{equation*} 
		Since we assume that $(\w,x)\mapsto\tau_{x}\w$ is jointly measurable, it follows from Fubini's theorem that $h_{\eta}$ is measurable in $\w$, and by smoothness also continuous in its second variable. Hence $h_{\eta}$ is jointly measurable and again by Fubini's theorem and regularity in the second variable, we deduce that $\varphi_{\eta}$ is jointly measurable. By construction, $\nabla\varphi_{\eta}=h_{\eta}$ is also jointly measurable. Then by \cite[Lemma 16 b), p. 196]{DS} the maps $\w\mapsto \varphi_{\eta}(\w,\cdot)$ and $\w\mapsto\nabla\varphi_{\eta}(\w,\cdot)$ are measurable with values in $L^1(B)$ and $L^1(B,\R^d)$, respectively. This concludes the proof. 
\end{proof}
We shall also need a suitable estimate for the local Lipschitz constant of $\xi\mapsto f(\w,x,\xi)$ that we prove in the next lemma. We show a slightly more general statement to include the claim in Remark \ref{r.weaksol} b).
\begin{lemma}\label{l.f_differentiable}
In addition to Assumption \ref{a.1}, assume that the map $\xi\mapsto f(\w,x,\xi)$ is separately convex. Then there exists a constant $C=C(p,d+m)$ such that for all $y:=(\w,x)\in\Omega\times\R^d$ and all $\xi_0,\xi_1\in\R^{m\times d}$
\begin{equation}\label{eq:f_quantitativeLipschitz}
|f(y,\xi_1)-f(y,\xi_0)|\leq C\left((\Lambda(y)^{\frac{1}{p}}+|\xi_0A(y)|+|\xi_0A(y)|)^{p-1}+\Lambda(y)^{\frac{p-1}{p}}\right)|(\xi_1-\xi_0)A(y)|.
\end{equation}
In particular, if $f(y,\cdot)$ is also differentiable in $\xi$, then for all $z\in\R^{m\times d}$ we have
\begin{equation}\label{eq:f_quantitativedifferentiable}
|\langle \partial_{\xi}f(y,\xi),z\rangle|\leq C\left((\Lambda(y)^{\frac{1}{p}}+|A(y)\xi|)^{p-1}+\Lambda(y)^{\frac{p-1}{p}}\right)|zA(y)|.
\end{equation}	
\end{lemma} 
\begin{proof}
Since $y=(\w,x)$ is fixed in the proof, we drop it from the notation. By \cite[Theorem 4.36]{FoLe}, a separately convex function $\widetilde{f}$ is Lipschitz-continuous on any ball $B_r(\xi)\subset\R^{m\times d}$ with Lipschitz constant bounded by
\begin{equation*}
	\sqrt{m+d}\,\frac{{\rm osc}(\widetilde{f};B_{2r}(\xi))}{r}:=\sqrt{m+d}\,\sup_{\xi',\xi''\in B_{2r}(\xi)}\frac{|\widetilde{f}(\xi'')-\widetilde{f}(\xi')|}{r}.
\end{equation*}
Consider the function $\xi\mapsto \widetilde{f}(\xi):=f(\xi A^{-1})$, which according to Assumption \ref{a.1} satisfies
\begin{equation*}
	|\widetilde{f}(\xi)|\leq |\xi|^p+\Lambda.
\end{equation*}
Since $A^{-1}$ is also a diagonal matrix, it follows that $\widetilde{f}$ is separately convex, too. Hence for any given $\xi_0,\xi_1\in\R^{m\times d}$ we infer that
\begin{align*}
	|\widetilde{f}(\xi_1)-\widetilde{f}(\xi_0)|&\leq\sqrt{m+d}\sup_{|\xi|\leq 2(\Lambda^{\frac{1}{p}}+|\xi_1|+|\xi_0|)} \frac{2|\widetilde{f}(\xi)|}{\Lambda^{\frac{1}{p}}+|\xi_0|+|\xi_1|}|\xi_1-\xi_0|
	\\
	&\leq \sqrt{m+d}\left(2^{p+1}(\Lambda^{\frac{1}{p}}+|\xi_0|+|\xi_1|)^{p-1}+2\Lambda^{\frac{p-1}{p}}\right)|\xi_1-\xi_0|,
\end{align*}
which implies \eqref{eq:f_quantitativeLipschitz} via the formula $f(\xi)=\widetilde{f}(\xi A)$. When $f$ is differentiable with respect to the last variable, we deduce from the chain rule and the above local Lipschitz estimate that
\begin{equation*}
	|\langle \partial_{\xi}f(\xi),z\rangle|=|\langle (\partial_{\xi}\widetilde{f})(\xi A) ,zA\rangle|\nonumber
	\leq C\left((\Lambda^{\frac{1}{p}}+|\xi A|)^{p-1}+\Lambda(\w,x)^{\frac{p-1}{p}}\right)|zA|.	
\end{equation*}
\end{proof}

Now we can formulate an alternative formula for $f_{\rm hom}$ that allows us to prove its strict convexity and differentiability with elementary arguments.
	\begin{lemma}\label{l.F_pot_formula}
		In addition to Assumption \ref{a.1}, suppose that the map $\xi\mapsto f(\w,x,\xi)$ is convex. Then for all $\xi\in\R^{m\times d}$, there exists $h_{\xi}\in (F_{\rm pot}^1)^m$ such that
		\begin{equation}\label{eq:pot_formula}
			f_{\rm hom}(\xi)=\inf_{h\in (F_{\rm pot}^1)^m}\mathbb{E}[f(\cdot,0,\xi+h)]= \mathbb{E}[f(\cdot,0,\xi+h_{\xi})].
		\end{equation}
	\end{lemma}
	\begin{proof}
		We first show the existence of a minimizer for the above minimization problem. The map $h\mapsto \mathbb{E}[f(\cdot,0,\xi+h)]$ is lower semicontinuous with respect to strong convergence in $L^1(\Omega,\R^{m\times d})$ due to Fatou's lemma and Assumption \ref{a.1}. As this functional is convex, it is also weakly lower semicontinuous. Being a linear space, $(F_{\rm pot}^1)^m$ is weakly closed due to Lemma \ref{l.F_pot}. Hence it suffices to show that a minimizing sequence $h_n$ is relatively weakly compact in $L^1(\Omega,\R^{m\times d})$. $\Omega$ has finite measure, so that it suffices to show that minimizing sequences are $L^1$-bounded and equi-integrable. Fix any measurable set $F\in\mathcal{F}$. Then, by H\"older's inequality and Assumption \ref{a.1}, we have that
		\begin{align*}
			\int_F |\xi+h_n|\,\mathrm{d}\mu&\leq \Big(\underbrace{\int_F|A(\w,0)(\xi+h_n(\w))|^p\,\mathrm{d}\mu}_{\leq C\mathbb{E}[f(\cdot,0,\xi+h_n)]\leq C(\xi)}\Big)^{\frac{1}{p}}\left(\int_F |A(\w,0)^{-1}|^{\frac{p}{p-1}}\right)^{\frac{p-1}{p}}
			\\
			&\leq C(\xi)\left(\int_F |A(\w,0)^{-1}|^{\frac{p}{p-1}}\,\mathrm{d}\mu\right)^{\frac{p-1}{p}}.
		\end{align*}
		Since the function $\w\mapsto |A(\w,0)^{-1}|^{p/(p-1)}$ is integrable by assumption, the above estimate proves the equi-integrability and boundedness of $h_n$ and thus there exists a minimizer $h_{\xi}\in (F_{\rm pot}^1)^m$. For the proof of the first equality in \eqref{eq:pot_formula}, we roughly follow \cite{DG_unbounded}. 
		Fix a cube $Q\subset D$. For every $\e>0$ consider a function $u_{\e}\in W^{1,1}(D,\R^m)$ with $\dashint_Q u_{\e}\dx=0$ and $\dashint_Q\nabla u_{\e}\dx=0$, satisfying
		\begin{equation*}
			\frac{1}{|Q|}F_{\e}(\w,u_{\e},Q)-\e\leq\inf\left\{\dashint_Q f(\w,\tfrac{x}{\e},\xi+\nabla u(x))\dx:\,\dashint_Q\nabla u\dx=0\right\}. 
		\end{equation*}
		Since $u=0$ is admissible in the above minimization problem, it follows from Lemma \ref{l.compactness} and the Poincar\'e inequality that up to a subsequence (not relabeled) we have $u_{\e}\to u$ in $L^1(Q,\R^m)$ for some $u\in W^{1,p}(Q,\R^m)$ with $\dashint_Q\nabla u\dx=0$. Applying Theorem \ref{thm.Gamma_pure} with $Q=D$, Jensen's inequality yields that
		\begin{align}\label{eq:zero_average_gradient}
			f_{\rm hom}(\xi)&=f_{\rm hom}\left(\dashint_Q\xi+\nabla u(x)\dx\right)\leq \dashint_Q f_{\rm hom}(\xi+\nabla u(x))\dx\nonumber
			\\
			&\leq \liminf_{\e\to 0}\inf\left\{\dashint_Q f(\w,\tfrac{x}{\e},\xi+\nabla u(x))\dx:\,\dashint_Q \nabla u\dx=0\right\}.
		\end{align}
		Here we used that $f_{\rm hom}$ inherits the convexity of $f$, which follows from the general fact that the $\Gamma$-limit of convex functionals is convex. Next, given a minimizer $h_{\xi}\in (F_{\rm pot}^1)^m$, let $\varphi_{\xi}:\Omega\to W^{1,1}_{\rm loc}(\R^d,\R^m)$ be the function given componentwise by Lemma \ref{l.F_pot}. Define the $W^{1,1}_{\rm loc}(\R^d,\R^m)$-valued random variable $u_{\xi,\e}$ by
		\begin{equation*}
			u_{\xi,\e}(x)=\e\varphi_{\xi}(\w)(x/\e)-\dashint_{Q}\nabla\varphi_{\xi}(\w)(y/\e)x\,\mathrm{d}y.
		\end{equation*}
		Then $\int_Q\nabla u_{\e}\dx=0$ and therefore almost surely
		\begin{align}\label{eq:corrector_good}
			\inf\left\{\dashint_Q f(\w,\tfrac{x}{\e},\xi+\nabla u(x))\dx:\,\dashint_Q\nabla u\dx=0\right\}&\leq \dashint_Qf(\w,\tfrac{x}{\e},\xi+\nabla u_{\xi,\e}(x))\dx\nonumber
			\\
			&=\dashint_Q f\left(\w,\tfrac{x}{\e},\xi+h_{\xi}(\tau_{x/\e}\w)-\dashint_Qh_{\xi}(\tau_{y/\e}\w)\dy\right)\dx.
		\end{align}
		The ergodic theorem \ref{thm.additiv_ergodic} implies that (expect for a null set depending on $\xi$ that we exclude for this proof)
		\begin{align}
			\lim_{\e\to 0}\dashint_Qh_{\xi}(\tau_{y/\e}\w)\dy&=\mathbb{E}[h_{\xi}]=0,\label{eq:ergodic_H}
			\\
			\lim_{\e\to 0}\dashint_Q f(\w,\tfrac{x}{\e},\xi+h_{\xi}(\tau_{x/\e}\w))\dx&=\mathbb{E}[f(\cdot,0,\xi+h_{\xi})].\label{eq:ergodic_energy}
		\end{align}
		To conclude the inequality $f_{\rm hom}(\xi)\leq \mathbb{E}[f(\cdot,0,\xi+h_{\xi})]$, we recall that due to Lemma \ref{l.f_differentiable} 
		\begin{equation*}
			|f(\w,x,\xi_1)-f(\w,x,\xi_0)|\leq C	\left((\Lambda(\w,x)^{\frac{1}{p}}+|\xi_0A(\w,x)|+|\xi_0A(\w,x)|)^{p-1}+\Lambda(\w,x)^{\frac{p-1}{p}}\right)|(\xi_1-\xi_0)A(\w,x)|.
		\end{equation*}
		Set $H_{\xi,\e}(\w)=\dashint_Q h_{\xi}(\tau_{y/\e}\w)\dy$. Due to \eqref{eq:ergodic_H} we may assume that $|H_{\xi,\e}(\w)|\leq 1$. Inserting the above local Lipschitz-estimate into \eqref{eq:corrector_good} and replacing all exponents $p-1$ by $p$ (without loss of generality $\Lambda\geq 1$) yields that
		\begin{align*}
			\dashint_Q f\left(\w,\tfrac{x}{\e},\xi+h_{\xi}(\tau_{x/\e}\w)-H_{\xi,\e}(\w)\right)\dx&\leq \dashint_Q f(\w,\tfrac{x}{\e},\xi+h_{\xi}(\tau_{x/\e}\w))\dx
			\\
			&\quad+\Bigg\{\dashint_Q \left(\Lambda(\w,\tfrac{x}{\e})+|A(\w,\tfrac{x}{\e})|^p +|(\xi+h_{\xi}(\tau_{x/\e}\w))A(\w,\tfrac{x}{\e})|^p\right)\dx
			\\
			&\quad\quad+\dashint_Q \Lambda(\w,\tfrac{x}{\e})\dx\Bigg\} \,C|H_{\xi,\e}(\w)|.
		\end{align*}
		Since $H_{\xi,\e}(\w)\to 0$ by \eqref{eq:ergodic_H}, we deduce from \eqref{eq:ergodic_energy}, the ergodic theorem applied to $|A|^p$ and $\Lambda$ as well as the lower bound in Assumption \ref{a.1} that
		\begin{equation*}
			\liminf_{\e\to 0}\dashint_Q f\left(\w,\tfrac{x}{\e},\xi+h_{\xi}(\tau_{x/\e}\w)-H_{\xi,\e}(\w)\right)\dx\leq \mathbb{E}[f(\cdot,0,\xi+h_{\xi})].
		\end{equation*}
		Combined with \eqref{eq:zero_average_gradient} this shows that \begin{equation}\label{eq:oneinequality}
			f_{\rm hom}\leq \mathbb{E}[f(\cdot,0,\xi+h_{\xi})].
		\end{equation}
		We still need to show the reverse inequality. By Lemma \ref{l.measurable}, we obtain for every $\e>0$ a measurable function $u_{\xi,\e}:\Omega\to W^{1,1}_0(Q/\e,\R^m)$ such that almost surely
		\begin{equation*}
			F_1(\w,\xi x+u_{\xi,\e}(\w),Q/\e)=\mu_{\xi}(\w,Q/\e):=\inf\{F_1(\w,u,Q_1/\e):\,u\in \xi x+W^{1,1}_0(Q/\e,\R^m)\}.
		\end{equation*}
		(Due to convexity, the almost sure existence of minimizers can be proven as in Lemma \ref{l.epsexistence}.) Next, we need to switch to a jointly measurable version. From \cite[Lemma 16, p. 196]{DS} we infer that there exist $\mathcal{F}\otimes\L^d$-measurable functions $v_{\xi,\e},b_{\xi,\e}:\Omega\times Q/\e\to\R$ such that $v_{\xi,\e}(\w,\cdot)=u_{\xi,\e}(\w)$ and $b_{\xi,\e}(\w,\cdot)=\nabla u_{\xi,\e}(\w)$ for a.e. $\w\in\Omega$. In particular, for a.e. $\w\in\Omega$ we have $v_{\xi,\e}(\w,\cdot)\in W_0^{1,1}(Q/\e,\R^m)$ and $\nabla v_{\xi,\e}(\w,\cdot)=b_{\xi,\e}(\w,\cdot)$. With a slight abuse of notation, we therefore write $b_{\xi,\e}=\nabla v_{\xi,\e}$. By \cite[Lemma 7.1]{JKO} we can assume that the set of $\w$, for which these properties hold, is invariant under the group action $\tau_{x}$ for almost all $x\in\R^d$. Finally, we extend $v_{\xi,\e}$ and $\nabla v_{\xi,\e}$ to $\Omega\times(\R^d\setminus Q/\e)$ by $0$. This extension is jointly measurable on $\Omega\times\R^d$. We now define $h_{\xi,\e}\in L^1(\Omega,\R^{m\times d})$ by
		\begin{equation*}
			h_{\xi,\e}(\w)=\dashint_{Q/\e}\nabla v_{\xi,\e}(\tau_{-y}\w,y)\dy.
		\end{equation*}
		Note that $h_{\xi,\e}$ is well defined and measurable due to the joint measurability of $\nabla v_{\xi,\e}$ and the joint measurability of the group action. To see that it is integrable, we can use Fubini's theorem and a change of variables in $\Omega$ to deduce that
		\begin{align*}
			\int_{\Omega}|\xi+h_{\xi,\e}(\w)|\,\mathrm{d}\mu&\leq \dashint_{Q/\e}\int_{\Omega}|\xi+\nabla v_{\xi,\e}(\tau_{-y}\w,y)|\,\mathrm{d}\mu\dy=\dashint_{Q/\e}\int_{\Omega}|\xi+\nabla v_{\xi,\e}(\w,y)|\,\mathrm{d}\mu\dy
			\\
			&=\int_{\Omega}\dashint_{Q/\e}|\xi+\nabla v_{\xi,\e}(\w,y)|\dy\,\mathrm{d}\mu
		\end{align*}
		The last term can be controlled via H\"older's inequality as at the beginning of this proof, using that for a.e. $\w\in\Omega$ the function $\nabla v_{\xi,\e}(\w,\cdot)$ is the gradient of an energy minimizer on $Q/\e$. We argue that $h_{\xi,\e}\in (F_{\rm pot}^1)^m$. Since for a.e. $\w\in\Omega$ the function $\nabla v_{\xi,\e}$ is the weak gradient of $u_{\xi,\e}(\w)\in W^{1,1}_0(Q/\e,\R^m)$, it follows from Fubini's theorem and a change of variables in $\Omega$ that
		\begin{equation*}
			\int_{\Omega}h_{\xi,\e}(\w)\,\mathrm{d}\mu=\int_{\Omega}\underbrace{\dashint_{Q/\e}\nabla v_{\xi,\e}(\w,y)\dy}_{=0 \text{ almost surely}}\,\mathrm{d}\mu=0.
		\end{equation*}
		Hence, it suffices to show that all rows of $x\mapsto h_{\xi,\e}(\tau _x\w)$ satisfy the curl-free condition of Definition~\ref{eq:def_F_pot}. To this end, we derive a suitable formula for the distributional derivative of this map. Fix $\theta\in C^{\infty}_c(\R^d)$ and an index $1\leq j\leq d$. Since $\nabla v_{\xi,\e}(\w,\cdot)=0$ on $\R^d\setminus (Q/\e)$, we can write 
		\begin{align*}
			\int_{\R^d}	h_{\xi,\e}(\tau_x\w)\partial_j\theta(x)\dx&=\int_{\R^d}\int_{\R^d}\frac{\nabla v_{\xi,\e}(\tau_{x-y}\w,y)}{|Q/\e|}\partial_j\theta(x)\dy\dx=\int_{\R^d}\int_{\R^d}\frac{\nabla v_{\xi,\e}(\tau_{z}\w,x-z)}{|Q/\e|}\partial_j\theta(x)\,\mathrm{d}z\dx
			\\
			&=\int_{\R^d}\int_{\R^d}\frac{\nabla v_{\xi,\e}(\tau_{z}\w,y)}{|Q/\e|}\partial_j\theta(y+z)\dy\,\mathrm{d}z.
		\end{align*}
		%		From the first to the second line we used Fubini's theorem (joint measurability follows from the joint measurability of $\nabla v_{\xi,\e}$ and of the group action $\tau$, while the integral on the left-hand side is almost surely finite since $x\mapsto |h_{\xi,\e}(\tau_x\w)|\in L^1_{\rm loc}(\R^d)$ for almost every $\w\in\Omega$ due to the integrability of $h_{\xi,\e}$). 
		In order to conclude that $h_{\xi,\e}\in (F_{\rm pot}^1)^m$, it suffices to note that for a.e. $\w\in\Omega$ and almost every $z\in\R^d$ the function $y\mapsto \nabla v_{\xi,\e}(\tau_z\w,y)$ is the gradient of the Sobolev function $u_{\xi,\e}(\tau_z\w)\in W^{1,1}_0(Q/\e,\R^m)$, so that the curl-free conditions follows. Now we can conclude the proof. Since $h_{\xi,\e}\in (F_{\rm pot}^1)^m$, it follows from Jensen's inequality that
		\begin{align*}
			\mathbb{E}[f(\cdot,0,\xi+h_{\xi})]&\leq \mathbb{E}[f(\cdot,0,\xi+h_{\xi,\e})]\leq \int_{\Omega}\dashint_{Q/\e}f(\w,0,\xi+\nabla  v_{\xi,\e}(\tau_{-y}\w,y))\dy\,\mathrm{d}\mu
			\\
			&=\int_{\Omega}\dashint_{Q/\e}f(\tau_{y}\w,0,\xi+\nabla  v_{\xi,\e}(\w,y))\dy\,\mathrm{d}\mu=\int_{\Omega}\dashint_{Q/\e}f(\w,y,\xi+\nabla  v_{\xi,\e}(\w,y))\dy\,\mathrm{d}\mu
			\\
			&=\frac{1}{|Q/\e|}\mathbb{E}[\mu_{\xi}(\w,Q/\e)],
		\end{align*}
		where we used the stationarity of $f$, and that $\nabla v_{\xi,\e}(\w,\cdot)=\nabla u_{\xi,\e}(\w)$ almost surely in the last step. Passing to the limit as $\e\to 0$, we obtain from $L^1$-convergence in the subadditive ergodic theorem (see \cite{Smythe} or \cite[Theorem 2.3, p. 203]{Krengel}) that
		\begin{equation*}
			\mathbb{E}[f(\cdot,0,\xi+h_{\xi})]\leq f_{\rm hom}(\xi),
		\end{equation*}
		which concludes the proof.
\end{proof} 	
	
We now prove that strict convexity of $\xi\mapsto f(\w,x,\xi)$ is inherited by $f_{\rm hom}$.
\begin{proposition}\label{p.f_hom_strictly_convex}
		In addition to Assumption \ref{a.1}, assume that the map $\xi\mapsto f(\w,x,\xi)$ is strictly convex. Then also $f_{\rm hom}$ is strictly convex.
\end{proposition}
\begin{proof}
		Fix $\xi_0,\xi_1\in\R^{m\times d}$ and $t\in (0,1)$ and let $h_{\xi_0},h_{\xi_1}\in (F_{\rm pot}^1)^m$ be as in Lemma \ref{l.F_pot_formula}. From the same lemma we know that
		\begin{align*}
			f_{\rm hom}(t\xi_0+(1-t)\xi_1)&\leq \mathbb{E}[f(\cdot,0,t(\xi_0+h_{\xi_0})+(1-t)(\xi_1+h_{\xi_1}))]
			\\
			&\leq t \mathbb{E}[f(\cdot,0,\xi_0+h_{\xi_0})]+(1-t)\mathbb{E}[f(\cdot,0,\xi_1+h_{\xi_1})]=tf_{\rm hom}(\xi_0)+(1-t)f_{\rm hom}(\xi_1).	
		\end{align*}
		To have an equality, we need in particular that the second inequality is an equality. Since $\xi\mapsto f(\w,0,\xi)$ is strictly convex, the only possibility to have an inequality is when $\xi_0+h_{\xi_0}=\xi_1+h_{\xi_1}$ almost surely. Taking expectations, we deduce in this case that $\xi_0=\xi_1$ (recall that $\mathbb{E}[h]=0$ for all $h\in (F_{\rm pot}^1)^m$). Hence $f_{\rm hom}$ is strictly convex as claimed.
\end{proof}

In the next lemma we show the crucial property that $f_{\rm hom}$ is differentiable.
\begin{proposition}\label{p.differentiable}
Under Assumption \ref{a.2}, the function $f_{\rm hom}$ given by Lemma \ref{l.existence_f_hom} is continuously differentiable and the derivative satisfies the estimate
\begin{equation*}
	|\nabla f_{\rm hom}(\xi)|\leq C(1+|\xi|^{p-1})\quad\text{ for all }\xi\in\R^{m\times d}\text{ and some constant }C>0.
\end{equation*}
\end{proposition}
\begin{proof}
We first show that $f_{\rm hom}$ is differentiable. We know from the previous lemma that $f_{\rm hom}$ is (separately) convex. Due to \cite[Corollary 2.5]{BKK} it thus suffices to show that $f_{\rm hom}$ is upper semidifferentiable, that is, for all  $\xi\in\R^{m\times d}$ there exists $\xi^*\in\R^{m\times d}$ such that
\begin{equation*}
	\limsup_{\eta\to \xi}\frac{f_{\rm hom}(\eta)-f_{\rm hom}(\xi)-\langle\xi^*,\eta-\xi\rangle}{|\eta-\xi|}\leq 0.
\end{equation*}
We prove this property using the formula for $f_{\rm hom}$ given by Lemma \ref{l.F_pot_formula}. Given $\xi\in\R^{m\times d}$ and $h_{\xi}\in (F_{\rm pot}^1)^m$ as in Lemma \ref{l.F_pot_formula}, we define $\xi^*\in\R^{m\times d}$ as the matrix given by
\begin{equation*}
\langle \xi^*,v\rangle:=\mathbb{E}[\langle\partial_{\xi}f(\cdot,0,\xi+h_{\xi}),v\rangle].
\end{equation*}
Then for $\eta\in\R^{m\times d}$ we have that
\begin{equation}\label{eq:upper_diff}
\frac{f_{\rm hom}(\eta)-f_{\rm hom}(\xi)-\langle\xi^*,\eta-\xi\rangle}{|\eta-\xi|}\leq \mathbb{E}\left[\frac{f(\cdot,0,\eta+h_{\xi})-f(\cdot,0,\xi+h_{\xi})-\langle\partial_{\xi}f(\cdot,0,\xi+h_{\xi}),\eta-\xi\rangle}{|\eta-\xi|}\right].
\end{equation}
The integrand on the right-hand side tends to $0$ almost surely when $\eta\to\xi$ since $f(\w,0,\cdot)$ is differentiable in the point $\xi+h_{\xi}(\w)$. It remains to apply the dominated convergence theorem. Due to the mean value theorem and \eqref{eq:f_quantitativedifferentiable} we have for some random number $t=t(\w)\in [0,1]$ that
\begin{align*}
&\quad\left|\frac{f(\w,0,\eta+h_{\xi}(\w))-f(\w,0,\xi+h_{\xi}(\w))}{|\eta-\xi|}\right|= \left|\frac{\langle\partial_{\xi}f(\w,0,t\eta+(1-t)\xi+h_{\xi}(\w)),\eta-\xi\rangle}{|\eta-\xi|}\right|
\\
&\leq C \left((\Lambda(\w,0)^{\frac{1}{p}}+|(\eta-\xi)A(\w,0)|+|(\xi+h_{\xi}(\w))A(\w,0)|)^{p-1}+\Lambda(\w,0)^{\frac{p-1}{p}}\right)|A(\w,0)|
\\
&\leq C \left((\Lambda(\w,0)^{\frac{1}{p}}+|A(\w,0)|+|(\xi+h_{\xi}(\w))A(\w,0)|)^{p-1}+\Lambda(\w,0)^{\frac{p-1}{p}}\right)|A(\w,0)|,
\end{align*}
where we assumed that $|\eta-\xi|\leq 1$ from the second to the last line.
To see that the term in the last line is integrable, one applies H\"older's inequality with exponents $p$ and $p/(p-1)$ and uses the integrability of $\Lambda(\w,0)$, $\|A(\w,0)\|^p$, and of $|(\xi+h_{\xi}(\w))A(\w,0)|^p$. For $t=0$ we obtain the bound for the remaining term in\eqref{eq:upper_diff}. Hence the dominated convergence theorem yields that $f_{\rm hom}$ is upper semidifferentiable and therefore also differentiable.

By \cite[Theorem 4.65]{FoLe} the derivative of a (separately) convex function is continuous. The claimed bound on the derivative follows from the estimate
\begin{equation*}
	|f_{\rm hom}(\xi_1)-f_{\rm hom}(\xi_2)|\leq C(1+|\xi_1|^{p-1}+|\xi_2|^{p-1})|\xi_1-\xi_2|,
\end{equation*}
which holds due to the (separate) convexity of $f_{\rm hom}$ and its $p$-growth from above (cf. \cite[Proposition 4.64]{FoLe}).
\end{proof} 

Having established the strict convexity and the differentiability of $f_{\rm hom}$, it remains to prove the existence of minimizers for fixed $\e$ and that they satisfy the associated Euler-Lagrange equation of $F_{\e}$. 
\begin{lemma}\label{l.epsexistence}
Under Assumption \ref{a.2}, for every $\e>0$, $g\in W^{1,\infty}_{\rm loc}(\R^d,\R^m)$ and $f_{\e}\in L^d(D,\R^m)$ there exists a unique minimizer of the problem
\begin{equation*}
	\inf \left\{\int_D f(\w,\tfrac{x}{\e},\nabla u(x))\dx-\int_D f_{\e}(x)\cdot u(x)\dx:\,u\in g+W_0^{1,1}(D,\R^m)\right\},
\end{equation*}
which is the  unique weak solution of the PDE \eqref{eq:epsPDE} in the affine energy space
\begin{equation*}
\mathcal{A}_{g,\e}(\w):=\{u\in g+W_0^{1,1}(D,\R^m):\,\int_D|\nabla u(x)A(\w,x)|^p\dx<+\infty\}.
\end{equation*}
\end{lemma}
\begin{proof}
The function $g$ is admissible for the infimum problem and has finite energy. Now consider a minimizing sequence $u_n\in g+W^{1,1}_0(D,\R^m)$. Using the same estimates as in the proof of Theorem \ref{thm:Dirichlet_and_forces} in Section \ref{s.bd+forces}, it follows that 
\begin{equation*}
	\sup_{n\in\N}F_{\e}(\w,u_n,D)<+\infty.
\end{equation*}
Combining H\"older's inequality with the fact that for fixed $\e>0$ the function $x\mapsto |A(\w,\tfrac{x}{\e})^{-1}|^{\frac{p}{p-1}}$ is equi-integrable, one shows as in the proof of Lemma \ref{l.compactness} that $\nabla u_n$ is bounded and equi-integrable. Due to Poincaré's inequality it follows that, up to a subsequence, $u_n$ converges to some $u\in g+W^{1,1}_0(D,\R^m)$ weakly in $W^{1,1}(D,\R^m)$ and by Theorem \ref{thm.embedding} also strongly in $L^{d/(d-1)}(D,\R^m)$, which allows to pass to the limit in the linear term. The functional $u\mapsto F_{\e}(\w,u,D)$ is lower semicontinuous with respect to strong convergence in $W^{1,1}(D,\R^m)$ due to Fatou's lemma and by convexity also weakly lower semicontinuous. Hence, by the direct method of the calculus of variations, $u$ is a minimizer. Due to strict convexity of $\xi\mapsto f(\w,\tfrac{x}{\e},\xi)$, the minimizer under Dirichlet boundary conditions is unique. To prove the assertion about the Euler-Lagrange equation, we fix $\varphi\in W_0^{1,1}(D,\R^m)$ such that $\int_D|A(\w,\tfrac{x}{\e})\nabla \varphi(x)|^p<+\infty$. An elementary calculation based on the dominated convergence theorem, the bound \eqref{eq:f_quantitativedifferentiable} and H\"older's inequality, implies
%and compute the Gat\'eaux-derivative of $F_{\e}$ at $u$ in the direction $\varphi$. For $t\in (-1,1)\setminus \{0\}$, by the fundamental theorem of calculus and Assumption \ref{a.2} it holds that
%\begin{align*}
%&\quad \left|\frac{F_{\e}(\w,u+t\varphi,D)-F_{\e}(\w,u,D)}{t}-\int_D\partial_{\xi}f(\w,\tfrac{x}{\e},\nabla u(x))\nabla\varphi(x)\dx\right|
%\\
%&\leq\int_D \int_0^1|\langle\partial_{\xi}f(\w,\tfrac{x}{\e},\nabla u(x)+st\nabla \varphi(x))-\partial_{\xi}f(\w,\tfrac{x}{\e},\nabla u(x)),\nabla\varphi(x)\rangle|\,\mathrm{d}s\dx
%\\
%&\leq C \int_D \left(1+|A(\w,\tfrac{x}{\e})\nabla u(x)|+|A(\w,\tfrac{x}{\e})\nabla\varphi(x)|\right)^{p-1-\alpha}t^{\alpha}|A(\w,\tfrac{x}{\e})\nabla\varphi(x)|^{1+\alpha}\dx
%\\
%&\leq Ct^{\alpha}\int_D 1+|A(\w,x)\nabla u(x)|^p+|A(\w,\tfrac{x}{\e})\nabla\varphi(x)|^p\dx.
%\end{align*}
%The last integral is finite, so letting $t\to 0$ we deduce 
that $F_{\e}(\w,\cdot,D)$ is Gat\'eaux-differentiable at $u$ in the direction $\varphi$ with derivative
\begin{equation*}
	\delta F_{\e}(\w,u,D)\varphi=\int_D \partial_{\xi}f(\w,\tfrac{x}{\e},\nabla u(x))\nabla\varphi(x)\dx.
\end{equation*}
Since the term $\int_D f_{\e}u\dx$ is linear and continuous on $W^{1,1}(D,\R^m)$, it follows that the minimizer solves \eqref{eq:epsPDE} in the claimed weak form. The solution of the PDE is unique in $\mathcal{A}_{g,\e}(\w)$ since any other solution would also be a global minimizer due to the convexity of the energy (which is infinite outside $\mathcal{A}_{g,\e}(\w)$).
\end{proof}

\begin{proof}[Proof of Theorem \ref{thm.PDEs}]
The claimed properties of $f_{\rm hom}$ follow from Proposition \ref{p.f_hom_strictly_convex} and Proposition \ref{p.differentiable}. The existence and uniqueness of solutions of the PDE at the $\e$-level is proven in Lemma \ref{l.epsexistence}, the existence and uniqueness of solutions (via the existence and uniqueness of minimizers) of the homogenized PDE follow by standard methods using the properties of $\nabla f_{\rm hom}$ stated in Lemma \ref{p.differentiable} and the almost sure convergence of solutions is a consequence of the convergence of minimizers under $\Gamma$-convergence, which holds due to Theorem \ref{thm:Dirichlet_and_forces}.	
\end{proof}
%As we argue now, condition B1) and B2) imply our growth conditions. Indeed, the fundamental theorem of calculus yields that
%\begin{align*}
%	f(\w,x,\xi)&=f(\w,x,0)+\int_0^1A(\w,x,t\xi)\xi\,\mathrm{d}t\overset{2)}{\leq}f(\w,x,0)+C\int_0^1\lambda(\w,x)(1+|t\xi|^{p})\,\mathrm{d}t
%	\\
%	&\leq f(\w,x,0)+C\lambda(\w,x)(1+|\xi|^p)
%\end{align*}
%and in the same way
%\begin{align*}
%	f(\w,x,\xi)&=f(\w,x,0)+\int_0^1A(\w,x,t\xi)\xi\,\mathrm{d}t\overset{3)}{\geq}f(\w,x,0)+c\int_0^1\lambda(\w,x)|t\xi|^{p}\,\mathrm{d}t
%	\\
%	&\geq f(\w,x,0)+c\lambda(\w,x)|\xi|^p.
%\end{align*}
%Hence, if we set $f(\w,x,0)=0$ (which does not affect $A$), we obtain our growth conditions. However, we also stress that the assumption of $A$ being the derivative of a function poses a restriction on the curl of $A$ with respect to $\xi$. 

%
% 

\appendix

\section{Differentiability of $f_{\rm hom}$ without convexity assumptions}\label{app:0}
In this section we show that the differentiability of $f_{\rm hom}$ can be obtained without the formula given by Lemma \ref{l.F_pot_formula}, but with the additional assumptions \eqref{eq:quantitativeC1} and \eqref{eq:quantitativeBound}.
\begin{lemma}\label{l.app_diff}
In addition to Assumption \ref{a.1}, assume that the map $\xi\mapsto f(\w,x,\xi)$ is differentiable and that its derivative satisfies \eqref{eq:quantitativeC1} and \eqref{eq:quantitativeBound}. Then $f_{\rm hom}$ given by Theorem \ref{thm.Gamma_pure} is continuously differentiable and the derivative satisfies the bound $|\nabla f_{\rm hom}(\xi)|\leq C(1+|\xi|^{p-1})$.
\end{lemma}
\begin{proof}
The function $f_{\rm hom}$ is quasiconvex and finite, so in particular it is separately convex. Hence the only point to be adapted in the proof of Proposition \ref{p.differentiable} is the upper semidifferentiability of $f_{\rm hom}$, that is, for all  $\xi\in\R^{m\times d}$ there exists $\xi^*\in\R^{m\times d}$ such that
\begin{equation*}
	\limsup_{\eta\to \xi}\frac{f_{\rm hom}(\eta)-f_{\rm hom}(\xi)-\langle\xi^*,\eta-\xi\rangle}{|\eta-\xi|}\leq 0.
\end{equation*}
Fix $\xi\in\R^{m\times d}$ and given $\e>0$ we choose a function $u_{\xi,\e}\in W^{1,1}_0(Q/\e,\R^m)$ such that
\begin{equation*}
	\dashint_{Q/\e}f(\w,x,\xi+\nabla u_{\xi,\e})\dx\leq \frac{1}{|Q/\e|}\mu_{\xi}(\w,Q/\e)+\e,
\end{equation*}
where $\mu_{\xi}(\w,Q/\e)$ is given by Lemma \ref{l.existence_f_hom}. We then set $\xi^*_{\e}\in\R^{m\times d}$ as the matrix defined by
\begin{equation*}
	\langle\xi_{\e}^*,v\rangle:=	\dashint_{Q/\e}\langle\partial_{\xi}f(\w,x,\xi+\nabla u_{\xi,\e}),v\rangle\dx.
\end{equation*}
Let $\eta\in\R^{m\times d}$ be such that $|\eta-\xi|\leq 1$. Then it holds that
\begin{align*}
	\frac{1}{|Q/\e|}\mu_{\eta}(\w,Q_/\e)-\frac{1}{|Q/\e|}\mu_{\xi}(\w,Q/\e)-\langle\xi_{\e}^*,\eta-\xi\rangle&\leq \dashint_{Q/\e}f(\w,x,\eta+\nabla u_{\xi,\e})-f(\w,x,\xi+\nabla u_{\xi,\e})\dx
	\\
	&\quad-\dashint_{Q/\e}\langle\partial_{\xi}f(\w,x,\xi+\nabla u_{\xi,\e}),\eta-\xi\rangle\dx+\e.
\end{align*}
To bound the right-hand side, we use the fundamental theorem of calculus and \eqref{eq:quantitativeC1} to deduce that
\begin{align*}
	&\quad\frac{1}{|Q/\e|}\mu_{\eta}(\w,Q_/\e)-\frac{1}{|Q/\e|}\mu_{\xi}(\w,Q/\e)-\langle\xi_{\e}^*,\eta-\xi\rangle-\e
	\\
	&\leq \dashint_{Q/\e}\int_0^1\langle\partial_{\xi}f(\w,x,t(\eta-\xi)+\xi+\nabla u_{\xi,\e})-\partial_{\xi}f(\w,x,\xi+\nabla u_{\xi,\e}),\eta-\xi\rangle\,\mathrm{d}t\dx
	\\
	&\leq C\dashint_{Q/\e}\int_0^1(\Lambda(\w,x)^{\frac{1}{p}}+t|(\eta-\xi)A(\w,x)|+2|(\xi+\nabla u_{\xi,\e})A(\w,x)|)^{p-1-\alpha}t^{\alpha}(|A(\w,x)||\eta-\xi|)^{1+\alpha}\dx.
\end{align*}
Since $t\in [0,1]$ and $p-1\geq\alpha>0$, we may replace $t$ by $1$ and absorb the factor $2$ in $C$. Then H\"older's inequality and the bound $|\eta-\xi|\leq 1$ yield that
\begin{align*}
	&\quad\frac{1}{|Q/\e|}\mu_{\eta}(\w,Q_/\e)-\frac{1}{|Q/\e|}\mu_{\xi}(\w,Q/\e)-\langle\xi_{\e}^*,\eta-\xi\rangle-\e
	\\
	&\leq C\left(\dashint_{Q/\e}(\Lambda(\w,x)^{\frac{1}{p}}+|A(\w,x)|+|(\xi+\nabla u_{\xi,\e})A(\w,x)|)^{p}\dx\right)^{\frac{p-1-\alpha}{p}}\left(\dashint_{Q/\e}(|A(\w,x)||\eta-\xi|)^p\dx\right)^{\frac{1+\alpha}{p}}
	\\
	&\leq C\left(\dashint_{Q/\e}\Lambda(\w,x)+|A(\w,x)|^p\dx+\frac{1}{|Q/\e|}\mu_{\xi}(\w,Q/\e)+\e\right)^{\frac{p-1-\alpha}{p}}\left(\dashint_{Q/\e}|A(\w,x)|^p\dx\right)^{\frac{1+\alpha}{p}}|\eta-\xi|^{1+\alpha}.
\end{align*}
Letting $\e\to 0$, we deduce from Lemma \ref{l.existence_f_hom} and the ergodic theorem \ref{thm.additiv_ergodic} that
\begin{equation*}
	f_{\rm hom}(\eta)-f_{\rm hom}(\xi)-\liminf_{\e\to 0}\langle\xi_{\e}^*,\eta-\xi\rangle\leq C\Big(\mathbb{E}[\Lambda(\cdot,0)+|A(\cdot,0)|^p]+f_{\rm hom}(\xi)\Big)^{\frac{p-1-\alpha}{p}}\mathbb{E}[|A(\cdot,0)|^p]^{\frac{1+\alpha}{p}}|\eta-\xi|^{1+\alpha}
\end{equation*}
and since $\alpha>0$, the upper semidifferentiability of $f_{\rm hom}$ follows once we show that $\xi_{\e}^*$ is relatively compact in $\R^{m\times d}$. To show boundedness, we use \eqref{eq:quantitativeC1} and \eqref{eq:quantitativeBound} to infer that
\begin{align*}
	|\xi_{\e}^*|&\leq\dashint_{Q/\e}|\partial_{\xi}f(\w,x,\xi+\nabla u_{\xi,\e})|\dx\leq \dashint_{Q/\e}|\partial_{\xi}f(\w,x,\xi+\nabla u_{\xi,\e})-\partial_{\xi}f(\w,x,0)|+\Lambda(\w,x)\dx
	\\
	&\leq C\dashint_{Q/\e}|A(\w,x)|\Big(\Lambda(\w,x)^{\frac{1}{p}}+|(\xi+\nabla u_{\xi,\e})A(\w,x)|\Big)^{p-1-\alpha}|(\xi+\nabla u_{\xi,\e})A(\w,x)|^{\alpha}+\Lambda(\w,x)\dx
	\\
	&\leq C\dashint_{Q/\e}|A(\w,x)|\Big(\Lambda(\w,x)^{\frac{1}{p}}+|(\xi+\nabla u_{\xi,\e})A(\w,x)|\Big)^{p-1}+\Lambda(\w,x)\dx
	\\
	&\leq C\left(\dashint_{Q/\e}|A(\w,x)|^p\dx\right)^{\frac{1}{p}}\left(\dashint_{Q/\e}\Lambda(\w,x)+|(\xi+\nabla u_{\xi,\e})A(\w,x)|^{p}\dx\right)^{\frac{p-1}{p}}+C\dashint_{Q/\e}\Lambda(\w,x),
\end{align*}
where we used H\"older's inequality to obtain the last line. Using Assumption \ref{a.1}, the almost minimality of $u_{\xi,\e}$, and the ergodic theorem \ref{thm.additiv_ergodic}, we deduce that
\begin{equation*}
	\limsup_{\e\to 0}|\xi_{\e}^*|\leq C \Big(\mathbb{E}[\Lambda(\cdot,0)]+f_{\rm hom}(\xi)\Big)^{\frac{p-1}{p}}\mathbb{E}[|A(\cdot,0)|^p]^{\frac{1}{p}}+C\,\mathbb{E}[\Lambda(\cdot,0)],
\end{equation*}
which implies the boundedness of $\xi_{\e}^*$ and thus the differentiability of $f_{\rm hom}$.
\end{proof} 

\section{Complete continuity of the embedding $W^{1,1}\hookrightarrow L^{d/(d-1)}$}\label{app:1}
We show that the non-compact Sobolev embedding $W^{1,1}\hookrightarrow L^{d/(d-1)}$ is completely continuous. 
\begin{theorem}\label{thm.embedding}
Let $d,m\in\N$ and $d\geq 2$. If $(u_n)_{n\in\N}\subset W^{1,1}(\R^d,\R^m)$ is a sequence such that $u_n\rightharpoonup u$ in $W^{1,1}(\R^d,\R^m)$, then $u_n\to u$ in $L^{d/(d-1)}(\R^d,\R^m)$. The same result holds true when $\R^d$ is replaced by an extension domain $O$, i.e., in case there exists a bounded, linear operator $E:W^{1,1}(O,\R^m)\to W^{1,1}(\R^d,\R^m)$ such that $E(u)=u$ a.e. on $O$.	
\end{theorem}
\begin{remark}
	Every bounded open set with Lipschitz boundary is an extension domain. For such sets it suffices to prove the equi-integrability of $|u_n|^{d/(d-1)}$. This has already been shown in \cite[Lemma A.3]{DLFS}.
\end{remark}
\begin{proof}[Proof of Theorem \ref{thm.embedding}]
Arguing separately for each component, it suffices to prove the scalar case $m=1$. Consider $u_n,u_0\in W^{1,1}(\R^d)$ such that $ u_n\rightharpoonup u_0$ in $W^{1,1}(\R^d)$. From the Gagliardo-Nirenberg-Sobolev inequality we infer hat $u_0\in L^{d/(d-1)}(\R^d)$. To show strong convergence in this space, we first show that the sequence $|u_n|^{d/(d-1)}$ does not lose mass at infinity. Since $\nabla u_{n}\rightharpoonup \nabla u_0$ in $L^1(\R^d,\R^d)$, it follows from the Dunford-Pettis theorem that for every $\e>0$ there exists a set $A_{\e}\subset\R^d$ with finite measure such that
\begin{equation*}
	\sup_{n\in\N}\int_{\R^d\setminus A_{\e}}|\nabla u_{n}|\dx<\e.
\end{equation*}
Although it might be known to experts, we argue that we can take the sets $A_{\e}$ to be compact. Indeed, due to the inner regularity of the Lebesgue measure there exists $K_{\e}\subset A_{\e}$ compact such that $|A_{\e}\setminus K_{\e}|\leq\rho(\e)$, where $\rho(\e)>0$ is chosen such that the equi-integrability of $|\nabla u_n|$ implies that
\begin{equation*}
	\sup_{n\in\N}\int_{A_{\e}\setminus K_{\e}}|\nabla u_n|\dx\leq\e.
\end{equation*}
We conclude that
\begin{equation*}
	\sup_{n\in\N}\int_{\R^d\setminus K_{\e}}|\nabla u_{n}|\dx<2\e.
\end{equation*}
Since $K_{\e}$ is compact, we find a radius $R=R_{\e}>0$ such that $K_{\e}\subset B_R$. For any $N\in\N$ and $1\leq i\leq N$ we define the radii $r_i=R+i$ and consider a cut-off function $\varphi\in C_c^{\infty}(\R^d,[0,1])$ such that $\varphi\equiv 1$ on $B_{r_{i-1}}$, ${\rm supp}(\varphi_i)\subset B_{r_{i}}$ and $\|\nabla\varphi_i\|_{\infty}\leq 2$. We then define the function 
\begin{equation*}
	u_{n,i}=(1-\varphi_i)u_n\in W^{1,1}(\R^d).
\end{equation*}
By the product rule it holds that $\nabla u_{n,i}=(1-\varphi_i)\nabla u_n-\nabla \varphi_iu_n$. The Gagliardo-Nirenberg-Sobolev inequality then yields that
\begin{align*}
	\left(\int_{\R^d\setminus B_{r_i}}|u_{n}|^{d/(d-1)}\dx\right)^{\frac{d-1}{d}}&\leq \left(\int_{\R^d}|u_{n,i}|^{d/(d-1)}\dx\right)^{\frac{d-1}{d}}\leq C\int_{\R^d}|\nabla u_{n,i}|\dx
	\\
	&\leq C\int_{\R^d\setminus B_{r_{i}}}|\nabla u_n|\dx+2C\int_{B_{r_{i}}\setminus B_{r_{i-1}}}|u_n|\dx
	\\
	&\leq C\int_{\R^d\setminus B_{R}}|\nabla u_n|\dx+2C\int_{B_{r_{i}}\setminus B_{r_{i-1}}}|u_n|\dx. 
\end{align*}
Since $B_R\supset K_{\e}$, the penultimate integral above can be bounded by $2C\e$ uniformly in $n$. Moreover, since the sets $(B_{r_{i+1}}\setminus B_{r_i})_{i=1}^N$ are pairwise disjoint, it follows that
\begin{equation*}
	\left(\int_{\R^d\setminus B_{r_N}}|u_{n}|^{d/(d-1)}\dx\right)^{\frac{d-1}{d}}\leq \frac{1}{N}\sum_{i=1}^N \left(\int_{\R^d\setminus B_{r_i}}|u_{n}|^{d/(d-1)}\dx\right)^{\frac{d-1}{d}}\leq 2C\e+\frac{2C}{N}\int_{\R^d}|u_n|\dx.
\end{equation*}
Since $u_n$ is a bounded sequence in $L^1(\R^d)$, it follows that for given $\e>0$ we can find a number $N=N_{\e}$ and a corresponding ball $B_{\e}=B_{R+N_{\e}}$ such that
\begin{equation}\label{tightness}
	\sup_{n\in\N}\int_{\R^d\setminus B_{\e}}|u_n|^{d/(d-1)}\dx\leq (4C\e)^{\frac{d}{d-1}}\leq \e,
\end{equation}
where we assumed that $\e\ll 1$ for the last estimate (recall that $d\geq 2$). Due to the compact embedding $W^{1,1}(B)\hookrightarrow\hookrightarrow L^1(B)$ for any ball $B\subset\R^d$, we deduce that $u_n\to u$ locally in measure on $\R^d$. Combined with \eqref{tightness} it follows that $u_n\to u$ in measure on $\R^d$. Given $k\in\N$, we consider the truncated sequence 
\begin{equation*}
	u_{n,k}=\min\{\max\{-k,u_n\},k\}.
\end{equation*} 
Then a.e. on $\R^d$ it holds that $\nabla u_{n,k}=\nabla u_n\chi_{\{|u_n|< k\}}$, so that $u_{n,k}\in W^{1,1}(\R^d)$. Moreover, pointwise it holds that $|u_{n,k}|\leq |u_n|$. Therefore also $|u_{n,k}|^{d/(d-1)}$ satisfies \eqref{tightness}. Since the truncation operator $x\mapsto \min\{\max\{-k,x\},k\}$ is $1$-Lipschitz, for any $\delta>0$ it holds that
\begin{equation*}
	\lim_{n\to +\infty}|\{|u_{n,k}-u_{0,k}|>\delta\}|\leq \lim_{n\to +\infty}|\{|u_n-u_0|>\delta\}|=0
\end{equation*}
and we conclude that $u_{n,k}$ converges to $u_{0,k}$ in measure. Moreover, $|u_{n,k}|^{d/(d-1)}$ is bounded in $L^{\infty}(\R^d)$. Hence it is equi-integrable and we conclude from Vitali's convergence theorem \cite[Theorem 2.24]{FoLe} that $u_{n,k}\to u_k$ in $L^{d/(d-1)}(\R^d)$. Next, note that
\begin{equation*}
	\int_{\R^d}|u_n-u_{n,k}|+|\nabla u_n-\nabla u_{n,k}|\dx\leq \int_{\{|u_n|> k\}}(|u_n|-k)+|\nabla u_n|\dx\leq \int_{\{|u_n|> k\}}|u_n|+|\nabla u_n|\dx.
\end{equation*}
Since $|\{|u_n|\geq k\}|\leq k^{-1}\|u_n\|_{L^1(\R^d)}$, the equi-integrability of $u_n$ and $\nabla u_n$ imply that
\begin{equation*}
	\lim_{k\to +\infty}\sup_{n\in\N}\int_{\R^d}|u_n-u_{n,k}|+|\nabla u_n-\nabla u_{n,k}|\dx=0.
\end{equation*}
The triangle inequality and the Gagliardo-Nirenberg-Sobolev inequality yield that
\begin{align*}
	\|u-u_n\|_{L^{d/(d-1)}(\R^d)}&\leq\|u-u_{k}\|_{L^{d/(d-1)}(\R^d)}+\|u_k-u_{n,k}\|_{L^{d/(d-1)}(\R^d)}+\|u_{n,k}-u_{n}\|_{L^{d/(d-1)}(\R^d)}
	\\
	&\leq \|u-u_{k}\|_{L^{d/(d-1)}(\R^d)}+\|u_k-u_{n,k}\|_{L^{d/(d-1)}(\R^d)}+C\sup_{n\in\N}\|u_{n,k}-u_{n}\|_{W^{1,1}(\R^d)}.
\end{align*}
Letting first $n\to +\infty$ and then $k\to +\infty$ we deduce that
\begin{equation*}
	\limsup_{n\to +\infty}\|u-u_n\|_{L^{d/(d-1)}(\R^d)}=0.
\end{equation*}
This proves the claim for $D=\R^d$. If $D\subset\R^d$ is an extension domain, then by definition there exists a bounded, linear extension operator $E:W^{1,1}(D)\to W^{1,1}(\R^d)$, which is also weakly continuous. Hence the claim follows from the continuity of the restriction map $L^{d/(d-1)}(\R^d)\to L^{d/(d-1)}(D)$.	
\end{proof}

\section{Measurability}\label{app:2}
Here we establish the following lemma:
\begin{lemma}\label{l.measurable}
Under Assumption \ref{a.1} the function $\w\mapsto \mu_{\xi}(\w,O)$ defined in Lemma \ref{l.existence_f_hom} is measurable for every open, bounded set $O\subset\R^d$. Moreover, if for almost every $\w\in\Omega$ there exists a minimizer for the problem defining $\mu_{\xi}(\w,O)$, then there exists a measurable function $u:\Omega\to \xi x+W^{1,1}_0(O,\R^m)$ such that $F_1(\w,u(\w),O)=\mu_{\xi}(\w,O)$ for almost every $\w\in\Omega$.
\end{lemma}
We first show a more general result for the measurability of infimum-values and minimizers that is well-known in some special cases.
\begin{lemma}\label{l.oninf}
Let $Y$ be a complete, separable metric space and $(T,\mathcal{A},m)$ be a complete measure space. Assume that $F:T\times Y\to \R\cup\{+\infty\}$ is $\mathcal{A}\otimes \mathcal{B}(Y)$-measurable and that $y\mapsto F(t,y)$ is lower semicontinuous and not constantly $+\infty$ for every $t\in T$. Then the function $t\mapsto\inf_{y\in Y} F(t,y)$ is $\mathcal{A}$-measurable. Moreover, if for every $t\in T$ there exists a minimizer of $y\mapsto F(t,y)$, then there exists an $\mathcal{A}\hbox{-}\mathcal{B}(Y)$-measurable function $y_{\rm min}:T\to Y$ such that $y_{\rm min}(t)\in{\rm argmin}\,F(t,\cdot)$.
\end{lemma}
\begin{proof}
It will be convenient to consider the epigraph of $F$ defined as the multifunction 
\begin{equation*}
	t\mapsto {\rm epi}\, F(t,\cdot):=\{(y,\alpha)\in Y\times\R:\,F(t,y)\leq\alpha\}.
\end{equation*}
The lower semicontinuity assumption on $F$ in the second variable shows that ${\rm epi}\, F(t,\cdot)$ is closed-valued. Moreover, it is non-empty by the finiteness assumption on $y\mapsto F(t,y)$. Since  $F$ is jointly measurable, the graph of ${\rm epi}\,F(t,\cdot)$ defined by
\begin{equation*}
	{\rm Gr}({\rm epi}\,F(t,\cdot))=\{(t,y,\alpha)\in T\times Y\times\R:\,F(t,y)\leq\alpha\}
\end{equation*}
belongs to $\mathcal{A}\otimes\mathcal{B}(Y)\otimes\mathcal{B}(\R)$. Due to the completeness of $\mathcal{A}$ with respect to $m$ and the properties of $Y$ we can apply \cite[Remark 6.11]{FoLe} and conclude that ${\rm epi}\,F(t,\cdot)$ is weakly measurable in the sense that
\begin{equation*}
	\{t\in T:\,{\rm epi}\,F(t,\cdot)\cap O\neq\emptyset\}
\end{equation*}
is $\mathcal{A}$-measurable for every open set $O\subset Y\times \R$. We now follow \cite[Theorem 14.37]{RoWe98} in order to prove the measurability of the infimum value. Denote by $\Pi:Y\times\R\to\R$ the projection map defined by $\Pi(y,\alpha)=\alpha$. We introduce a multifunction $\Gamma:T\to \mathcal{P}(\R)$ setting $\Gamma(t)=\Pi ({\rm epi}\,F(t,\cdot))$. Let $U\subset\R$ be open. Then
\begin{equation*}
\{t\in T:\,\Pi({\rm epi}\,F(t,\cdot))\cap U\neq\emptyset\}=\{t\in T:\,{\rm epi}\,F(t,\cdot)\cap (Y\times U)\neq\emptyset\},
\end{equation*}
which is $\mathcal{A}$-measurable since $Y\times U$ is open in $Y\times\R$. Hence also $\Gamma$ is weakly measurable. Since $\overline{\Gamma(t)}\cap U\neq\emptyset$ if and only if $\Gamma(t)\cap U\neq\emptyset$, it follows that also the closure of $\Gamma$ is weakly measurable. Finally, since $\overline{\Gamma(t)}\subset\R$, we know that it is even strongly measurable, that is, 
\begin{equation*}
	\{t\in T:\,\overline{\Gamma(t)}\cap C\neq\emptyset\}
\end{equation*}
is $\mathcal{A}$-measurable for every closed set $C\subset \R$ (cf. \cite[Remark 6.4]{FoLe}). An elementary argument shows that
\begin{equation*}
\{t\in T:\, \inf_{y\in Y}F(t,y)\leq\alpha\}=\{t\in T:\,\alpha\in \overline{\Gamma(t)}\}=\{t\in T:\, \overline{\Gamma(t)}\cap \{\alpha\}\neq\emptyset\}.
\end{equation*}
The set on the right-hand side is $\mathcal{A}$-measurable. Hence also $t\mapsto \inf_{y\in Y}F(t,y)$ is $\mathcal{A}$-measurable. To obtain the measurable selection of minimizers, define the multifunction $M:T\to\mathcal{P}(Y)$ by
\begin{equation*}
M(t)=\{y\in Y:\,F(t,y)=\inf_{y\in Y}F(t,y)\}.
\end{equation*}
By assumption, $M(t)\neq\emptyset$ and due to lower semicontinuity $M(t)$ is closed for all $t\in T$. Moreover, by the measurability of the infimum value the graph of $M$ is $\mathcal{A}\otimes\mathcal{B}(Y)$-measurable. Since $T$ is complete, Aumann's measurable selection Theorem (see \cite[Theorem 6.10]{FoLe}) implies the existence of a measurable selection of minimizers.
\end{proof}
With the above lemma at hand we can now prove the measurability of the process $\mu_{\xi}(\w,A)$.
\begin{proof}[Proof of Lemma \ref{l.measurable}]
Since we assume $\Omega$ to be a complete probability space, we can set $f(\w,x,\xi)=|\xi|^p$ on the null set where $|A(\w,\cdot)|^p+\Lambda(\w,\cdot)$ is not locally integrable and this modification does not affect measurability. We have to transfer the measurability properties of the integrand to the energy. To this end, we first regularize the integrand in $\xi$. Given $k\in\N$, define the Moreau-Yosida-regularization of $f$ by
\begin{equation*}
f_k(\w,x,\xi)=\inf_{\zeta\in\R^{m\times n}}\{f(\w,x,\zeta)+k|\zeta-\xi|\}.
\end{equation*}
It is well-known that $f_k$ is $k$-Lipschitz in the last variable. In order to apply Lemma \ref{l.oninf} we need to complete the product $\sigma$-algebra $\mathcal{F}\otimes\mathcal{L}^d$ with respect to the product measure $\mathbb{P}\times |\cdot|$. Denote the completed $\sigma$-algebra by $\overline{\mathcal{F}\otimes\mathcal{L}^d}$. Considering $\xi$ as a parameter, we deduce from Lemma \ref{l.oninf} and Assumption \ref{a.1} that the function $(\w,x)\to f_k(\w,x,\xi)$ is measurable with respect to $\overline{\mathcal{F}\otimes\mathcal{L}^d}$. By a well-known argument for Carathéodory-functions it follows that $f_k$ is $\overline{\mathcal{F}\otimes\mathcal{L}^d}\otimes\mathcal{B}(\R^{m\times d})$-measurable. Hence for any $u\in \xi x+W_0^{1,1}(O,\R^m)$ the function $(\w,x)\mapsto f_k(\w,x,\nabla u(x))$ is $\overline{\mathcal{F}\otimes\mathcal{L}^d}$-measurable and, by Tonelli's theorem in the form of \cite[Theorem 1.121]{FoLe}, we can define $F_k(\w,u):\Omega\times \left(\xi x+W_0^{1,1}(O,\R^m)\right)\to [0,+\infty)$ by
\begin{equation*}
F^k(\w,u)=\int_O f_k(\w,x,\nabla u(x))\,\dx.
\end{equation*}
The integral is indeed finite since the nonnegativity of $f$ and the Lipschitz continuity of $f_k$ imply that
\begin{equation}\label{eq:lineargrowth}
0\leq f_k(\w,x,\xi)\leq  f_k(\w,x,0)+k|\xi|\leq f(\w,x,0)+k|\xi|\leq \Lambda(\w,x)+k|\xi|.	
\end{equation}
(Recall that $\Lambda(\w,\cdot)$ is locally integrable on $\R^d$.)
Due to \eqref{eq:lineargrowth} and the Lipschitz-continuity of $f_k$ in the last variable, the functional $F^k(\w,\cdot)$ is continuous on $\xi x+W_0^{1,1}(O,\R^m)$. Moreover, again by Tonelli's theorem, for fixed $u\in\xi x+W^{1,1}_0(O,\R^m)$ the function $\w\mapsto F^k(\w,u)$ is measurable. In particular, the functional $F^k$ is $\mathcal{F}\otimes\mathcal{B}(\xi x+W^{1,1}_0(O,\R^m))$-measurable. Due to lower semicontinuity it holds that $f_k\uparrow f$ pointwise and therefore also $F^k\to F_1$ pointwise. It follows that $F_1$ is also $\mathcal{F}\otimes\mathcal{B}(\xi x+W^{1,1}_0(O,\R^m))$-measurable (and not constantly $+\infty$ for fixed $\w\in\Omega$ since $F_1(\w,\xi x,O)<+\infty$). Using that $\xi x+W_0^{1,1}(O,\R^m)$ is a separable, complete metric space, the measurability of $w\mapsto\mu_{\xi}(\w,O)$ follows once again from Lemma \ref{l.oninf} due to the completeness of $\Omega$. In the case when minimizers exist almost surely, we set $f(\w,x,\xi)=0$ for those $\w$, for which no minimizer exists. As just proved the (modified) function $F_1$ is $\mathcal{F}\otimes\mathcal{B}(\xi x+W^{1,1}_0(O,\R^m))$-measurable and the measurable selection of minimizers follows from Lemma \ref{l.oninf}.
\end{proof}

\end{document}